\documentclass[11pt]{article}
\usepackage[top=1in, bottom=1in, left=1in, right=1in]{geometry}
\usepackage{tikz}
\usetikzlibrary{shapes,arrows,matrix,patterns,positioning}
\usepackage{color}
\usepackage{graphicx}
\usepackage{algorithmic}
\usepackage{amssymb}
\usepackage{epstopdf}
\usepackage[ruled]{algorithm2e}
\usepackage{float}
\usepackage{amsmath,amsthm,bm,color,epsfig,enumerate,caption}
\usepackage{hyperref}
\usepackage{booktabs}
\usepackage{multirow}
\usepackage{array}

\newtheorem{theorem}{Theorem}[section]

\newtheorem{lemma}[theorem]{Lemma}

\renewcommand{\appendix}[1]{
\section*{Appendix: #1}
}

\renewcommand{\O}{O}

\newcommand{\bbC}{\mathbb{C}}
\newcommand{\bbR}{\mathbb{R}}
\newcommand{\rID}{{\it rID}}
\newcommand{\cID}{{\it cID}}

\makeatletter
\newcommand*{\extendadd}{
  \mathbin{
    \mathpalette\extend@add{}
  }
}
\newcommand*{\extend@add}[2]{
  \ooalign{
    $\m@th#1\leftrightarrow$%
    \vphantom{$\m@th#1\updownarrow$}
    \cr
    \hfil$\m@th#1\updownarrow$\hfil
  }
}
\makeatother

\begin{document}

\title{Interpolative Decomposition Butterfly Factorization}
\author{Qiyuan Pang \\ Tsinghua University, China\\  \href{mailto:ppangqqyz@foxmail.com}{ppangqqyz@foxmail.com} 
   \and Kenneth L. Ho \\ San Francisco, CA, USA\\ \href{mailto:klho@alumni.caltech.edu}{klho@alumni.caltech.edu}
     \and Haizhao Yang \\ Department of Mathematics\\ National University of Singapore, Singapore\\ \href{mailto:haizhao@nus.edu.sg}{haizhao@nus.edu.sg} }

\maketitle

\begin{abstract}
This paper introduces a ``kernel-independent'' interpolative decomposition butterfly factorization (IDBF) as a data-sparse
approximation for matrices that satisfy a complementary low-rank
property. The IDBF can be constructed in $O(N\log N)$ operations for an $N\times N$ matrix via hierarchical interpolative decompositions (IDs), if matrix entries can be sampled individually and each sample takes $O(1)$ operations. The resulting factorization is a product of
$O(\log N)$ sparse matrices, each with $O(N)$ non-zero entries. Hence,
it can be applied to a vector rapidly in $O(N\log N)$ operations. IDBF is a general framework for nearly optimal fast matvec useful in a wide range of applications, e.g., special function transformation, Fourier integral operators, high-frequency wave computation. Numerical
results are provided to demonstrate the effectiveness of the butterfly
factorization and its construction algorithms.
\end{abstract}

{\bf Keywords.} Data-sparse matrix, butterfly factorization, interpolative decomposition, operator compression, Fourier integral operators, special functions, high-frequency integral equations.

{\bf AMS subject classifications: 44A55, 65R10 and 65T50.}

\section{Introduction}
\label{sec:intro}

One of the key computational task in scientific computing is to
evaluate dense matrix-vector multiplication (matvec) rapidly. Given a dense matrix $K
\in \bbC^{N\times N}$ and a vector $x\in\bbC^N$, it takes $O(N^2)$ operations
to naively compute the vector $y=Kx\in \bbC^N$. There has been extensive research in constructing data-sparse representation of structured matrices (e.g., low-rank matrices \cite{Rec1,Rec3,ID2,mahoney2009cur}, $\mathcal{H}$ matrices \cite{HMat,HMatrix,Grasedyck2003}, $\mathcal{H}^2$ matrics \cite{HSS2,H2}, HSS matrices \cite{HSS1,HSSMatrix}, complementary low-rank matrices \cite{Butterfly1,Butterfly2,BF,IBF,MBF,MVBF}, FMM \cite{ROKHLINFMM1,ROKHLINFMM2,CHENGFMM,FMM3,FMM4,FMM5,FMM6,doi:10.1137/16M1095949}, directional low-rank matrices \cite{YingFFT,YingDirect,YingFMM,MESSNER20121175}, and the combination of these matrices \cite{HSSBF,LUBF}) aiming for linear or nearly linear scaling matvec. In particular, this paper concerns nearly optimal matvec for complementary low-rank matrices.

A wide range of transforms in harmonic analysis \cite{Butterfly2, BF,wavemoth,SHT,Jacobi,FIO09}, and integral equations in the high-frequency regime \cite{HSSBF,LUBF} admit a matrix or its submatrices 
satisfying a complementary low-rank property. For a {1D} complementary low-rank matrix, its rows are typically indexed by a point set {$X \subset \bbR$ and its columns by another point set $\Omega \subset \bbR$}. Associated with $X$ and $\Omega$ are two trees $T_{X}$ and $T_{\Omega}$ constructed by dyadic partitioning of each domain. Both trees have the same level $L +1= O(\log N)$, with the top root being the $1$-th level and the bottom leaf level being the $(L+1)$-th level. We say a matrix satisfies the complementary low-rank property if, {for} any node $A$ at level {$\ell$} in $T_{X}$ and any node $B$ at level {$L+2-\ell$}, the submatrix {$K_{A,B}^{\ell}$} of $K$, obtained by restricting the rows of $K$ to the points in node $A$ and the columns to the points in node $B$, {is} numerically low-rank; that is, given a precision $\epsilon$, there exists an approximation of {$K_{A,B}^{\ell}$} with the 2-norm of the error bounded by $\epsilon$ and the rank $k$
bounded by a polynomial in $\log N$ and $\log 1/\epsilon$.

Points in $X\times\Omega$ may be non-uniformly distributed. Hence, submatrices {$\{K_{A,B}^{\ell}\}_{A,B}$} at the same level $\ell$ may have different sizes but they have almost the same rank. If the point distribution is uniform, then at the $\ell$-th level starting from the root of $T_X$, submatrices have the same size $\frac{N}{2^{\ell-1}}\times 2^{\ell-1}$. See Figure {\ref{fig:complement}} for an illustration of low-rank submatrices in a {1D} complementary low-rank matrix of size $16 \times 16$ with uniform point distributions in $X$ and $\Omega$.  It is easy to generalize the complementary low-rank matrices to higher dimensional space as in \cite{MBF}. For simplicity, we only present the IDBF for the {1D} case with uniform point distributions and leave the extension for non-uniform point distributions and higher dimensional cases to the reader.

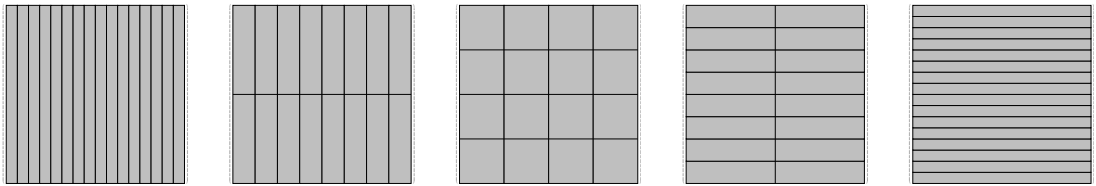
\begin{figure}[htp]
\begin{minipage}{\textwidth}
\centering
\resizebox{2.5cm}{!}{
\begin{tikzpicture}[baseline=-0.5ex]
      \tikzset{every left delimiter/.style={xshift=-1ex},every right delimiter/.style={xshift=1ex}}
      \matrix (mat) [matrix of math nodes, left delimiter=(, right delimiter=)] {
      \draw;
\foreach \j in {0,2,...,30}{
    \draw[fill=lightgray] (\j,32) rectangle (\j+2,0);
}
\\
      };
\end{tikzpicture}
}
\quad
\resizebox{2.5cm}{!}{
\begin{tikzpicture}[baseline=-0.5ex]
      \tikzset{every left delimiter/.style={xshift=-1ex},every right delimiter/.style={xshift=1ex}}
      \matrix (mat) [matrix of math nodes, left delimiter=(, right delimiter=)] {
      \draw;
\foreach \j in {0,4,...,28}{
    \draw[fill=lightgray] (\j,32) rectangle (\j+4,0);
}
\draw[fill=gray] (0,16) rectangle (32,16);
\\
      };
\end{tikzpicture}
}
\quad
\resizebox{2.5cm}{!}{
\begin{tikzpicture}[baseline=-0.5ex]
      \tikzset{every left delimiter/.style={xshift=-1ex},every right delimiter/.style={xshift=1ex}}
      \matrix (mat) [matrix of math nodes, left delimiter=(, right delimiter=)] {
      \draw;
\foreach \j in {0,8,...,24}{
    \draw[fill=lightgray] (\j,32) rectangle (\j+8,0);
}
\draw[fill=gray] (0,8) rectangle (32,8);
\draw[fill=gray] (0,16) rectangle (32,16);
\draw[fill=gray] (0,24) rectangle (32,24);
\\
      };
\end{tikzpicture}
}
\quad
\resizebox{2.5cm}{!}{
\begin{tikzpicture}[baseline=-0.5ex]
      \tikzset{every left delimiter/.style={xshift=-1ex},every right delimiter/.style={xshift=1ex}}
      \matrix (mat) [matrix of math nodes, left delimiter=(, right delimiter=)] {
      \draw;
\foreach \j in {0,4,...,28}{
    \draw[fill=lightgray] (32,\j) rectangle (0,\j+4);
}
\draw[fill=gray] (16,0) rectangle (16,32);
\\
      };
\end{tikzpicture}
}
\quad
\resizebox{2.5cm}{!}{
\begin{tikzpicture}[baseline=-0.5ex]
      \tikzset{every left delimiter/.style={xshift=-1ex},every right delimiter/.style={xshift=1ex}}
      \matrix (mat) [matrix of math nodes, left delimiter=(, right delimiter=)] {
      \draw;
\foreach \j in {0,2,...,30}{
    \draw[fill=lightgray] (32,\j) rectangle (0,\j+2);
}
\\
      };
\end{tikzpicture}
}
\\
\end{minipage}
\caption{Hierarchical decomposition of the row and column indices
    of a $16\times 16$ matrix. The dyadic trees $T_{X}$ and $T_{\Omega}$ have roots containing $16$ rows and $16$ columns respectively, and their leaves containing only a single row or column. The partition above indicates the complementary low-rank property of the matrix, and assumes that each submatrix is rank-$1$. }
\label{fig:complement}
\end{figure}

This paper introduces an {\bf Interpolative Decomposition Butterfly Factorization (IDBF)} as a data-sparse
approximation for matrices that satisfy the complementary low-rank
property. The IDBF can be constructed in $O(\frac{k^3}{n_0} N\log N)$ operations for an $N\times N$ matrix $K$ with a local rank parameter $k$ and a leaf size parameter $n_0$ via hierarchical linear interpolative decompositions (IDs), if matrix entries can be sampled individually and each sample takes $O(1)$ operations. The resulting factorization is a product of $O(\log N)$ sparse matrices, each of which contains $O(\frac{k^2}{n_0}N)$ nonzero entries as follows:
\begin{equation}
K \approx U^{L}U^{L-1}\cdots U^{h} S^h V^{h}\cdots V^{L-1}V^{L},
\end{equation}
where $h=L/2$ and the level $L$ is assumed to be even. Hence,
it can be applied to a vector rapidly in $O(\frac{k^2}{n_0}N\log N)$ operations. Previously, purely algebraic butterfly factorizations (in the sense that the complementary matrix is not the discretization of a kernel function $K(x,\xi)=a(x,\xi)e^{2\pi i \Phi(x,\xi)}$ with smooth $a(x,\xi)$ and $\Phi(x,\xi)$) have at least $O(N^{1.5})$ scaling \cite{Butterfly1,Butterfly2,BF,MBF}. The IDBF is the first {\bf purely algebraic butterfly factorization} (BF) with $O(N\log N)$ scaling in both factorization and application.

\section{Interpolative Decomposition Butterfly Factorization (IDBF)}

We will describe IDBF in detail in this section. For the sake of simplicity, we assume that $N=2^Ln_0$, where $L$ is an even integer, and $n_0=O(1)$ is the number of column or row indices in a leaf in the dyadic trees of row and column spaces, i.e., $T_{X}$ and $T_{\Omega}$, respectively. Let's briefly introduce the main ideas of designing $O(\frac{k^3}{n_0}N \log N)$ IDBF using a linear ID. In IDBF, we compute $O(\log N)$ levels of low-rank submatrix factorizations. At each level, according to the matrix partition by the dyadic trees in column and row (see Figure \ref{fig:complement} for an example), there are $\frac{N}{n_0}$ low-rank submatrices. Linear IDs only require $O(k^3)$ operations for each submatrix, and hence at most $O(\frac{k^3}{n_0} N)$ for each level of factorization, and $O(\frac{k^3}{n_0} N\log N)$ for the whole IDBF. There are two differences between IDBF and other BFs \cite{Butterfly1,Butterfly2,BF}.
\begin{enumerate}
\item The order of {factorization} is from the leaf-root and root-leaf levels of matrix partitioning (e.g., the left and right panels in Figure \ref{fig:complement}) and moves towards the middle level of matrix partitioning (e.g., the middle panel of Figure \ref{fig:complement}).
\item Linear IDs are organized in an appropriate way such that it is {cheap} in terms of both memory and operations to provide all necessary information for each level of {factorization}.
\end{enumerate}



{In what follows, uppercase} letters will generally denote matrices, while the lowercase letters $c$,
$p$, $q$, $r$, and $s$ denote ordered sets of indices. For a given index set $c$, its cardinality is written $|c|$. Given a matrix $A$, $A_{pq}$, $A_{p,q}$, or $A(p,q)$ is the submatrix with rows and columns restricted to
the index sets $p$ and $q$, respectively. We also use the notation $A_{:,q}$ to denote the submatrix with columns restricted to $q$. $s:t$ is an index set containing indices $\{s,s+1,s+2,\dots,t-1,t\}$.

\subsection{Linear scaling Interpolative Decompositions}
Interpolative decomposition and other low-rank decomposition techniques \cite{Rec1,ID2,Rec2} are important elements in modern scientific computing. These techniques usually require $O(kmn)$ arithmetic operations to get a rank $k=O(1)$ matrix factorization to approximate a matrix $A\in\bbC^{m\times n}$. Linear scaling randomized techniques can reduce the cost to {$O(k(m+n))$} \cite{RandomSample}. \cite{subCUR} further shows that in the CUR low-rank approximation $A\approx CUR$, where $C=A_{:,c}$, $R=A_{r,:}$, and $U\in\mathbb{C}^{k\times k}$ with $|c|=|r|=k$, if only $U$, $c$, and $r$ are needed, there exists an $O(k^3)$ algorithm for constructing $U$, $c$, and $r$.

In the construction of IDBF, we use an $O(nk^2)$ linear scaling column ID to construct $V$ and select skeleton indices $q$ such that $A\approx A_{:,q}V$ when $n\ll m$. Similarly, we can construct a row ID $A\approx {U}A_{q,:}$ in $O(mk^2)$ operations when $m\ll n$. As in \cite{RandomSample,subCUR}, randomized sampling can be applied to reduce the quadratic computational cost to linear. Here we present a simple lemma of interpolative decomposition (ID) to motivate the proposed linear scaling ID.

\begin{lemma}
\label{lem:1}
For a matrix $A \in \bbC^{m \times n}$ with rank ${k} \leq \min\{m,n\}$, there exists a partition of the column indices of $A$, $p \cup q$ with $|q| = {k}$, and a matrix $T \in \bbC^{{k \times (n-k)}}$, such that $A_{:,p} = A_{:,q}T$.
\end{lemma}

\begin{proof}
A rank revealing QR decomposition of A gives
\begin{equation}
\label{eq:pivotedQR}
A\Lambda = QR = Q[R_{1} \ R_{2}],
\end{equation}
where $Q\in \bbC^{m\times {k}}$ is an orthogonal matrix, $R\in \bbC^{{k}\times n}$ is upper triangular, and $\Lambda \in \bbC^{n\times n}$ is a carefully chosen permutation matrix such that $R_{1}\in \bbC^{{k\times k}}$ is nonsingular. Let \begin{equation}
A_{:,q} = Q R_{1},
\end{equation}
and then 
\begin{equation}
A_{:,p} = Q R_{2} = Q R_{1}R_{1}^{-1}R_{2} = A_{:,q}T,
\end{equation}
where
\begin{equation}
\label{eqn:R1}
T = R_{1}^{-1}R_{2}.
\end{equation}
\end{proof}

$A_{:,p} = A_{:,q}T$ in Lemma \ref{lem:1} is equivalent to the traditional form of a column ID,
\begin{equation}
A = A_{:,q}[I \ T]\Lambda^*:=A_{:,q}V,
\end{equation}
where $^*$ denotes the conjugate transpose of a matrix. We call $p$ and $q$ as {\textit{redundant} and \textit{skeleton}} indices, respectively. $V$ can be understood as a {\it column interpolation matrix}. Our goal for linear scaling ID is to construct the skeleton index set $q$, the redundant index set $p$, $T$, and $\Lambda$ in $O({k}^2n)$ operations and $O({k}n)$ memory.

For a tall skinny matrix $A$, i.e., $m \gg n$, the rank revealing QR decomposition of A in \eqref{eq:pivotedQR} typically requires {$O(kmn)$} operations. To reduce the complexity to $O({k}^2n)$, we actually apply the rank revealing QR decomposition to $A_{s,:}$:
\begin{equation}
\label{eq:pivotedQR2}
A _{s,:}\Lambda = QR = Q[R_{1} \ R_{2}],
\end{equation}
 where $s$ is an index set {containing} $tk$ carefully selected rows of $A$, where $t$ is an oversampling parameter. These rows can be chosen independently and uniformly from the row space as {in} the sublinear CUR in \cite{subCUR} or the linear scaling algorithm in \cite{RandomSample};  or they can be chosen from the Mock-Chebyshev
grids of the row indices as in \cite{MVBF,Mock1,Mock2}. In fact, numerical results show that Mock-Chebyshev points lead to a more efficient and accurate ID than randomly sampled points when matrices are from physical systems. After the rank revealing QR decomposition, the other steps to generate $T$ and $\Lambda$ take only $O({k}^2n)$ operations since $R_1$ in \eqref{eqn:R1} is an upper triangular matrix.

{In practice, the true rank of $A$ is not available i.e., $k$ is unknown. In this case, the above computation procedure should be applied with some test rank $k\leq n$. Furthermore, we are often} interested in an ID with a numerical rank ${k_\epsilon}$ specified by an accuracy parameter $\epsilon$, i.e.
\begin{equation}
\label{eq:req}
\|A-A_{:,q}V\|_2\leq O(\epsilon)
\end{equation}
with $T\in \mathbb{C}^{{k_\epsilon\times (n-k_\epsilon)}}$ and $V\in \mathbb{C}^{{k_\epsilon}\times n}$. We can choose
\begin{equation}
\label{eqn:acc}
{k_\epsilon}=\min \{{k}: R_1({k,k})\leq \epsilon R_1(1,1)\},
\end{equation}
where $R_1$ is given by the rank-revealing QR factorization in \eqref{eq:pivotedQR2}. Then define
\begin{equation}
\label{eqn:R2}
T = (R_{1}(1:{k_\epsilon},1:{k_\epsilon}))^{-1}[R_1(1:{k_\epsilon},{k_\epsilon+1:k}) \ R_{2}(1:{k_\epsilon},:)]\in \mathbb{C}^{{k_\epsilon\times (n-k_\epsilon)}},
\end{equation}
and
\[
V = [I\ T]\Lambda^* \in  \mathbb{C}^{{k_\epsilon}\times n}.
\]
Correspondingly, let $q$ be the index set such that
\[
A_{:,q}=QR_1(1:{k_\epsilon},1:{k_\epsilon}),
\]
and $p$ be the complementary set of $q$, then $q$ and $V$ satisfy the requirement in \eqref{eq:req}. We refer {to} this linear scaling column ID with an accuracy tolerance $\epsilon$ and a rank parameter $k$ as {\it $(\epsilon,k)$-cID}.  For convenience, we will drop the term $(\epsilon,k)$ when it is not necessary to specify it.

 For a short and fat matrix $A\in \bbC^{m \times n}$ with  $m \ll n$, a similar row ID 
\begin{equation}
A \approx \Lambda [I \ T]^* A_{q,:}:=UA_{q,:}
\end{equation}
can be devised similarly with $O(k^2m)$ operations and $O(km)$ memory. We refer {to} this linear scaling row ID as {\it ${\epsilon}$-rID} and $U$ as the {\it row interpolation matrix}.

\subsection{Leaf-root complementary skeletonization (LRCS)}
\label{sec:lrskeleton}

For a complementary low-rank matrix $A$, we introduce the {\it leaf-root complementary skeletonization (LRCS)}
\[
A\approx USV
\]
via $\cID$s of the submatrices corresponding to the leaf-root levels of the column-row dyadic trees (e.g., see the associated matrix partition in Figure \ref{fig:partition-K} (right)), and $\rID$s of the submatrices corresponding to the root-leaf levels of the column-row dyadic trees (e.g., see the associated matrix partition in Figure \ref{fig:partition-K} (middle)). We always assume that IDs in this section {are} applied with a rank parameter $k=O(1)$. We'll not specify $k$ again in the following discussion.

Suppose that at the leaf level of the row (and column) dyadic trees, the row index set $r$ (and the column index set $c$) of $A$ are partitioned into leaves $\{r_i\}_{1\leq i\leq m}$ (and $\{c_i\}_{1\leq i\leq m}$) as follows 
\begin{equation}
\label{eq:partrc}
r = [r_{1},r_{2},\cdots,r_{m}] \qquad (\text{and } c = [c_{1},c_{2},\cdots,c_{m}]),
\end{equation}
with $|r_i|=n_0$ (and $|c_i|=n_0$) for all $1\leq i\leq m$, 
where $m = 2^{L}=\frac{N}{n_0}$, $L = \log_2 N - \log_2 n_0$, and $L+1$ is the depth of the dyadic trees $T_{X}$ (and $T_{\Omega}$). Figure \ref{fig:partition-K} shows an example of row and column dyadic trees with $m=16$. We apply $\rID$ to each $A_{r_{i},:}$ to obtain the row interpolation matrix in its ID and denote it as $U_{i}$;  the associated skeleton indices of the ID is denoted as $\hat{r}_{i}\subset r_{i}$. Let
\begin{equation}
\label{eq:partrc2}
\hat{r} = [\hat{r}_{1},\hat{r}_{2},\cdots,\hat{r}_{m}],
\end{equation}
then $A_{\hat{r},:}$ is the important skeleton of $A$ and we have
\[
A\approx \begin{pmatrix}
U_{1} & & & \\
& U_{2} & & \\
& & \ddots & \\
& & & U_{m} 
\end{pmatrix}
\begin{pmatrix}
A_{\hat{r}_{1},c_{1}} & A_{\hat{r}_{1},c_{2}} & \hdots & A_{\hat{r}_{1},c_{m}} \\
A_{\hat{r}_{2},c_{1}} & A_{\hat{r}_{2},c_{2}} & \hdots & A_{\hat{r}_{2},c_{m}} \\
\vdots & \vdots & \ddots & \vdots \\
A_{\hat{r}_{m},c_{1}} & A_{\hat{r}_{m},c_{2}} & \hdots & A_{\hat{r}_{m},c_{m}}  
\end{pmatrix}
:=UM.
\]

Similarly, $\cID$ is applied to each $A_{\hat{r},c_{j}}$ to obtain the column interpolation matrix $V_{j}$ and the skeleton indices $\hat{c}_{j}\subset c_{j}$ in its ID. Then finally we form the LRCS of $A$ as

\begin{equation}
\label{facA}
A\approx \begin{pmatrix}
U_{1} & & & \\
& U_{2} & & \\
& & \ddots & \\
& & & U_{m} 
\end{pmatrix}
\begin{pmatrix}
A_{\hat{r}_{1},\hat{c}_{1}} & A_{\hat{r}_{1},\hat{c}_{2}} & \hdots & A_{\hat{r}_{1},\hat{c}_{m}} \\
A_{\hat{r}_{2},\hat{c}_{1}} & A_{\hat{r}_{2},\hat{c}_{2}} & \hdots & A_{\hat{r}_{2},\hat{c}_{m}} \\
\vdots & \vdots & \ddots & \vdots \\
A_{\hat{r}_{m},\hat{c}_{1}} & A_{\hat{r}_{m},\hat{c}_{2}} & \hdots & A_{\hat{r}_{m},\hat{c}_{m}}  
\end{pmatrix}
\begin{pmatrix}
V_{1} & & & \\
& V_{2} & & \\
& & \ddots & \\
& & & V_{m} 
\end{pmatrix}:=USV.
\end{equation}
For a concrete example, Figure \ref{fig:rlfac} visualizes the non-zero pattern of the LRCS  in \eqref{facA} of the complementary low-rank matrix $A$ in Figure \ref{fig:partition-K}.

The novelty of the LRCS is that $M$ and $S$ are not computed explicitly; instead, they are generated and stored via the skeleton of row and column index sets. Hence, it only takes $O(\frac{k^3}{n_0} N)$ operations and $O(\frac{k^2}{n_0}N)$ memory to generate and store the factorization in \eqref{facA}, since there are $2m=\frac{2N}{n_0}$ IDs in total.

It is worth emphasizing that in the LRCS of a complementary matrix $A\approx USV$, the matrix $S$ is again a complementary matrix. The row (and column) dyadic tree $\hat{T}_X$ (and $\hat{T}_\Omega$) of $S$ is the compressed version of the row (and column) dyadic trees ${T}_X$ (and ${T}_\Omega$) of $A$. Figure \ref{fig:row1} (or \ref{fig:col1}) visualizes the relation of ${T}_X$ and $\hat{T}_X$ (or ${T}_\Omega$ and $\hat{T}_\Omega$) for the complementary matrix $A$ in Figure \ref{fig:partition-K}. $\hat{T}_X$ (or $\hat{T}_\Omega$) is not compressible at the leaf level of $T_X$ (or $T_\Omega$) but it is compressible if it is considered as a dyadic tree with one depth less (see Figure \ref{fig:rc} for an example of a new compressible dyadic tree with one depth less).

\begin{figure}[htp]
\begin{minipage}{\textwidth}
\centering
\resizebox{3.5cm}{!}{
\begin{tikzpicture}[baseline=-0.5ex]
      \tikzset{every left delimiter/.style={xshift=-1ex},every right delimiter/.style={xshift=1ex}}
      \matrix (mat) [matrix of math nodes, left delimiter=(, right delimiter=)] {
\draw[fill=gray] (32,0) rectangle (0,32);
\\
      };
\end{tikzpicture}
}
=
\resizebox{3.5cm}{!}{
\begin{tikzpicture}[baseline=-0.5ex]
      \tikzset{every left delimiter/.style={xshift=-1ex},every right delimiter/.style={xshift=1ex}}
      \matrix (mat) [matrix of math nodes, left delimiter=(, right delimiter=)] {
\draw;
\foreach \i in {0,2,...,30}{
   \draw[fill=gray] (0,32-\i) rectangle (32,32-\i-2);
}\\
      };
\end{tikzpicture}
}
=
\resizebox{3.5cm}{!}{
\begin{tikzpicture}[baseline=-0.5ex]
      \tikzset{every left delimiter/.style={xshift=-1ex},every right delimiter/.style={xshift=1ex}}
      \matrix (mat) [matrix of math nodes, left delimiter=(, right delimiter=)] {
\draw;
\foreach \i in {0,2,...,30}{
   \draw[fill=gray] (32-\i,0) rectangle (32-\i-2,32);
}\\
      };
\end{tikzpicture}
}\\
\end{minipage}
\caption{The left matrix is a complementary low-rank matrix. Assume that the depth of the dyadic trees of column and row spaces is $5$. The middle figure visualizes the root-leaf partitioning that divides the row index set into $16$ continuous subsets as $16$ leaves. The right one is for the leaf-root partitioning that divides the column index set into $16$ continuous subsets as $16$ leaves.}
\label{fig:partition-K}
\end{figure}
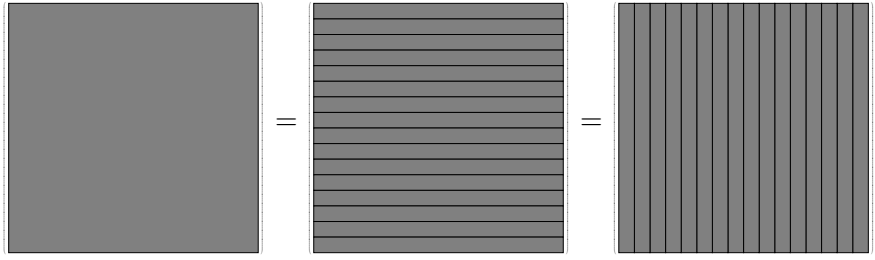

\begin{figure}[htp]
\begin{minipage}{\textwidth}
\centering
\resizebox{3.5cm}{!}{
\begin{tikzpicture}[baseline=-0.5ex]
      \tikzset{every left delimiter/.style={xshift=-1ex},every right delimiter/.style={xshift=1ex}}
      \matrix (mat) [matrix of math nodes, left delimiter=(, right delimiter=)] {
\draw[fill=gray] (32,0) rectangle (0,32);
\\
      };
\end{tikzpicture}
}
$\approx$
\resizebox{1.85cm}{!}{
\begin{tikzpicture}[baseline=-0.5ex]
      \tikzset{every left delimiter/.style={xshift=-1ex},every right delimiter/.style={xshift=1ex}}
      \matrix (mat) [matrix of math nodes, left delimiter=(, right delimiter=)] {
      \draw;
\foreach \i in {0,1,...,15}{
   \draw[fill=gray]  (\i,32-\i-\i) rectangle (\i+1,32-\i-\i-2);
}\\
      };
\end{tikzpicture}
}
\resizebox{1.85cm}{!}{
\begin{tikzpicture}[baseline=-0.5ex]
      \tikzset{every left delimiter/.style={xshift=-1ex},every right delimiter/.style={xshift=1ex}}
      \matrix (mat) [matrix of math nodes, left delimiter=(, right delimiter=)] {
      \draw;
\foreach \i in {0,1,...,15}{
\foreach \j in {0,1,...,15}{
      \draw[fill=gray]  (\i,16-\j) rectangle (\i+1,16-\j-1);
}
}\\
      };
\end{tikzpicture}
}
\resizebox{3.5cm}{!}{
\begin{tikzpicture}[baseline=-0.5ex]
      \tikzset{every left delimiter/.style={xshift=-1ex},every right delimiter/.style={xshift=1ex}}
      \matrix (mat) [matrix of math nodes, left delimiter=(, right delimiter=)] {
      \draw;
\foreach \i in {0,1,...,15}{
   \draw[fill=gray]  (32-\i-\i,\i) rectangle (32-\i-\i-2,\i+1);
}\\
      };
\end{tikzpicture}
}\\
\end{minipage}
\caption{An example of the LRCS in \eqref{facA} of the complementary low-rank matrix $A$ in Figure \ref{fig:partition-K}. Non-zero submatrices in \eqref{facA} are shown in gray areas.}
\label{fig:rlfac}
\end{figure}
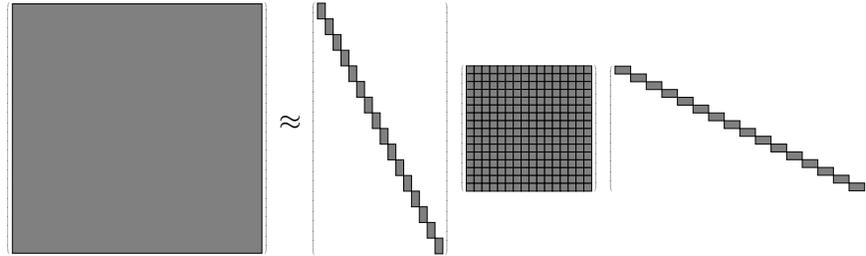

\begin{figure}[ht!]
  \begin{center}
    \begin{tabular}{cc}
      \includegraphics[height=1.7in]{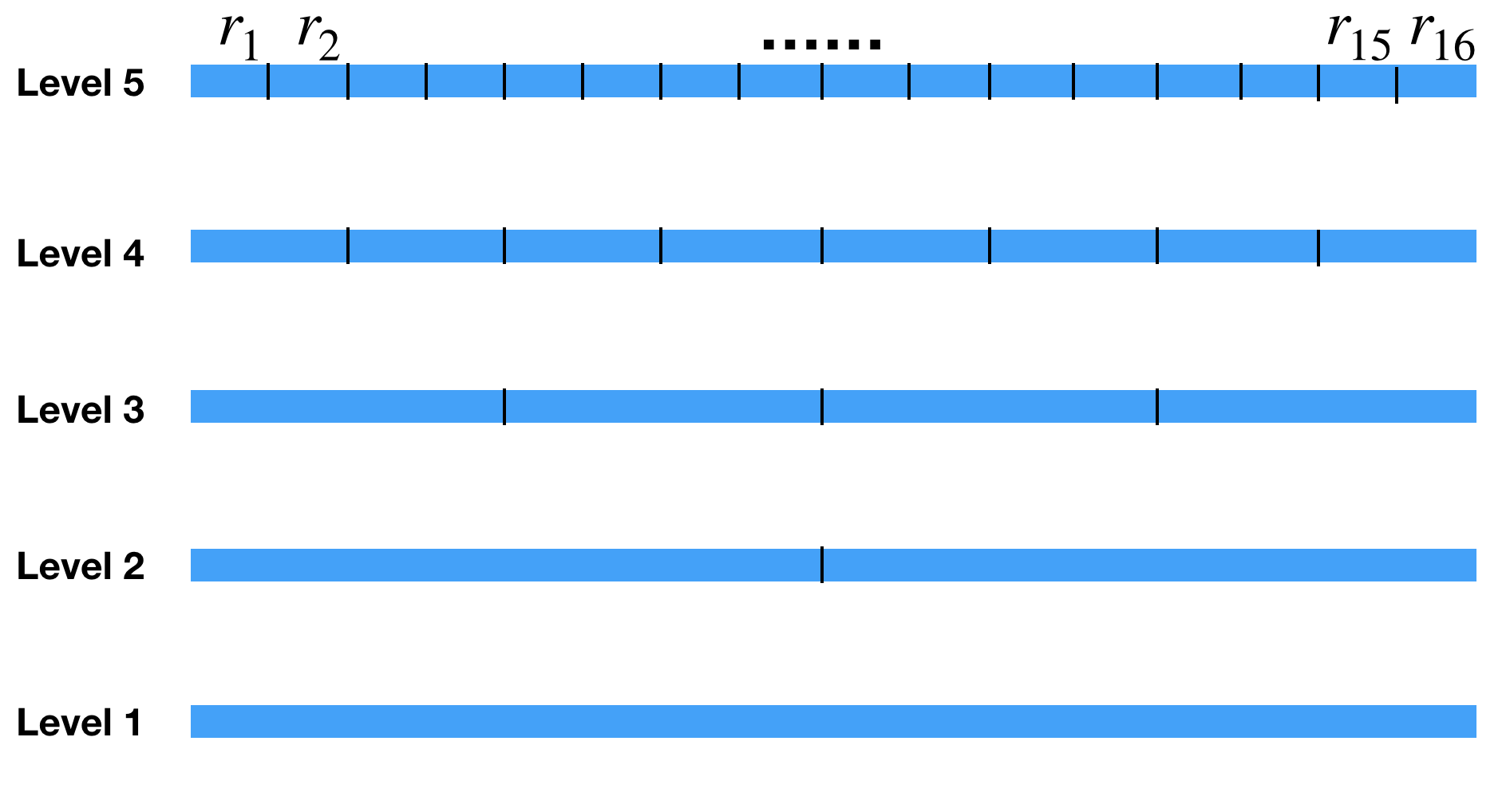}&
      \includegraphics[height=1.7in]{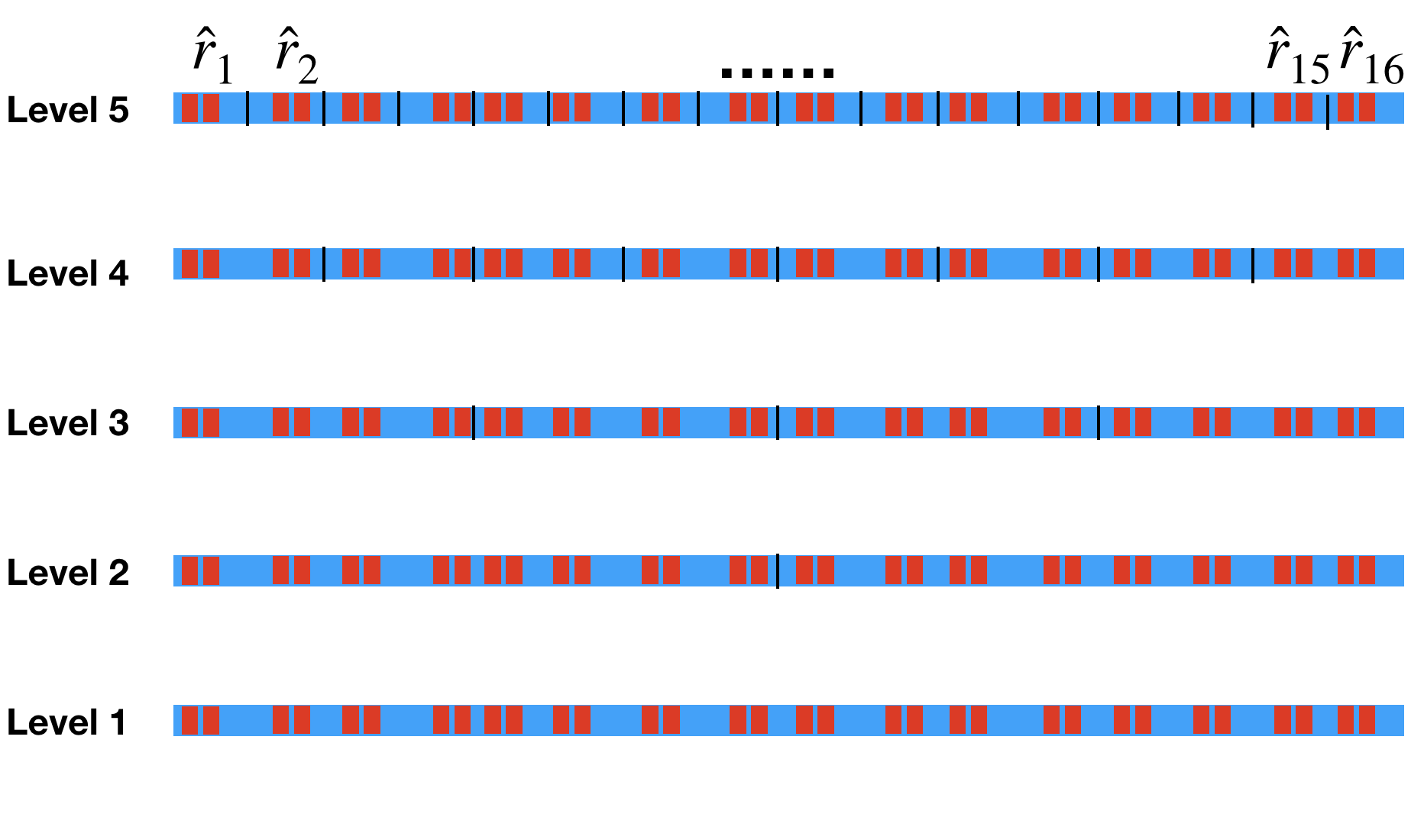}
    \end{tabular}
  \end{center}
\caption{Left: The dyadic tree $T_X$ of the row space with leaves $\{r_i\}_{1\leq i\leq 16}$ denoted as in \eqref{eq:partrc} for the example in Figure \ref{fig:partition-K}. Right: Selected important rows of $T_X$ naturally form a compressed dyadic tree in red with leaves $\{\hat{r}_i\}_{1\leq i\leq 16}$ denoted as in \eqref{eq:partrc2}. }
\label{fig:row1}

\end{figure}\begin{figure}[ht!]
  \begin{center}
    \begin{tabular}{cc}
      \includegraphics[height=1.7in]{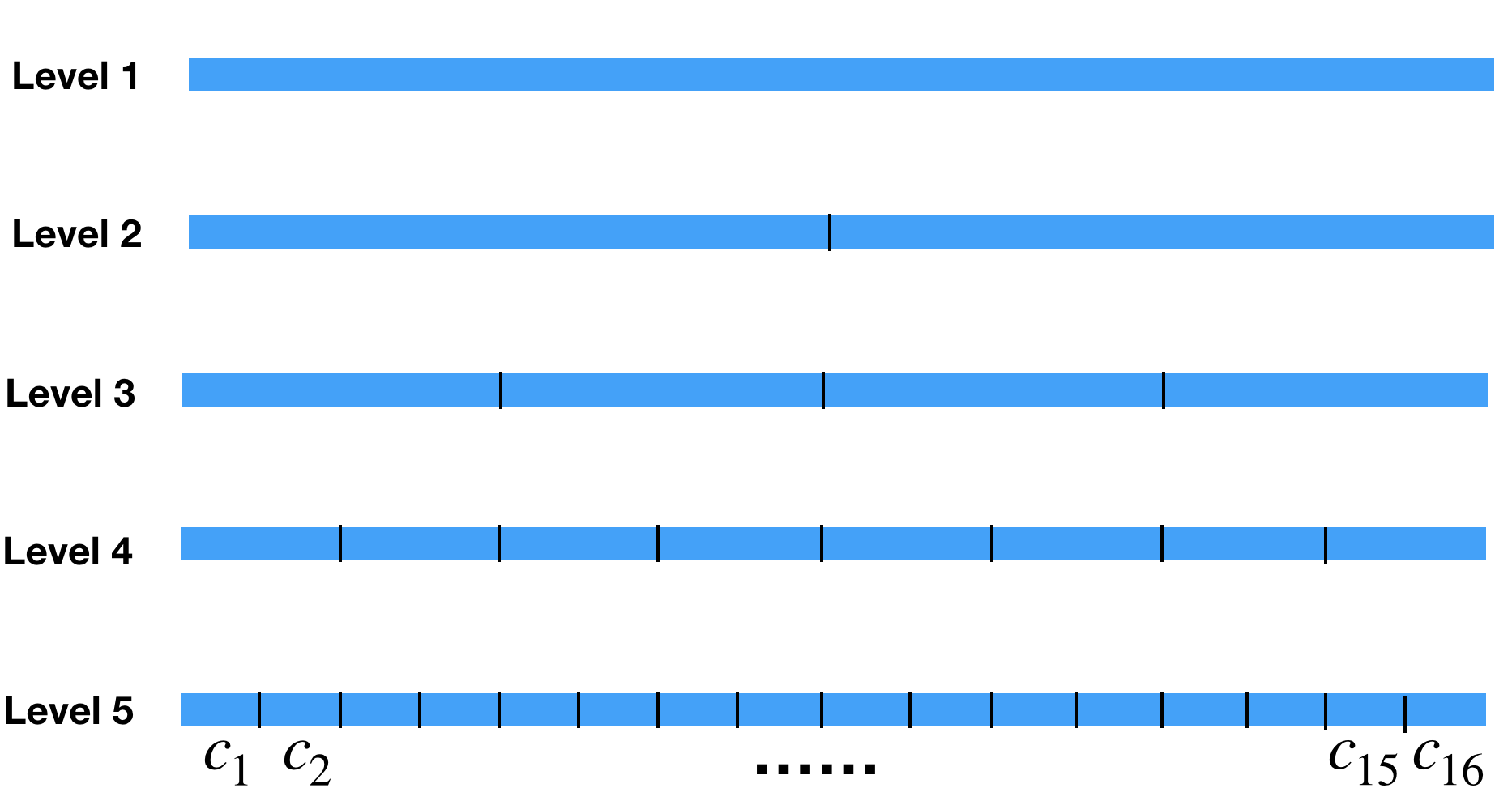}&
      \includegraphics[height=1.7in]{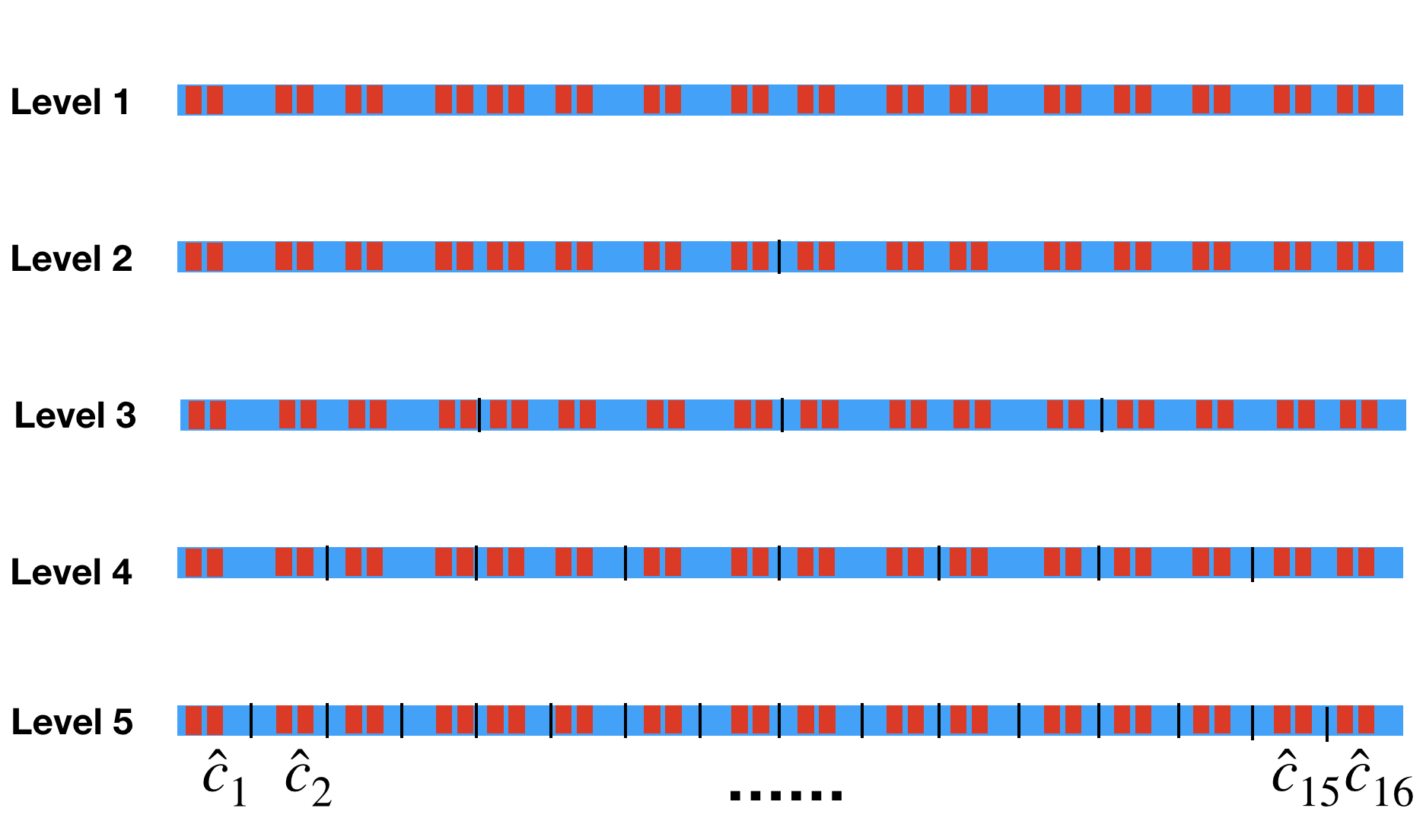}
    \end{tabular}
  \end{center}
\caption{Left: The dyadic tree $T_\Omega$ of the column space with leaves $\{c_i\}_{1\leq i\leq 16}$ denoted as in \eqref{eq:partrc} for the example in Figure \ref{fig:partition-K}. Right: Selected important columns of $T_\Omega$ naturally form a compressed dyadic tree in red with leaves $\{\hat{c}_i\}_{1\leq i\leq 16}$. }
\label{fig:col1}
\end{figure}

\begin{figure}[ht!]
  \begin{center}
    \begin{tabular}{cc}
      \includegraphics[height=1.7in]{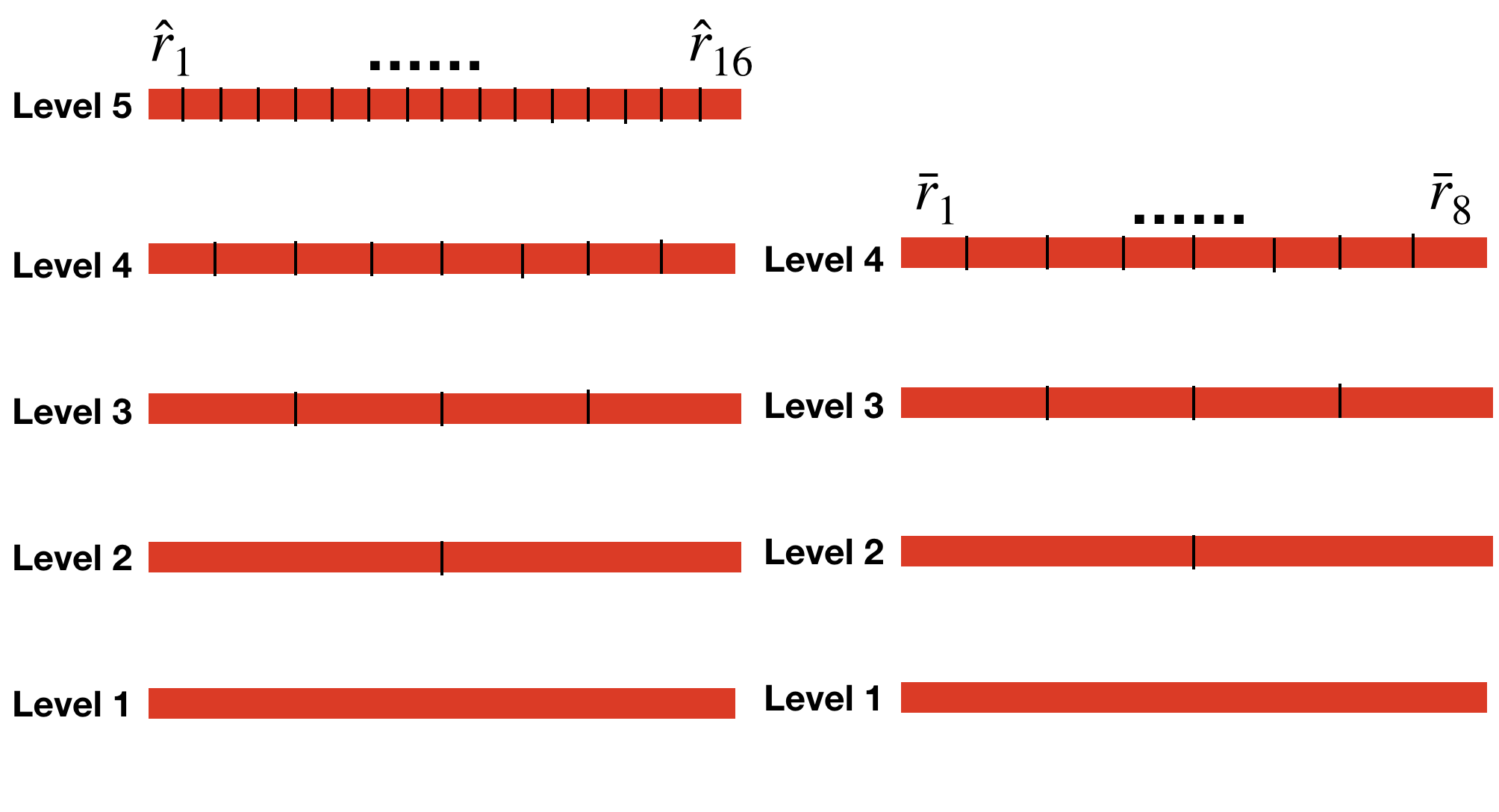}&
      \includegraphics[height=1.7in]{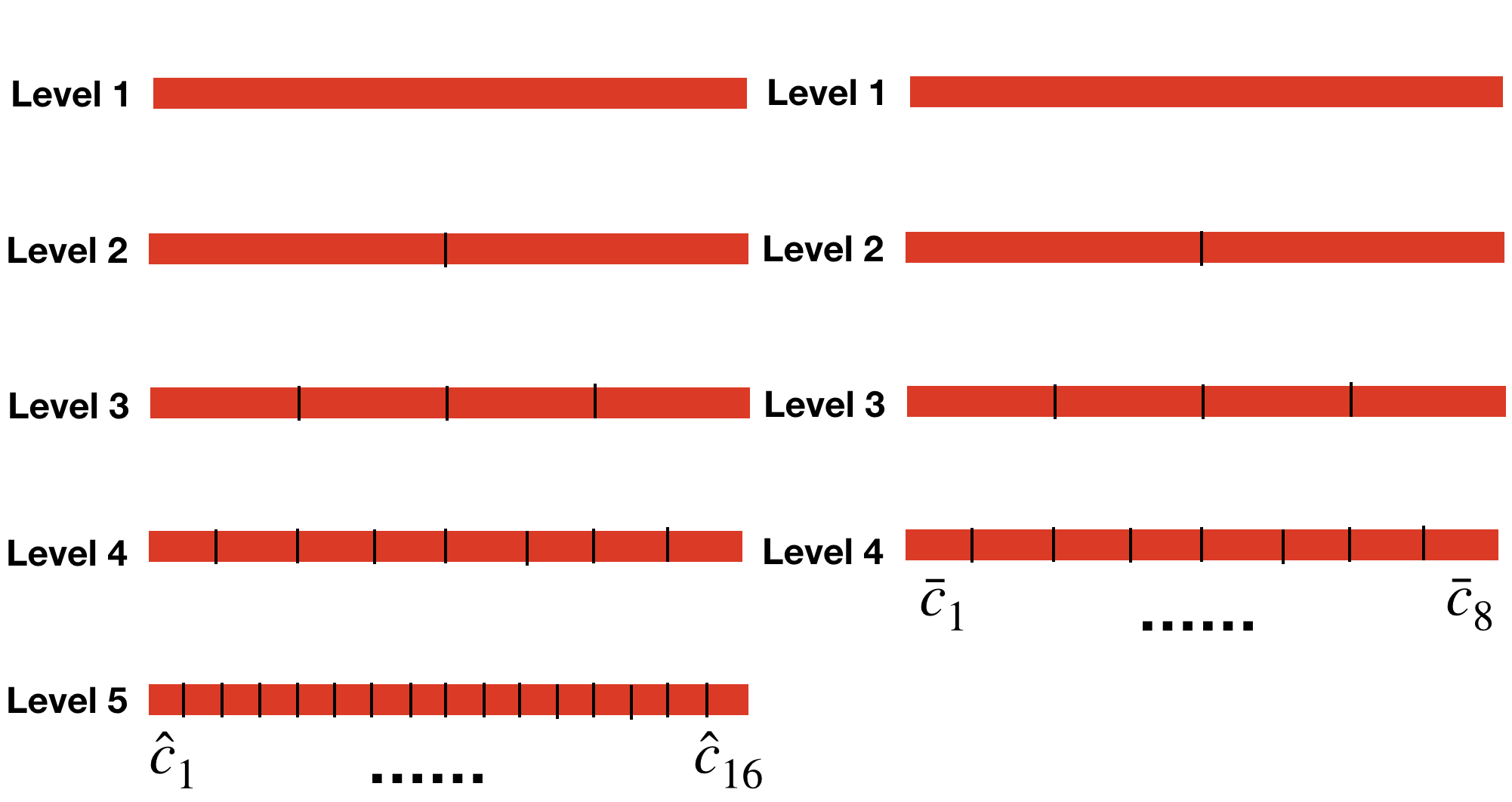}
    \end{tabular}
  \end{center}
\caption{Left: The compressed dyadic tree of $T_X$ of the row space in Figure \ref{fig:row1}. Level $5$ is not compressible. Middle left: Combining adjacent leaves at Level $5$, i.e., $\bar{r}_i=\hat{r}_{2i-1}\cup \hat{r}_{2i}$, forms a compressible dyadic tree with depth $4$. Middle right: the compressed dyadic tree of $T_\Omega$ of the column space in Figure \ref{fig:col1}. Level $5$ is not compressible. Right: Combining adjacent leaves at Level $5$, i.e., $\bar{c}_i=\hat{c}_{2i-1}\cup \hat{c}_{2i}$, forms a compressible dyadic tree with depth $4$.}
\label{fig:rc}
\end{figure}

\subsection{Matrix splitting with complementary skeletonization (MSCS)}
\label{sec:matrixsplit}

Here we describe another elementary idea of IDBF that is applied repeatedly: MSCS. A complementary low-rank matrix $A$ (with row and column dyadic trees $T_{X}$ and $T_{\Omega}$ of depth $L$ and with $m=2^L$ leaves) can be split into a $2\times 2$ block matrix
\begin{equation}
\label{eq:matrixsplit}
A = 
\begin{pmatrix}
A_{11} & A_{12} \\
A_{21} & A_{22}
\end{pmatrix}
\end{equation}
according to the nodes of the second level of the dyadic trees $T_X$ and $T_\Omega$ (those nodes right next to the root level). By the complementary low-rank property of $A$, we know that $A_{ij}$ is also complementary low-rank, for all $i$ and $j$, with row and column dyadic trees $T_{X,ij}$ and $T_{\Omega,ij}$ of depth $L-1$ and with $m/2$ leaves.

Suppose {$A_{ij} \approx U_{ij}S_{ij}V_{ij}$}, for $i,j = 1,2$, is the LRCS of $A_{ij}$.
Then $A$ can be factorized as {$A \approx USV$}, where
\begin{equation}
\label{eq:expressUSV}
\begin{split}
U &= 
\begin{pmatrix}
U_{11} & & U_{12} & \\
& U_{21} & & U_{22}
\end{pmatrix},\\
S &= \begin{pmatrix}
S_{11} & & & \\
& & S_{21} & \\
& S_{12} & & \\
& & & S_{22} 
\end{pmatrix},\\
V &= 
\begin{pmatrix}
V_{11} & \\
& V_{12} \\
V_{21} & \\
& V_{22}
\end{pmatrix}.
\end{split}
\end{equation}
The factorization in \eqref{eq:expressUSV} is referred as the {\it matrix splitting with complementary skeletonization (MSCS)} in this paper. Recall that the middle factor $S$ is not explicitly computed, resulting in a linear scaling algorithm for forming \eqref{eq:expressUSV}. Figure \ref{fig:exlfac} visualizes the MSCS of a complementary low-rank matrix $A$ with dyadic trees of depth $5$ and $16$ leaf nodes in Figure \ref{fig:partition-K}.

\begin{figure}[htp]
\begin{minipage}{\textwidth}
\centering
\resizebox{3.5cm}{!}{
\begin{tikzpicture}[baseline=-0.5ex]
      \tikzset{every left delimiter/.style={xshift=-1ex},every right delimiter/.style={xshift=1ex}}
      \matrix (mat) [matrix of math nodes, left delimiter=(, right delimiter=)] {
\draw[fill=gray] (32,0) rectangle (0,32);
\\
      };
\end{tikzpicture}
}
$\approx$
\resizebox{3.5cm}{!}{
\begin{tikzpicture}[baseline=-0.5ex]
      \tikzset{every left delimiter/.style={xshift=-1ex},every right delimiter/.style={xshift=1ex}}
      \matrix (mat) [matrix of math nodes, left delimiter=(, right delimiter=)] {
      \draw;
\foreach \j in {0,1,...,15}{
    \draw[fill=gray] (\j,32-\j-\j) rectangle (\j+1,32-\j-\j-2);
    \draw[fill=gray] (16+\j,32-\j-\j) rectangle (16+\j+1,32-\j-\j-2);
}
\\
      };
\end{tikzpicture}
}
\resizebox{3.5cm}{!}{
\begin{tikzpicture}[baseline=-0.5ex]
      \tikzset{every left delimiter/.style={xshift=-1ex},every right delimiter/.style={xshift=1ex}}
      \matrix (mat) [matrix of math nodes, left delimiter=(, right delimiter=)] {
      \draw;
      \draw[fill=gray] (0,32) rectangle (8,24);
      \draw[fill=gray] (16,16) rectangle (24,24); 
      \draw[fill=gray] (8,8) rectangle (16,16);  
      \draw[fill=gray] (24,0) rectangle (32,8); 
\\
      };
\end{tikzpicture}
}
\resizebox{3.5cm}{!}{
\begin{tikzpicture}[baseline=-0.5ex]
      \tikzset{every left delimiter/.style={xshift=-1ex},every right delimiter/.style={xshift=1ex}}
      \matrix (mat) [matrix of math nodes, left delimiter=(, right delimiter=)] {
      \draw;
\foreach \j in {0,1,...,15}{
    \draw[fill=gray] (\j+\j,32-\j) rectangle (\j+\j+2,32-\j-1);
    \draw[fill=gray] (\j+\j,32-\j-16) rectangle (\j+\j+2,32-\j-17);
}
\\
      };
\end{tikzpicture}
}\\
\end{minipage}
\caption{The visualization of a MSCS of a complementary low-rank matrix $A \approx USV$ with dyadic trees of depth $5$ and $16$ leaf nodes in Figure \ref{fig:partition-K}. Non-zero blocks in \eqref{eq:expressUSV} are shown in gray areas.}
\label{fig:exlfac}
\end{figure}
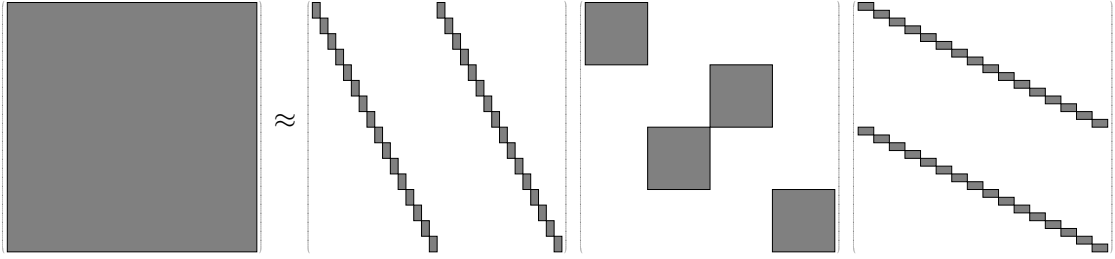

\subsection{Recursive MSCS}
\label{sec:recursivefac}
Now we apply MSCS recursively to get  the full
IDBF of a complementary low-rank matrix $A$ (with row and column dyadic trees $T_{X}$ and $T_{\Omega}$ of depth $L$ and with $m=2^L$ leaves). As in \eqref{eq:expressUSV}, suppose we have constructed the first level of MSCS and denote it as
\begin{equation}
\label{eq:leaffac}
{A \approx U^{L}S^{L}V^{L}}
\end{equation}
with
\begin{equation}
\label{eqn:USV2}
\begin{split}
U^{L} &= 
\begin{pmatrix}
U^{L}_{11} & & U^{L}_{12} & \\
& U^{L}_{21} & & U^{L}_{22}
\end{pmatrix},\\
S^{L} &= \begin{pmatrix}
S^{L}_{11} & & & \\
& & S^{L}_{21} & \\
& S^{L}_{12} & & \\
& & & S^{L}_{22} 
\end{pmatrix},\\
V^{L} &= 
\begin{pmatrix}
V^{L}_{11} & \\
& V^{L}_{12} \\
V^{L}_{21} & \\
& V^{L}_{22}
\end{pmatrix},
\end{split}
\end{equation}
as in \eqref{eq:expressUSV}.

Suppose that at the leaf level of the row and column dyadic trees, the row index set $r$ and the column index set $c$ of $A$ are partitioned into leaves $\{r_i\}_{1\leq i\leq m}$ and $\{c_i\}_{1\leq i\leq m}$ as in \eqref{eq:partrc}. By the $\rID$s and $\cID$s applied in the construction of \eqref{eq:leaffac}, we have obtained skeleton index sets $\hat{r}_i\subset r_i$ and $\hat{c}_i\subset c_i$. Then
\begin{equation}
S^{L}_{ij} = 
\begin{pmatrix}
A_{\hat{r}_{(i-1)m/2+1},\hat{c}_{(j-1)m/2+1}} & \cdots & A_{\hat{r}_{(i-1)m/2+1},\hat{c}_{jm/2}} \\
\vdots & \ddots & \vdots \\
A_{\hat{r}_{im/2},\hat{c}_{(j-1)m/2+1}} & \cdots & A_{\hat{r}_{im/2},\hat{c}_{jm/2}} \\
\end{pmatrix}
\end{equation}
for $i,j = 1,2$. 

As explained in Section \ref{sec:lrskeleton}, each non-zero block $S^L_{ij}$ in $S^{L}$  is a submatrix of $A_{ij}$ consisting of important rows and columns of $A_{ij}$ for $i,j=1,2$. Hence, $S^L_{ij}$ inherits the complementary low-rank property of $A_{ij}$ and is {itself} a complementary low-rank matrix. Suppose $T_{X,ij}$ and $T_{\Omega,ij}$ are the dyadic trees of the row and column spaces of $A_{ij}$ with $m/2$ {leaves} and $L-1$ depth, then according to Section \ref{sec:lrskeleton}, $S^L_{ij}$ has compressible row and column dyadic trees $\hat{T}_{X,ij}$ and $\hat{T}_{\Omega,ij}$ with $m/4$ {leaves} and $L-2$ depth.

Next, we apply MSCS to each $S^L_{ij}$ in a recursive way. In particular, we divide each $S^{L}_{ij}$ into a $2\times 2$ block matrix according to the nodes at the second level of its row and column dyadic trees:
 \begin{equation}
\label{eq:sl}
S^{L}_{ij} = 
\begin{pmatrix}
(S^{L}_{ij})_{11} & (S^{L}_{ij})_{12} \\
(S^{L}_{ij})_{21} & (S^{L}_{ij})_{22}
\end{pmatrix}.
\end{equation}
After constructing the LRCS of the $(k,\ell)$-th block of $S^{L}_{ij}$, i.e., {$(S^{L}_{ij})_{k\ell} \approx (U^{L-1}_{ij})_{k\ell}(S^{L-1}_{ij})_{k\ell}(V^{L-1}_{ij})_{k\ell}$} for $k,\ell=1,2$, we assemble them to obtain the MSCS of $S^{L}_{ij} $ as follows:
\begin{equation}
\label{eq:facSL1234}
{S^{L}_{ij} \approx U^{L-1}_{ij}S^{L-1}_{ij}V^{L-1}_{ij},}
\end{equation}
where 
\begin{equation}
\label{eq:expressUSVL-1}
\begin{split}
U^{L-1}_{ij} &= 
\begin{pmatrix}
(U^{L-1}_{ij})_{11} & & (U^{L-1}_{ij})_{12} & \\
& (U^{L-1}_{ij})_{21} & & (U^{L-1}_{ij})_{22}
\end{pmatrix},\\
S^{L-1}_{ij} &= 
\begin{pmatrix}
(S^{L-1}_{ij})_{11} & & & \\
& & (S^{L-1}_{ij})_{21} & \\
& (S^{L-1}_{ij})_{12} & & \\
& & & (S^{L-1}_{ij})_{22}
\end{pmatrix},\\
V^{L-1}_{ij} &= 
\begin{pmatrix}
(V^{L-1}_{ij})_{11} & \\
& (V^{L-1}_{ij})_{12} \\
(V^{L-1}_{ij})_{21} & \\
& (V^{L-1}_{ij})_{22}
\end{pmatrix},
\end{split}
\end{equation}
according to Section \ref{sec:matrixsplit}.
 
Finally, we organize the factorizations in \eqref{eq:facSL1234} for all $i,j=1,2$ to form a factorization of $S^{L}$ as
\begin{equation}
\label{eq:facSL}
S^{L} \approx U^{L-1}S^{L-1}V^{L-1},
\end{equation}
where 
\begin{equation}
\label{eq:expressUSV_L-1}
\begin{split}
U^{L-1} &= 
\begin{pmatrix}
U^{L-1}_{11} & & & \\
& U^{L-1}_{21} & & \\
& & U^{L-1}_{12} & \\
& & & U^{L-1}_{11}
\end{pmatrix},\\
S^{L-1} &= 
\begin{pmatrix}
S^{L-1}_{11} & & & \\
& & S^{L-1}_{21} & \\
& S^{L-1}_{12} & & \\
& & & S^{L-1}_{22}
\end{pmatrix},\\
V^{L-1} &= 
\begin{pmatrix}
V^{L-1}_{11} & & & \\
& V^{L-1}_{12} & & \\
& & V^{L-1}_{21} & \\
& & & S^{L-1}_{22}
\end{pmatrix},
\end{split}
\end{equation}
leading to a second level factorization of $A$:
\[
A\approx U^L U^{L-1}S^{L-1}V^{L-1}V^L.
\]
Figure \ref{fig:exsfac} visualizes the recursive MSCS of $S^L$ in \eqref{eq:facSL} when $A$ is a complementary low-rank matrix with dyadic trees of depth $5$ and $16$ leaf nodes in Figure \ref{fig:partition-K}.

\begin{figure}[htp]
\begin{minipage}{\textwidth}
\centering
\resizebox{3.5cm}{!}{
\begin{tikzpicture}[baseline=-0.5ex]
      \tikzset{every left delimiter/.style={xshift=-1ex},every right delimiter/.style={xshift=1ex}}
      \matrix (mat) [matrix of math nodes, left delimiter=(, right delimiter=)] {
      \draw;
      \draw[fill=gray] (0,32) rectangle (8,24);
      \draw[fill=gray] (16,16) rectangle (24,24); 
      \draw[fill=gray] (8,8) rectangle (16,16);  
      \draw[fill=gray] (24,0) rectangle (32,8); 
\\
      };
\end{tikzpicture}
}
$\approx$
\resizebox{3.5cm}{!}{
\begin{tikzpicture}[baseline=-0.5ex]
      \tikzset{every left delimiter/.style={xshift=-1ex},every right delimiter/.style={xshift=1ex}}
      \matrix (mat) [matrix of math nodes, left delimiter=(, right delimiter=)] {
      \draw;
\foreach \i in {0,8,...,24}{
\foreach \j in {0,1,...,3}{
    \draw[fill=lightgray] (\i+\j,32-\i-\j-\j) rectangle (\i+\j+1,32-\i-\j-\j-2);
    \draw[fill=lightgray] (\i+\j+4,32-\i-\j-\j) rectangle (\i+\j+5,32-\i-\j-\j-2);
}
}\\
      };
\end{tikzpicture}
}
\resizebox{3.5cm}{!}{
\begin{tikzpicture}[baseline=-0.5ex]
      \tikzset{every left delimiter/.style={xshift=-1ex},every right delimiter/.style={xshift=1ex}}
      \matrix (mat) [matrix of math nodes, left delimiter=(, right delimiter=)] {
      \draw;
\draw[fill=lightgray] (0,32) rectangle (2,30);
\draw[fill=lightgray] (4,30) rectangle (6,28);
\draw[fill=lightgray] (2,28) rectangle (4,26);
\draw[fill=lightgray] (6,26) rectangle (8,24);
\draw[fill=lightgray] (16,24) rectangle (18,22);
\draw[fill=lightgray] (20,22) rectangle (22,20);
\draw[fill=lightgray] (18,20) rectangle (20,18);
\draw[fill=lightgray] (22,18) rectangle (24,16);
\draw[fill=lightgray] (8,16) rectangle (10,14);
\draw[fill=lightgray] (12,14) rectangle (14,12);
\draw[fill=lightgray] (10,12) rectangle (12,10);
\draw[fill=lightgray] (14,10) rectangle (16,8);
\draw[fill=lightgray] (24,8) rectangle (26,6);
\draw[fill=lightgray] (28,6) rectangle (30,4);
\draw[fill=lightgray] (26,4) rectangle (28,2);
\draw[fill=lightgray] (30,2) rectangle (32,0);
\\
      };
\end{tikzpicture}
}
\resizebox{3.5cm}{!}{
\begin{tikzpicture}[baseline=-0.5ex]
      \tikzset{every left delimiter/.style={xshift=-1ex},every right delimiter/.style={xshift=1ex}}
      \matrix (mat) [matrix of math nodes, left delimiter=(, right delimiter=)] {
      \draw;
\foreach \i in {0,8,...,24}{
\foreach \j in {0,1,...,3}{
    \draw[fill=lightgray] (\i+\j+\j,32-\i-\j) rectangle (\i+\j+\j+2,32-\i-\j-1);
    \draw[fill=lightgray] (\i+\j+\j,32-\i-\j-4) rectangle (\i+\j+\j+2,32-\i-\j-5);
}
}\\
      };
\end{tikzpicture}
}\\
\end{minipage}
\caption{The visualization of the recursive MSCS of $S^L=U^{L-1}S^{L-1}V^{L-1}$ in \eqref{eq:facSL} when $A$ is a complementary low-rank matrix with dyadic trees of depth $5$ and $16$ leaf nodes in Figure \ref{fig:partition-K}.}
\label{fig:exsfac}
\end{figure}
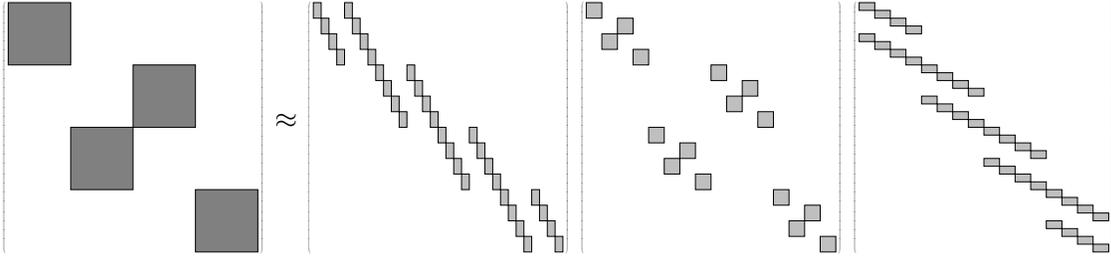

Comparing \eqref{eq:leaffac}, \eqref{eqn:USV2}, \eqref{eq:facSL}, and \eqref{eq:expressUSV_L-1}, we can see a fractal structure in each level of the middle factor $S^{\ell}$ for $\ell=L$ and $L-1$. For example in \eqref{eq:expressUSV_L-1} (see Figure \ref{fig:exsfac} for its visulaization), $S^{L-1}$ has $4$ submatrices $S^{L-1}_{ij}$ with the same structure as $S^{L}$ for all $i$ and $j$. $S^{L-1}_{ij}$ can be factorized into a product of three matrices with the same sparsity structure as the factorization {$S^L \approx U^{L-1}S^{L-1}V^{L-1}$}. Hence, we can apply MSCS recursively to each $S^{\ell}$ and assemble matrix factors hierarchically for $\ell=L$, $L-1$, $\dots$, $L/2$ to obtain
\begin{equation}
\label{eq:finishIDBF}
A \approx U^{L}U^{L-1}\cdots U^{h} S^h V^{h}\cdots V^{L-1}V^{L},
\end{equation}
where $h=L/2$. In the $\ell$-th  recursive MSCS, $S^\ell$ has $2^{2(L-\ell+1)}$ dense submatrices with compressible row and column dyadic trees with $\frac{m}{2^{2(L-\ell+1)}}$ leaves and depth $L-2(L-\ell+1)$. Hence, the recursive MSCS stops after $h=L/2$ iterations when $S^h$ no longer contains any compressible submatrix.

When $S^\ell$ is still compressible, since there are $2^{2(L-\ell+1)}$ dense submatrices and each contains $\frac{m}{2^{2(L-\ell+1)}}$ leaves, there are $2^{2(L-\ell+1)}\frac{m}{2^{2(L-\ell+1)}}m=\frac{N}{n_0}$ low-rank submatrices to be factorized. Linear IDs only require $O(k^3)$ operations for each low-rank submatrix, and hence at most $O(\frac{k^3}{n_0} N)$ for each level of factorization, and $O(\frac{k^3}{n_0} N\log N)$ for the whole IDBF.

\section{Numerical results}
\label{sec:results}
This section presents several numerical examples to demonstrate the
effectiveness of the algorithms proposed above.  The first three examples 
are complementary low-rank matrices coming from non-uniform Fourier 
transform, Fourier integral operators, and special function transforms. The 
last two examples are hierarchical complementary matrices \cite{HSSBF} 
from 2D Helmholtz boundary integral methods in the high-frequency regime. All implementations are in MATLAB\textsuperscript{\textregistered} on a server computer with a single thread and 3.2 {GHz CPU}. This new framework will be incorperated into the ButterflyLab\footnote{Available on \url{https://github.com/ButterflyLab}.} in the future.

Let $\{u^d(x),x\in X\}$ and $\{u^a(x),x\in X\}$ denote the results
given by the direct matrix-vector multiplication and the butterfly
factorization.  The accuracy of applying the butterfly factorization
algorithm is estimated by the following relative error
\begin{equation}
\epsilon^a = \sqrt{\cfrac{\sum_{x\in S}|u^a(x)-u^d(x)|^2}
{\sum_{x\in S}|u^d(x)|^2}},
\end{equation}
where $S$ is a point set of size $256$ randomly sampled from $X$. In all of our examples, the oversampling parameter $t$ in the linear scaling ID is set to $1$, and the number of points in a leave node is set to $n_0=8$. {Then the} number of randomly sampled grid points {in the ID is equal to the rank parameter $k$, which we will here also call the truncation rank.}

\paragraph{Example 1.}
Our first example is to evaluate a {1D} FIO of the
following form:
\begin{equation}\label{eq:example1}
u(x) = \int_{\mathbb{R}}e^{2\pi \imath \Phi(x,\xi)}\hat{f}(\xi)d\xi,
\end{equation}
where $\hat{f}$ is the {Fourier} transform of $f$,
and $\Phi(x,\xi)$ is a phase function given by
\begin{equation}
\Phi(x,\xi) = x\cdot \xi + c(x)|\xi|,~~~c(x) = (2+0.2\sin(2\pi x))/16.
\label{eqn:1D-FIO-kernal}
\end{equation}
The discretization of \eqref{eq:example1} is
\begin{equation}
\label{eqn:1D-FIO}
u(x_i) = \sum_{\xi_j}e^{2\pi \imath \Phi(x_i,\xi_j)}\hat{f}(\xi_j),
\quad i,j=1,2,\dots,N,
\end{equation}
where $\{x_i\}$ and $\{\xi_j\}$ are uniformly distributed points
in $[0,1)$ and $[-N/2,N/2)$ following
\begin{equation}\label{eqn:1D-xandxi}
x_i = (i-1)/N \text{ and } \xi_j = j-1-N/2.
\end{equation}
\eqref{eqn:1D-FIO} can be represented in a matrix form as $u=Kg$,
where $u_i=u(x_i)$, $K_{ij} = e^{2\pi\imath \Phi(x_i,\xi_j)}$ and $g_j
= \hat{f}(\xi_j)$. The matrix $K$ satisfies the complementary
low-rank property with a rank parameter $k$ independent of the problem size $N$ when $\xi$ is sufficiently far away from the origin as proved in \cite{FIO09,FIO14}. To make the presentation simpler, we will directly apply IDBF to the whole $K$ instead of performing a polar transform as in \cite{FIO09} or apply IDBF hierarchically as in \cite{MBA}. Hence, due to the non-smoothness of the $\Phi(x,\xi)$ at $\xi=0$, submatrices intersecting with or close to the line $\xi=0$ have a local rank increasing slightly in $N$, while other submatrices have rank independent of $N$.  

Figure~\ref{fig:1D-fio1} to \ref{fig:1D-fio3} summarize the results of
this example for different grid sizes $N$. To compare IDs with Mock-Chebyshev points and randomly selected points in different cases, Figure \ref{fig:1D-fio1} shows the results for tolerance in \eqref{eqn:acc} $\epsilon=10^{-6}$ and the truncation rank $k$ being the smallest size of a submatrix (i.e., $k=\min\{m,n\}$ for a submatrix of size $m\times n$); Figure \ref{fig:1D-fio2} shows the results for $\epsilon=10^{-15}$ and $k=30$; Figure \ref{fig:1D-fio3} shows the results for $\epsilon=10^{-6}$ and $k=30$. Note that the accuracy of IDBF is expected to be $O(\epsilon)$, which may not be guaranteed, since the overall accuracy of IDBF is determined by all IDs in a hiearchical manner. Furthermore, if the rank parameter $k$ is too small for some low-rank matrices, then the error of the corresponding ID will propagate through the whole IDBF process and increase the error of the IDBF. 

We see that the IDBF applied to the whole matrix $K$ has $O(N\log^2 (N))$ factorization and application time in all cases with different parameters. The running time agrees with the scaling of the number of non-zero entries required in the data-sparse representation. In fact, when $N$ is large enough, the number of non-zero entries in the IDBF tends to scale as $O(N\log N)$, which means that the numerical scaling can approach to $O(N \log N)$ in both factorization and application when $N$ is large enough. IDBF via IDs with Mock-Chebyshev points is much more accurate than IDBF via IDs with random samples. The running time for three kinds of parameter pairs $(\epsilon,k)$ is almost the same. For the purpose of numerical accuracy, we prefer IDs with Mock-Chebyshev points with $(\epsilon,k)=(10^{-15},30)$. Hence, we will only present numerical results for IDs with Mock-Chebyshev points in later examples.

\begin{figure}[ht!]
  \begin{center}
    \begin{tabular}{ccc}
      \includegraphics[height=1.7in]{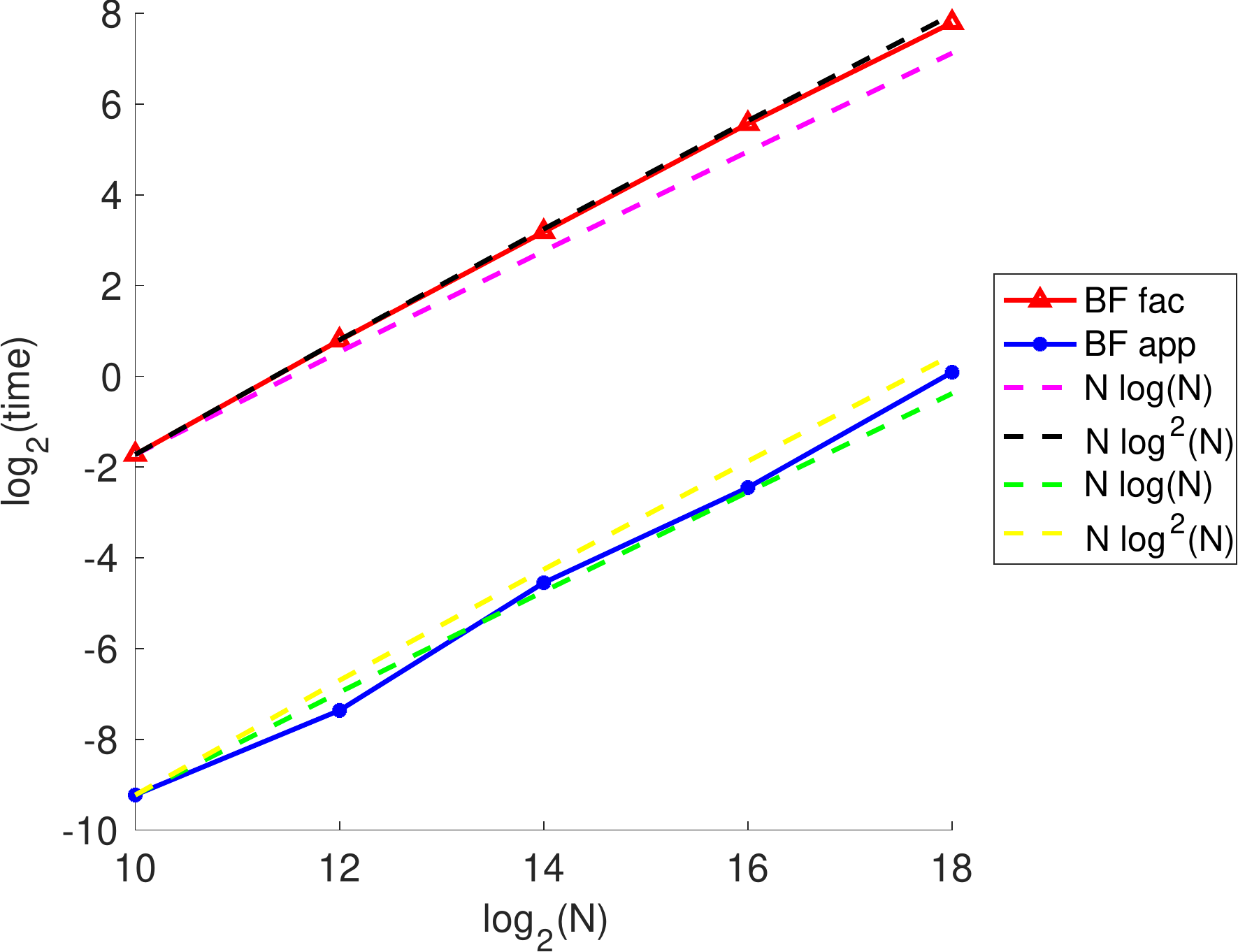}&
      \includegraphics[height=1.7in]{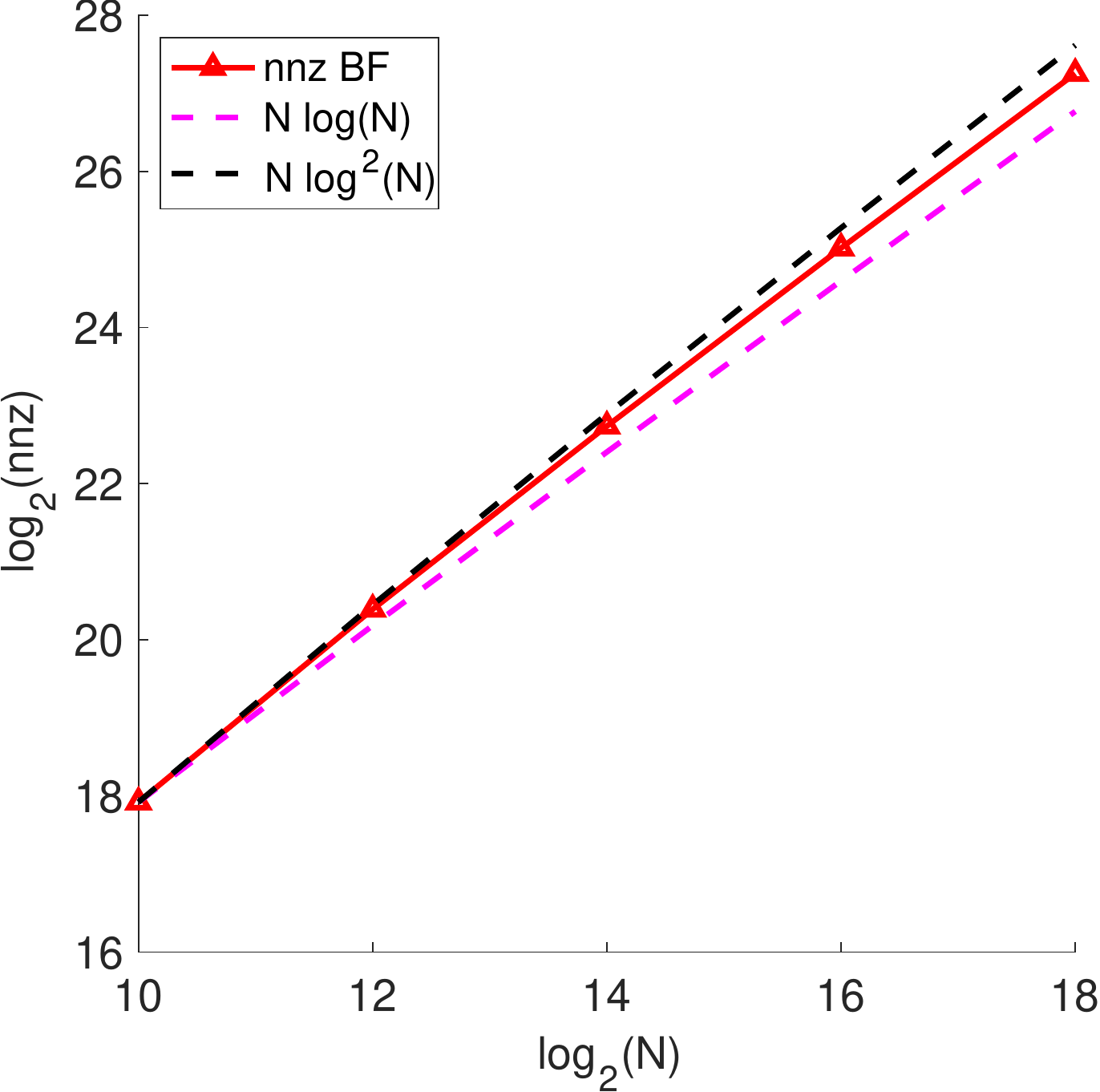}&
      \includegraphics[height=1.7in]{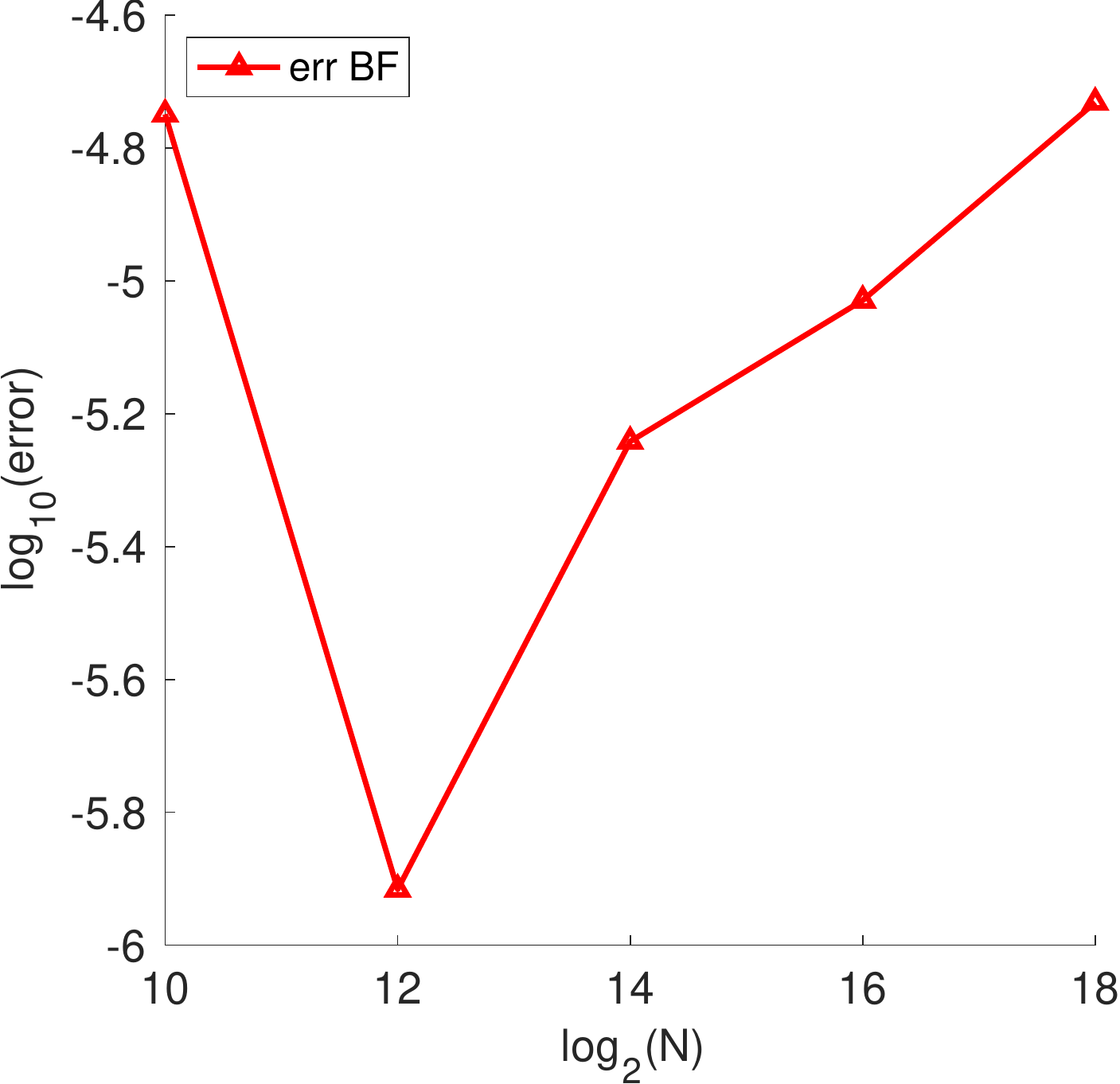}\\
      \includegraphics[height=1.7in]{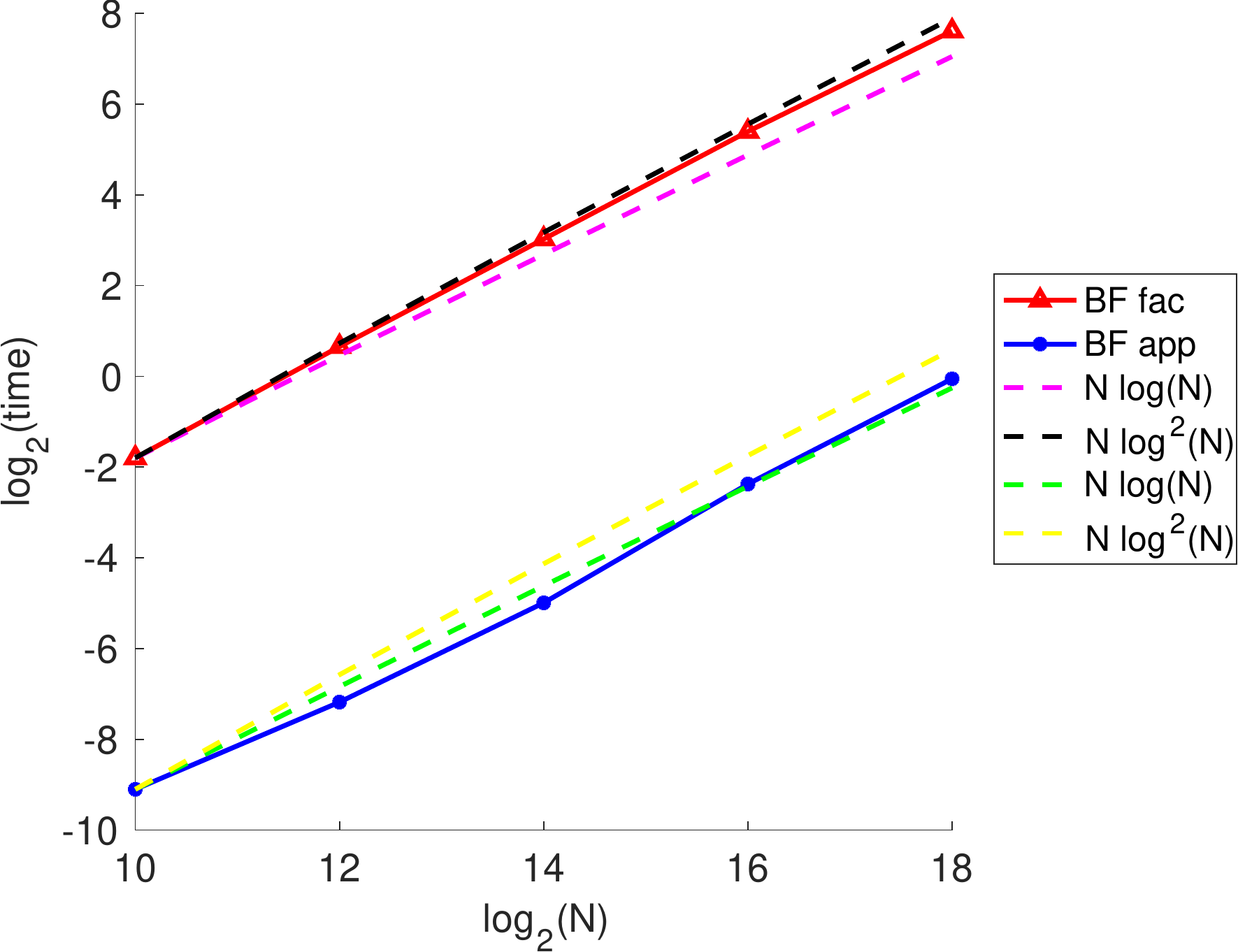}&
      \includegraphics[height=1.7in]{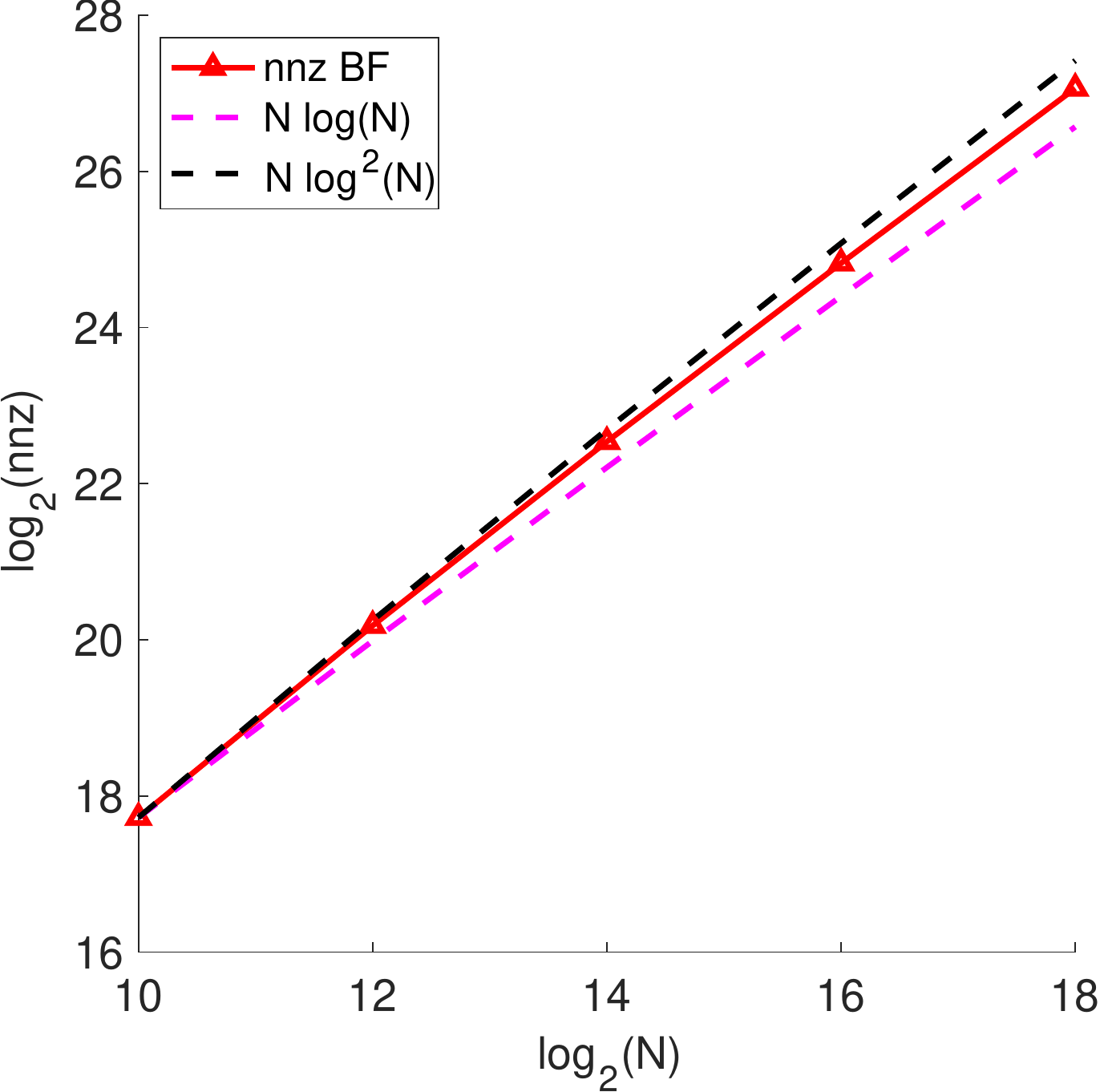}&
      \includegraphics[height=1.7in]{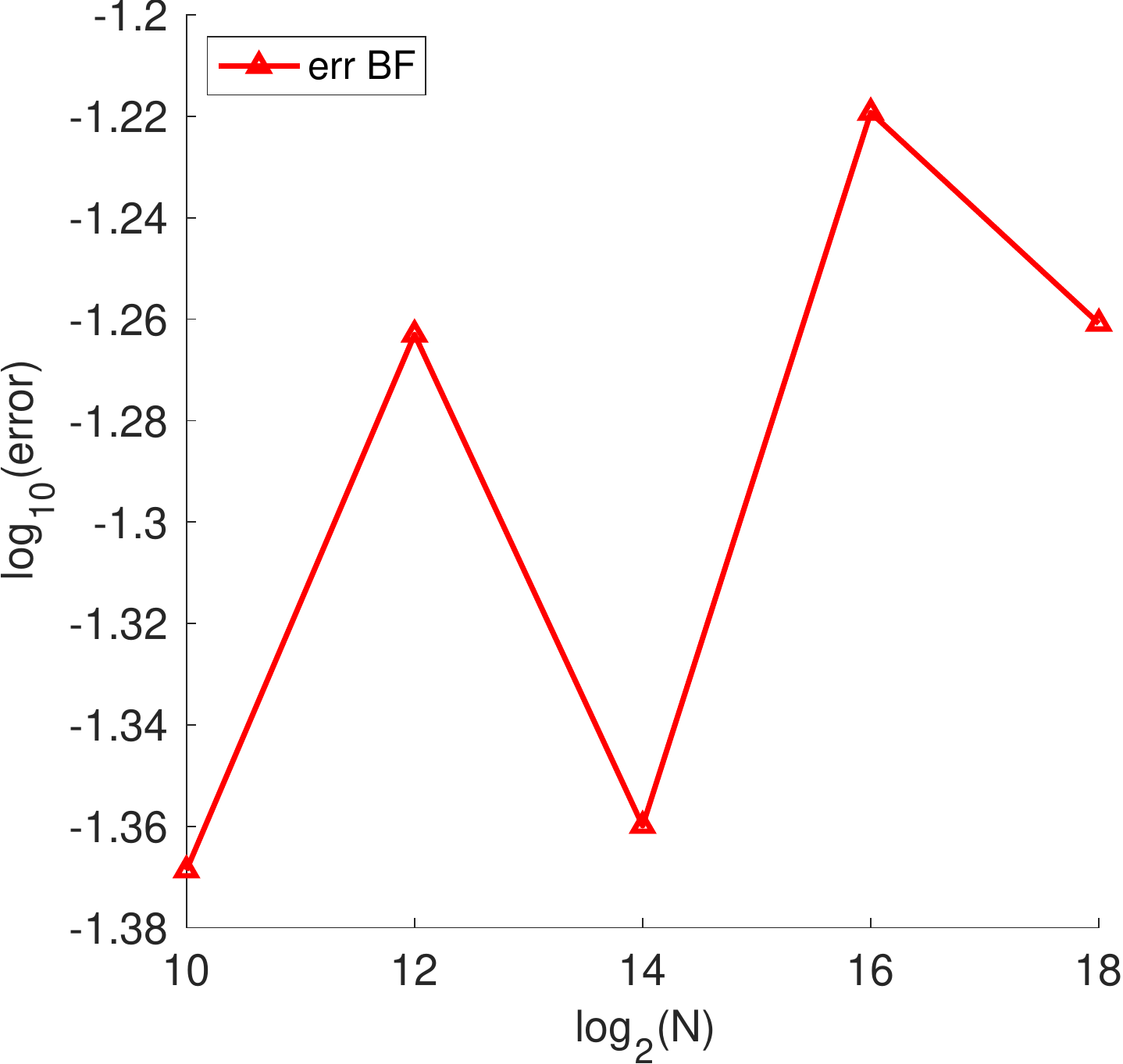}      
    \end{tabular}
  \end{center}
\caption{Numerical results for the FIO given in \eqref{eqn:1D-FIO}.
  $N$ is the size of the matrix; $nnz$ is the number of non-zero entries in the butterfly factorization, $err$ is the approximation error of the IDBF matvec. $\epsilon=10^{-6}$ and $k$ is the smallest size of a submatrix (i.e., $k=\min\{m,n\}$ for a submatrix of size $m\times n$).}
\label{fig:1D-fio1}
\end{figure}

\begin{figure}[ht!]
  \begin{center}
    \begin{tabular}{ccc}
      \includegraphics[height=1.7in]{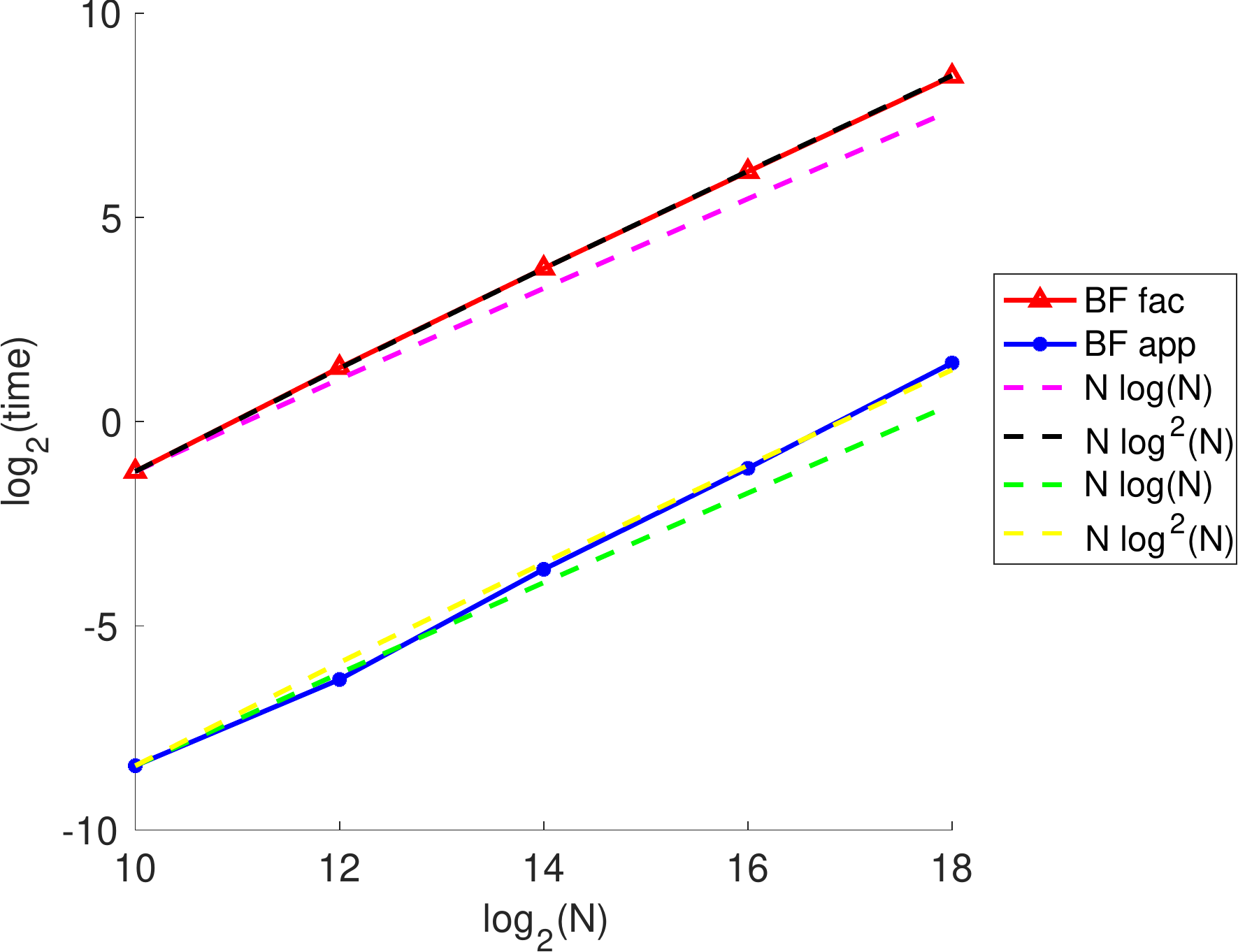}&
      \includegraphics[height=1.7in]{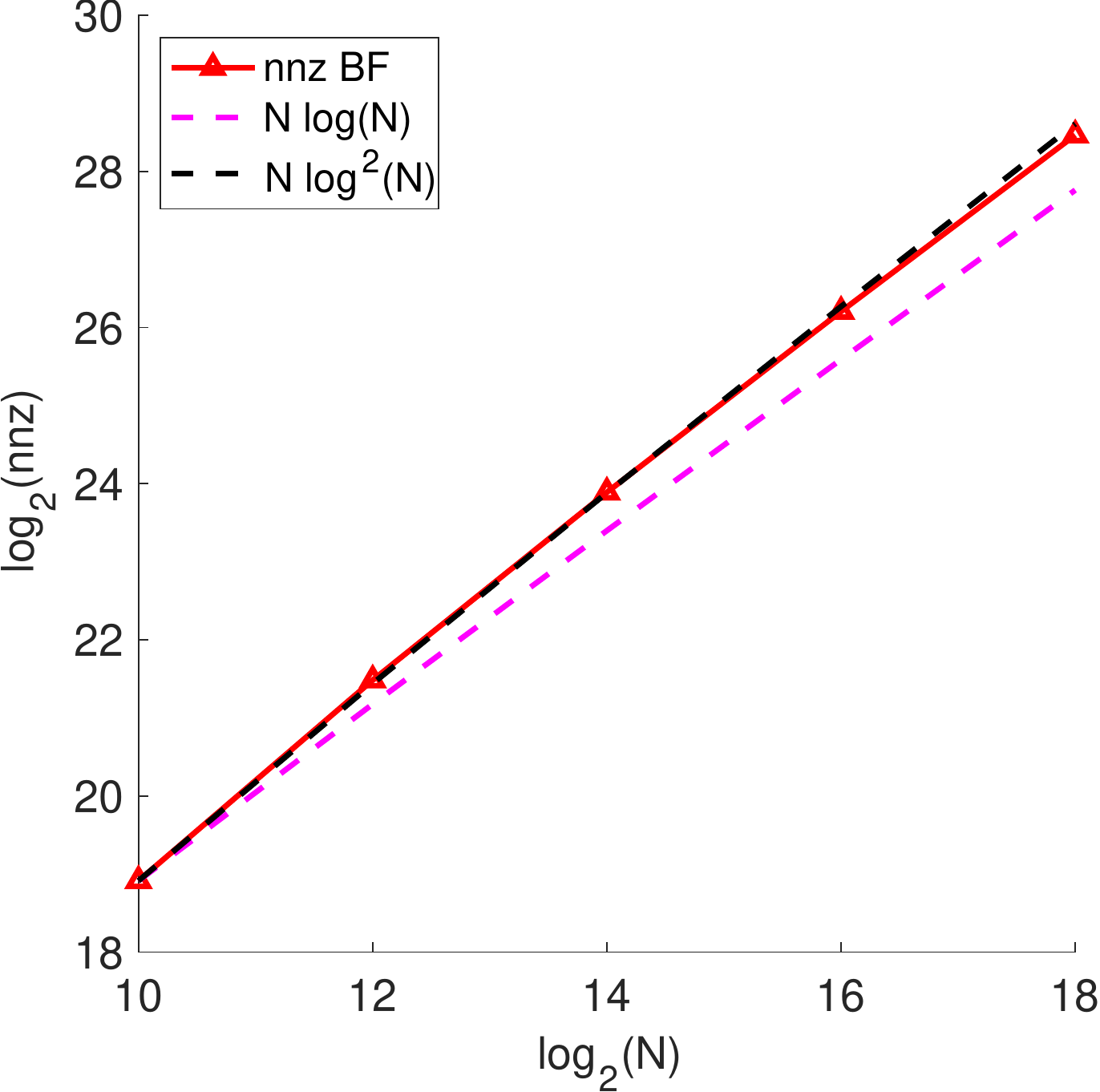}&
      \includegraphics[height=1.7in]{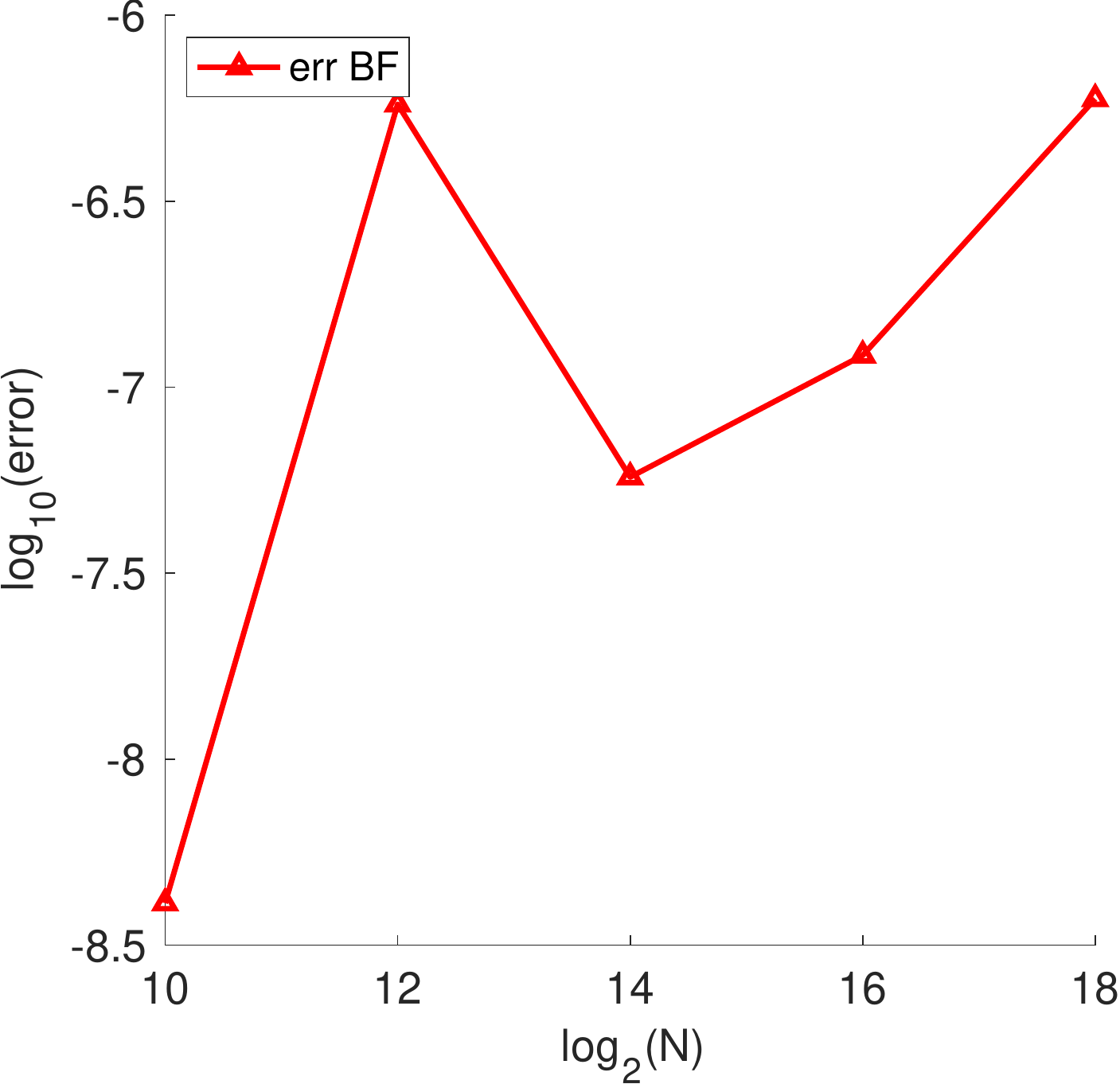}\\
      \includegraphics[height=1.7in]{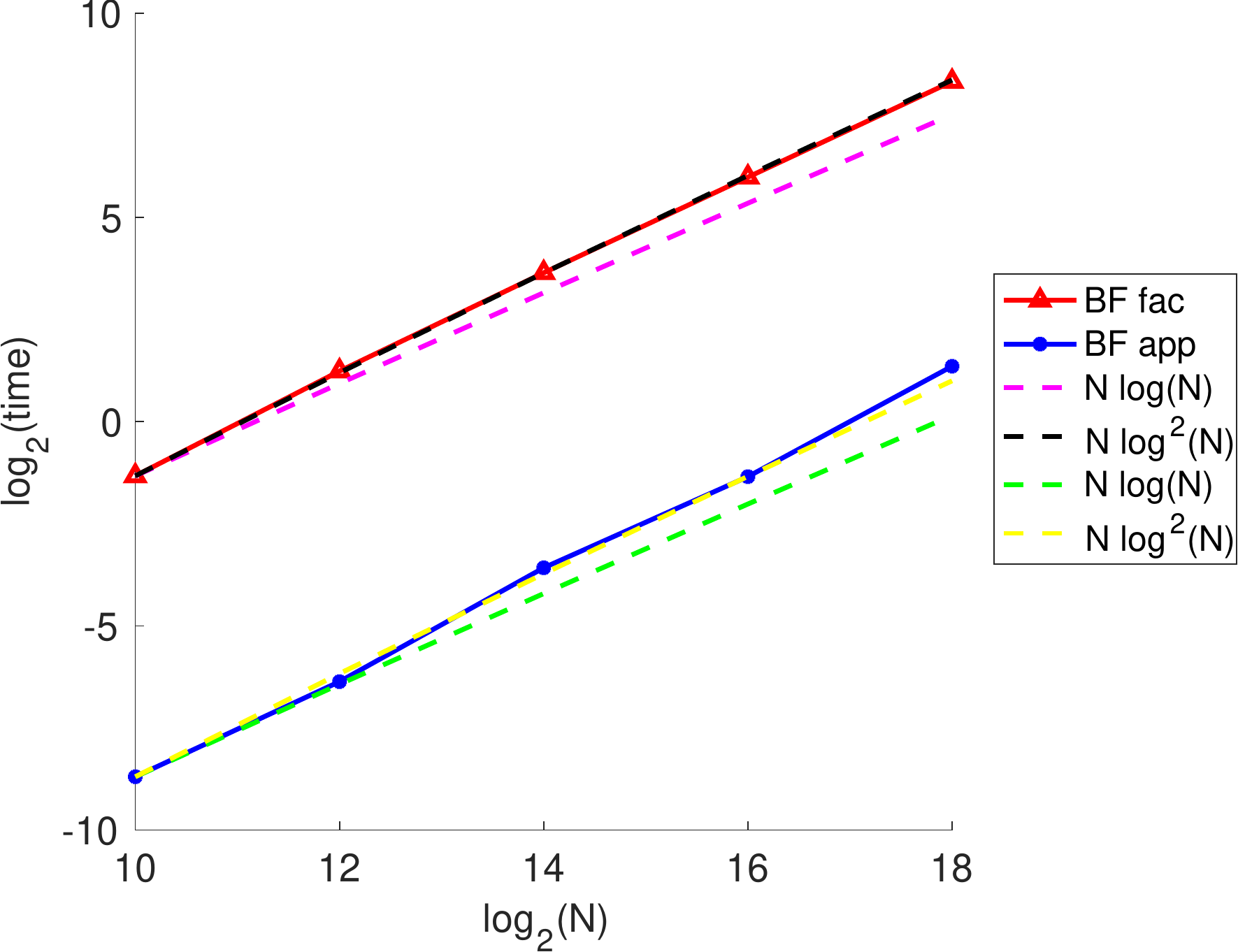}&
      \includegraphics[height=1.7in]{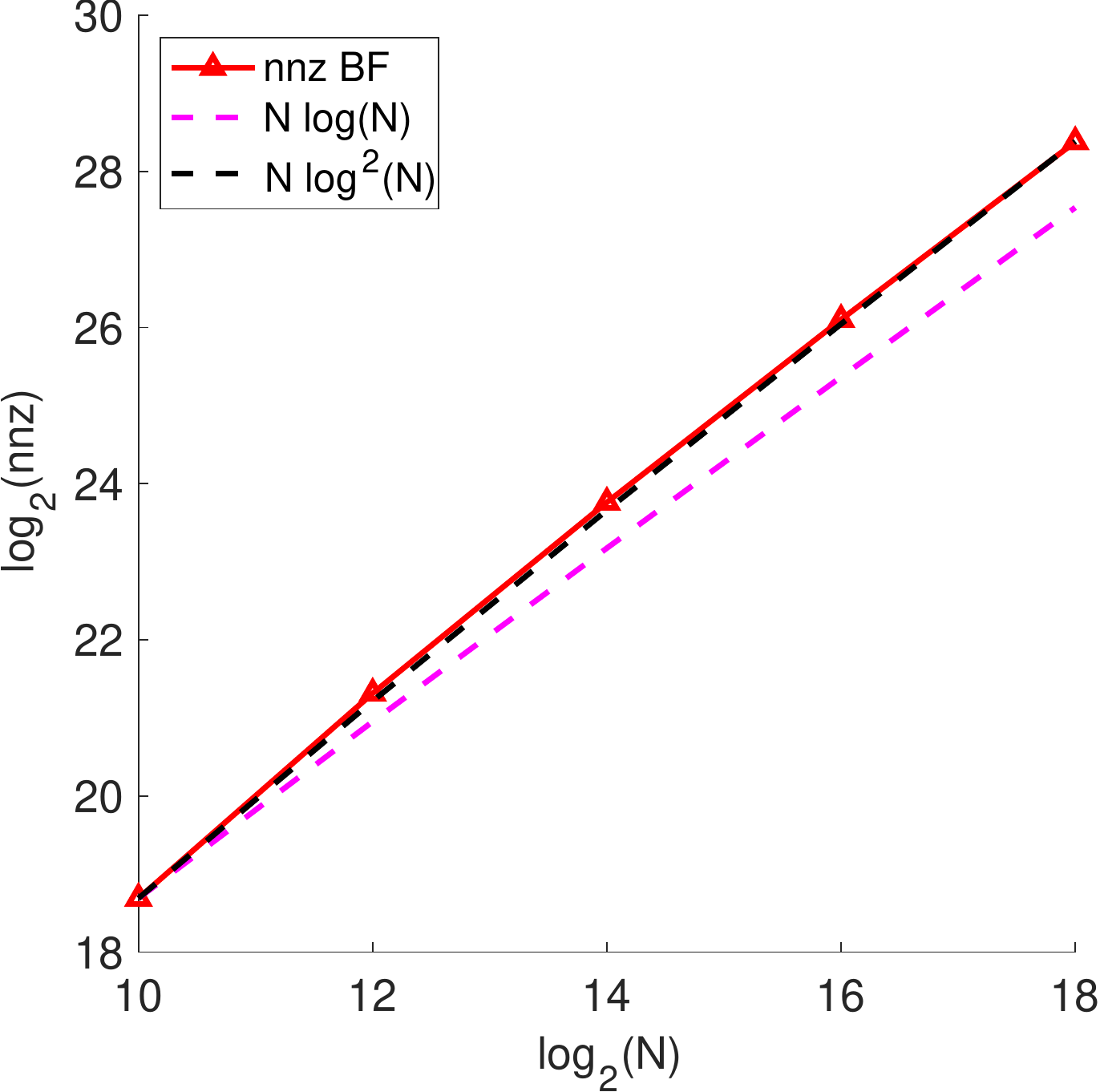}&
      \includegraphics[height=1.7in]{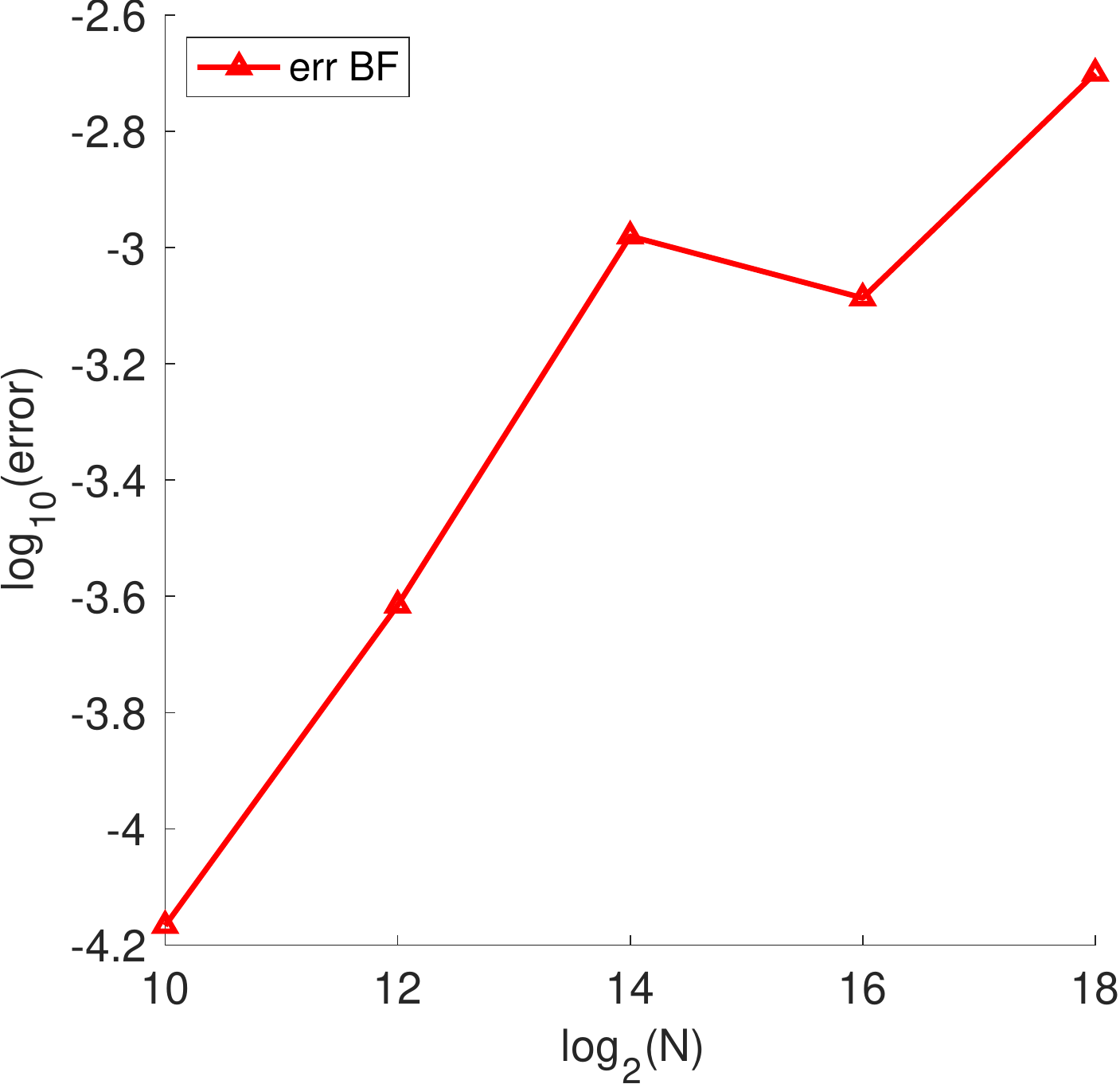}      
    \end{tabular}
  \end{center}
\caption{Numerical results for the FIO given in \eqref{eqn:1D-FIO}.
  $N$ is the size of the matrix; $nnz$ is the number of non-zero entries in the butterfly factorization, $err$ is the approximation error of the IDBF matvec. $\epsilon=10^{-15}$ and $k=30$.}
\label{fig:1D-fio2}
\end{figure}

\begin{figure}[ht!]
  \begin{center}
    \begin{tabular}{ccc}
      \includegraphics[height=1.7in]{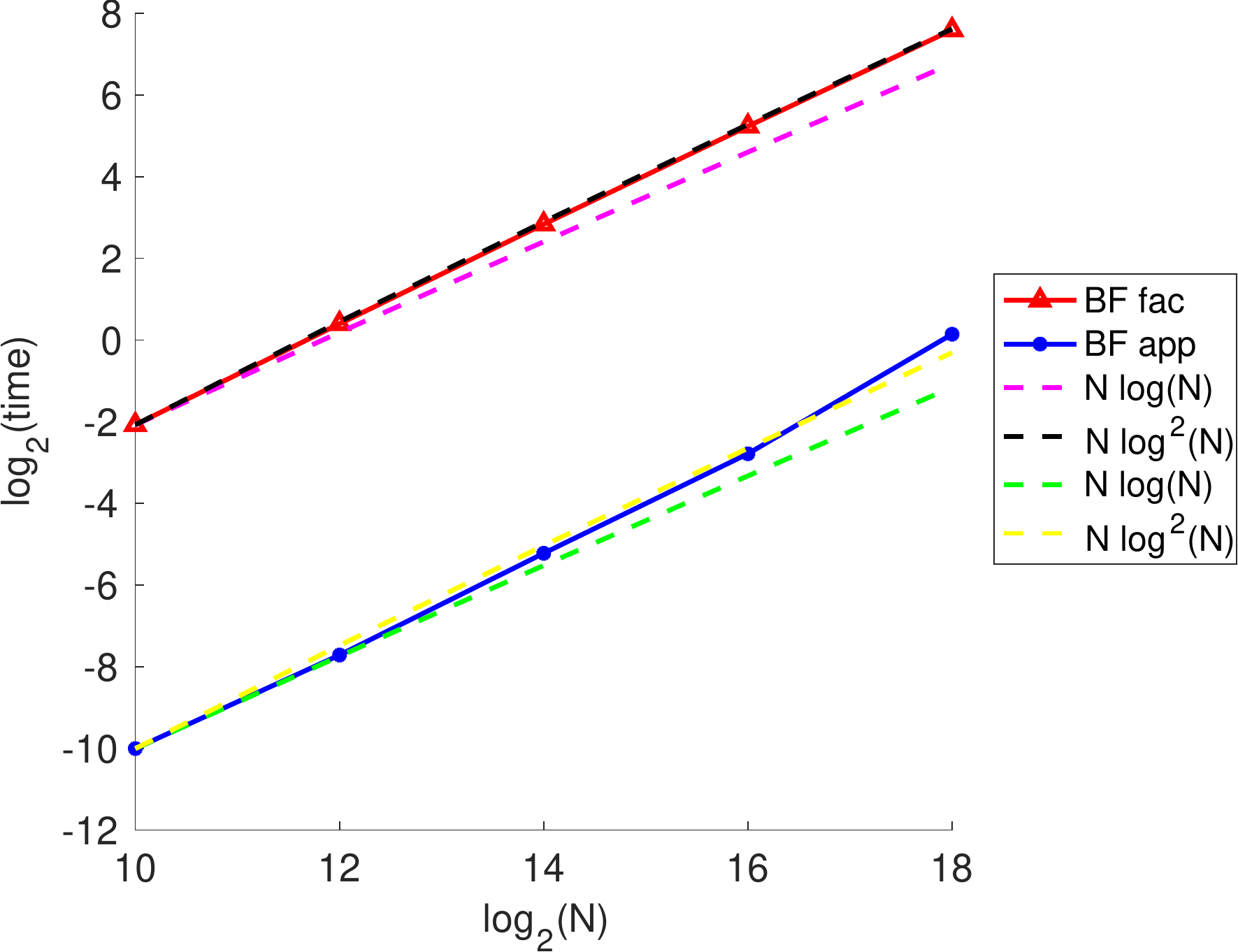}&
      \includegraphics[height=1.7in]{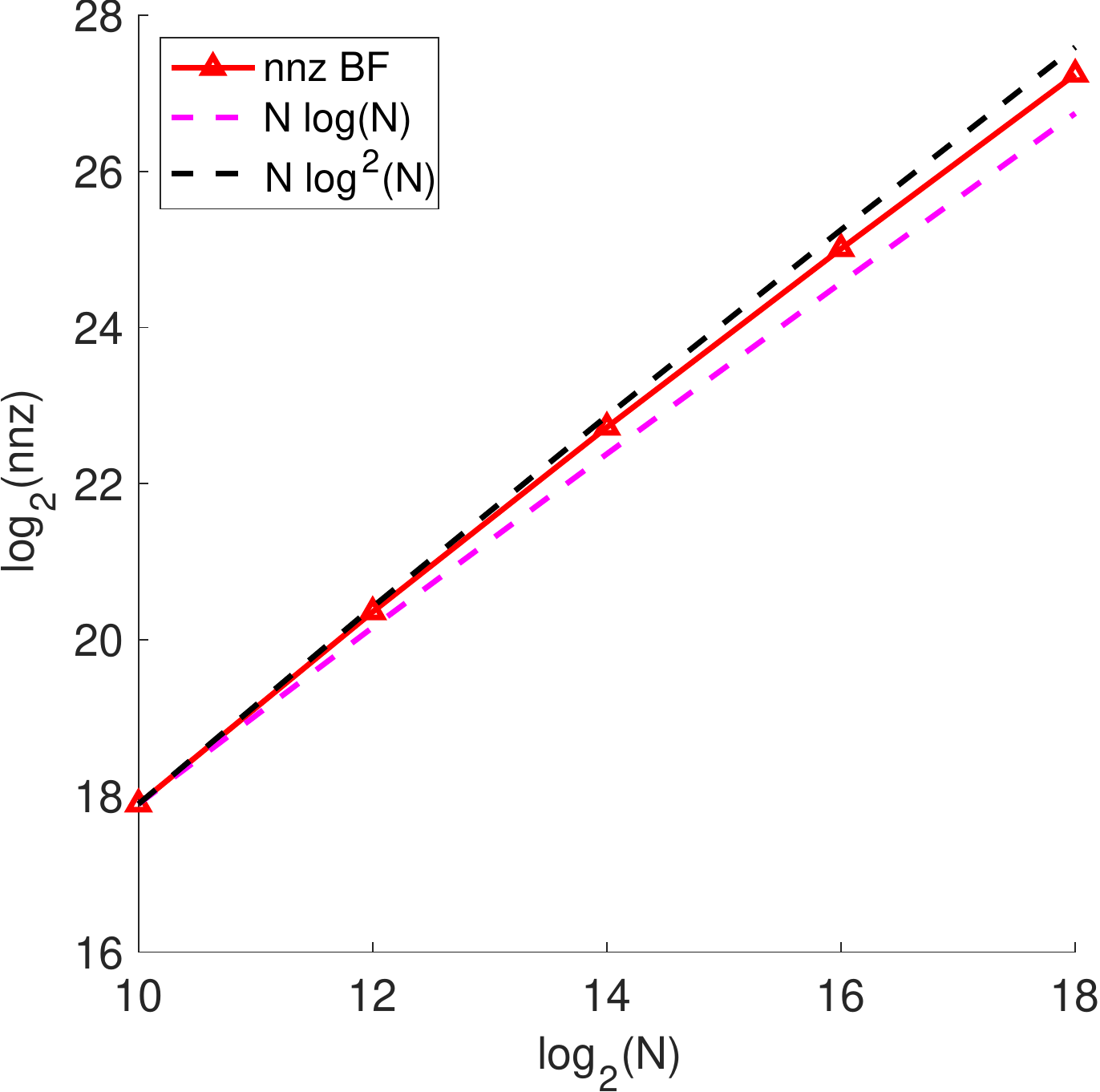}&
      \includegraphics[height=1.7in]{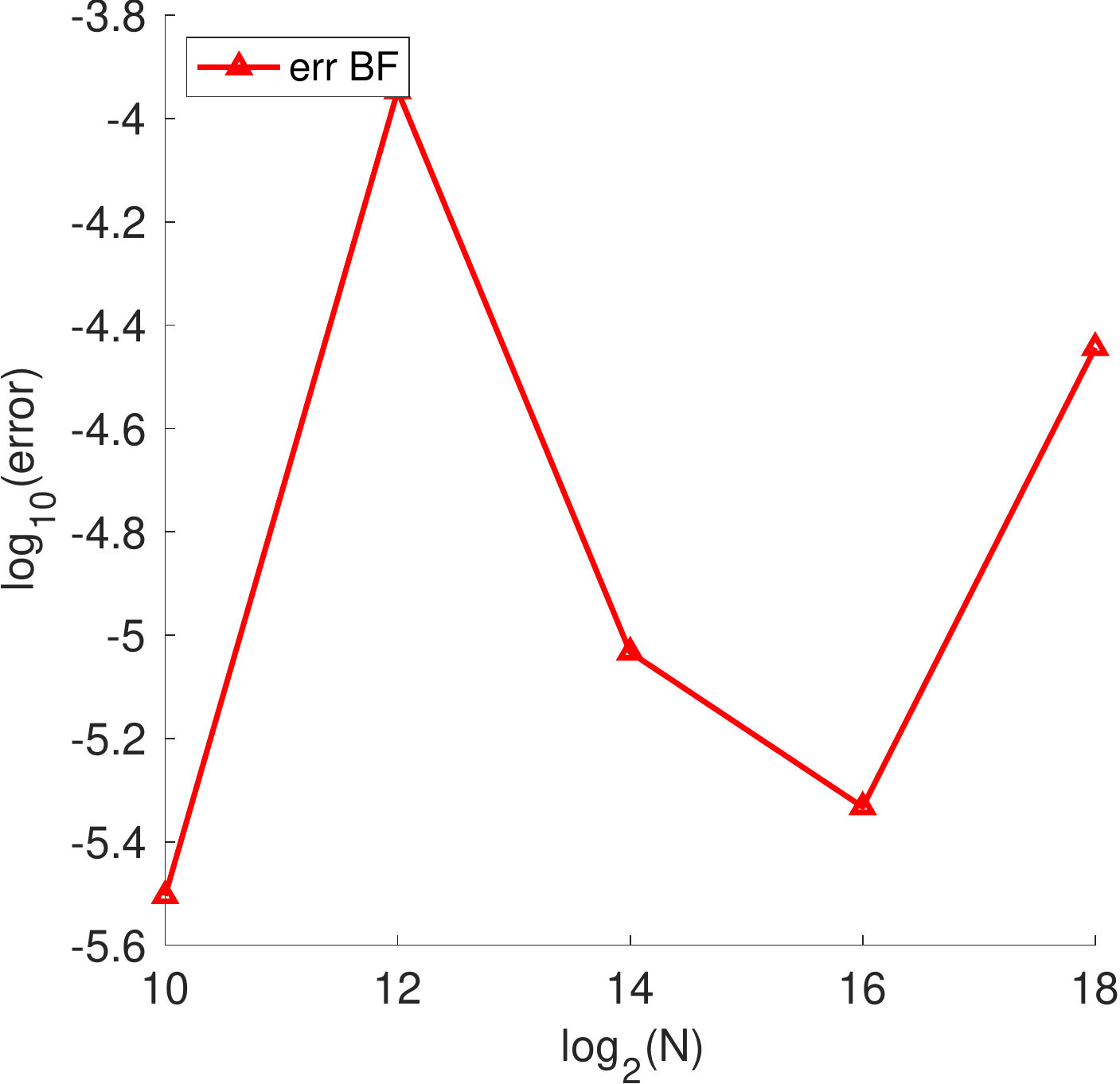}\\
      \includegraphics[height=1.7in]{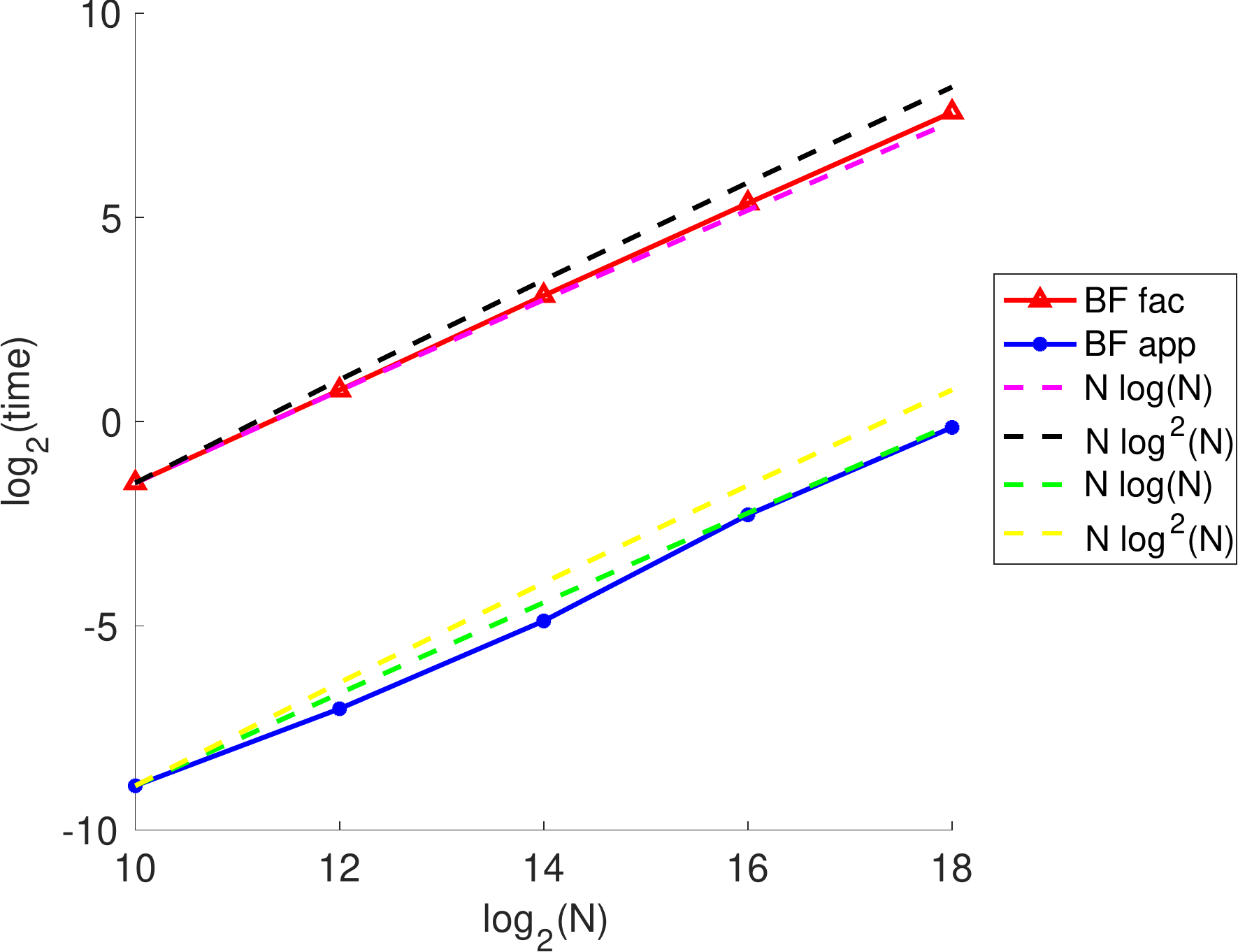}&
      \includegraphics[height=1.7in]{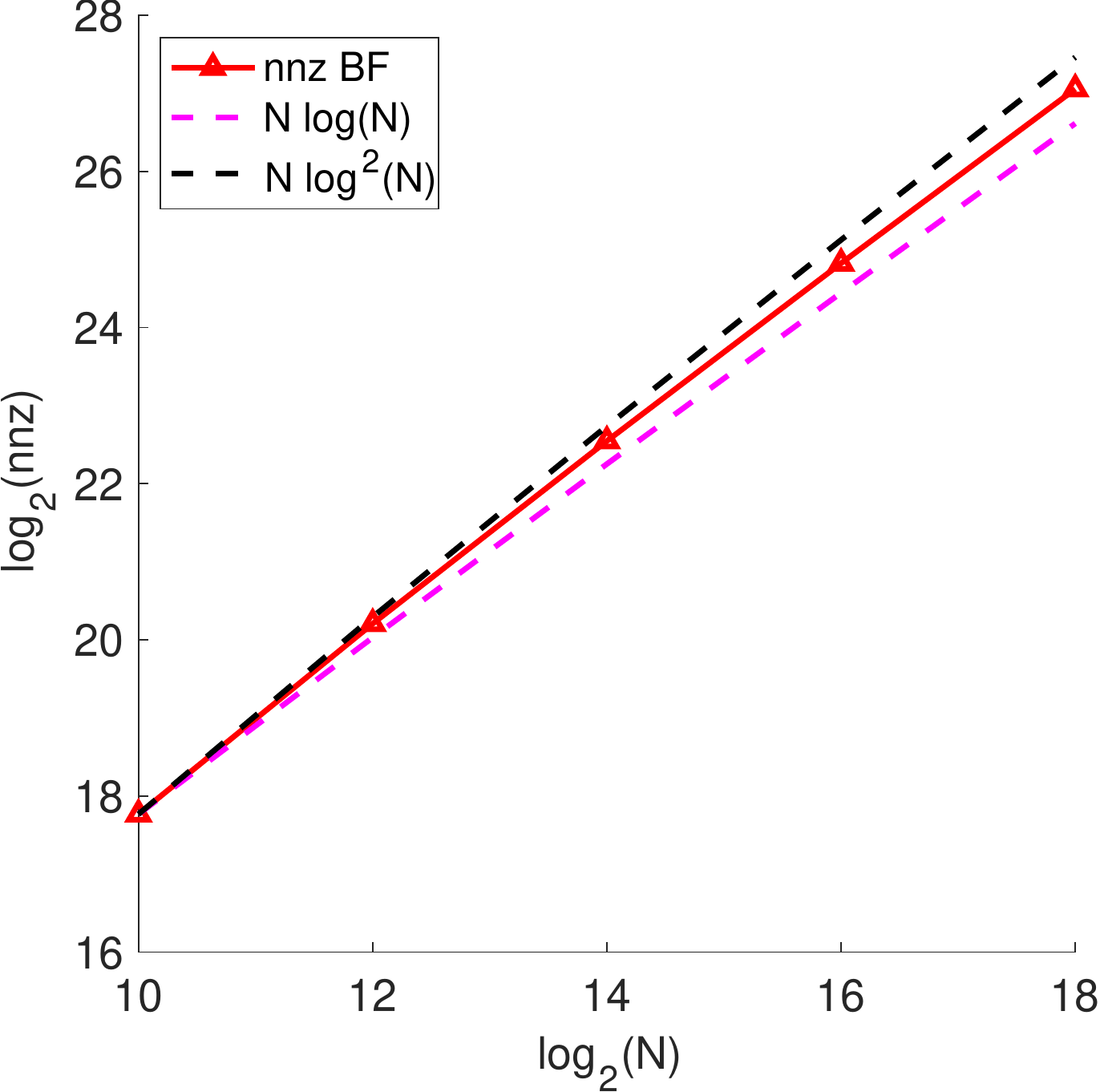}&
      \includegraphics[height=1.7in]{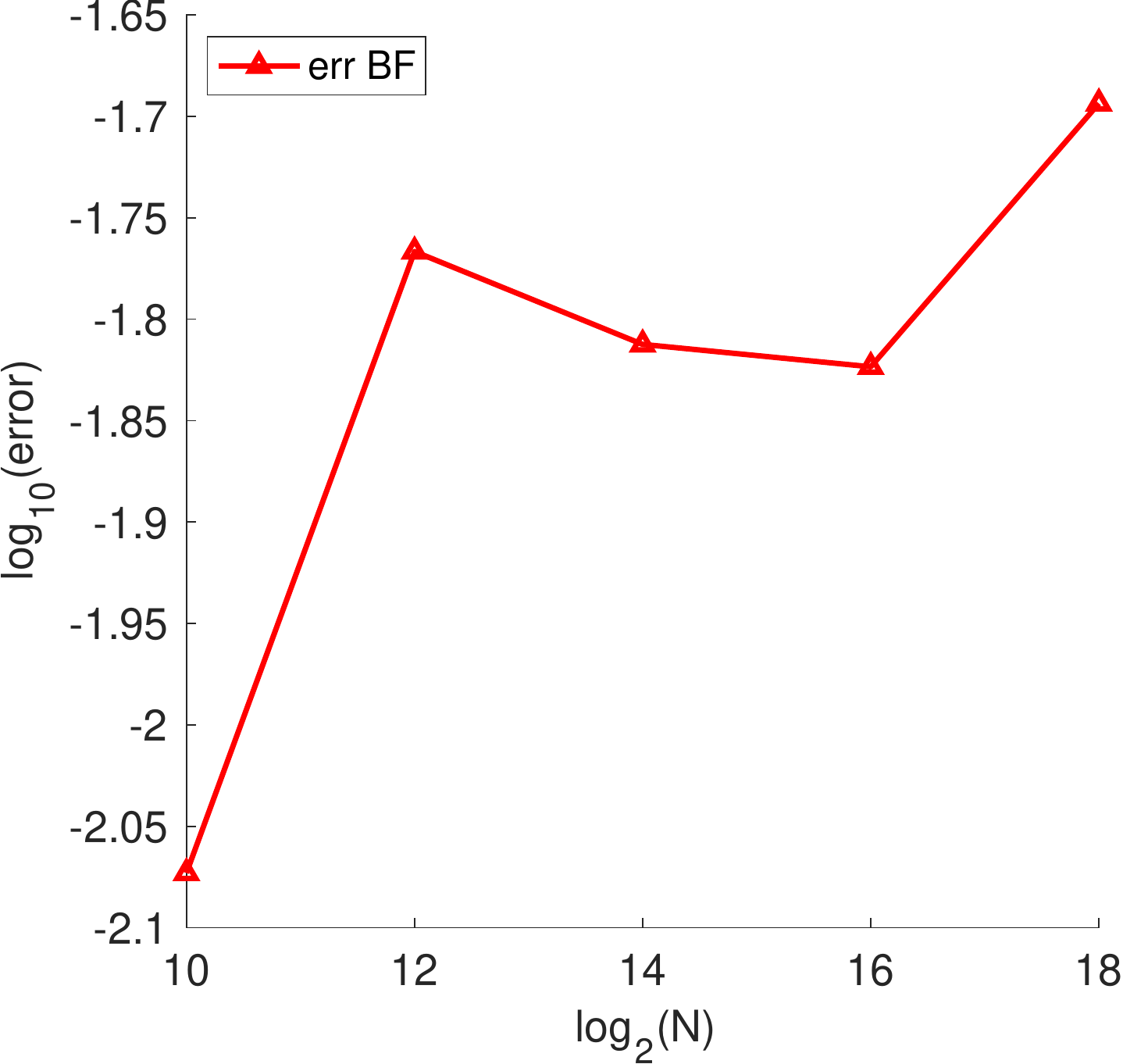}      
    \end{tabular}
  \end{center}
\caption{Numerical results for the FIO given in \eqref{eqn:1D-FIO}.
  $N$ is the size of the matrix; $nnz$ is the number of non-zero entries in the butterfly factorization, $err$ is the approximation error of the IDBF matvec. $\epsilon=10^{-6}$ and $k=30$.}
\label{fig:1D-fio3}
\end{figure}

\paragraph{Example 2.}
Next, we provide an example of a special function transform, the evaluation of Schl{\"o}milch expansions \cite{Sch} at $g_k=\frac{k-1}{N}$ for $1\leq k\leq N$:
\begin{equation}
\label{eqn:S}
u_k=\sum_{n=1}^N c_n J_\nu(g_k\omega_n),
\end{equation}
where $J_\nu$ is the Bessel function of the first kind with parameter $\nu=0$, and $\omega_n=n\pi$. It is demonstrated in
\cite{Butterfly2} that \eqref{eqn:S} can be represented via a matvec 
$u=Kg$, where $K$ satisfies the complementary low-rank property. An arbitrary entry of
$K$ can be calculated in $O(1)$ operations \cite{Bes} and hence IDBF is suitable for accelerating the matvec $u=Kg$. Other similar examples when $\nu\neq 0$ can be found in \cite{Sch} and they can be also evaluated by IDBF with the same operation counts. 

Figure \ref{fig:S} summarizes the results of this example for different problem sizes $N$ with different parameter pairs $(\epsilon,k)$. The results show that IDBF applied to this example has $O(N\log^2 (N))$ factorization and application time. The running time agrees with the scaling of the number of non-zero entries required in the data-sparse representation to guarantee the approximation accuracy.  In fact, when $N$ is large enough, the number of non-zero entries in the IDBF tends to scale as $O(N\log N)$, which means that the numerical scaling can approach to $O(N \log N)$ in both factorization and application when $N$ is large enough.

\begin{figure}[ht!]
  \begin{center}
    \begin{tabular}{ccc}
      \includegraphics[height=1.7in]{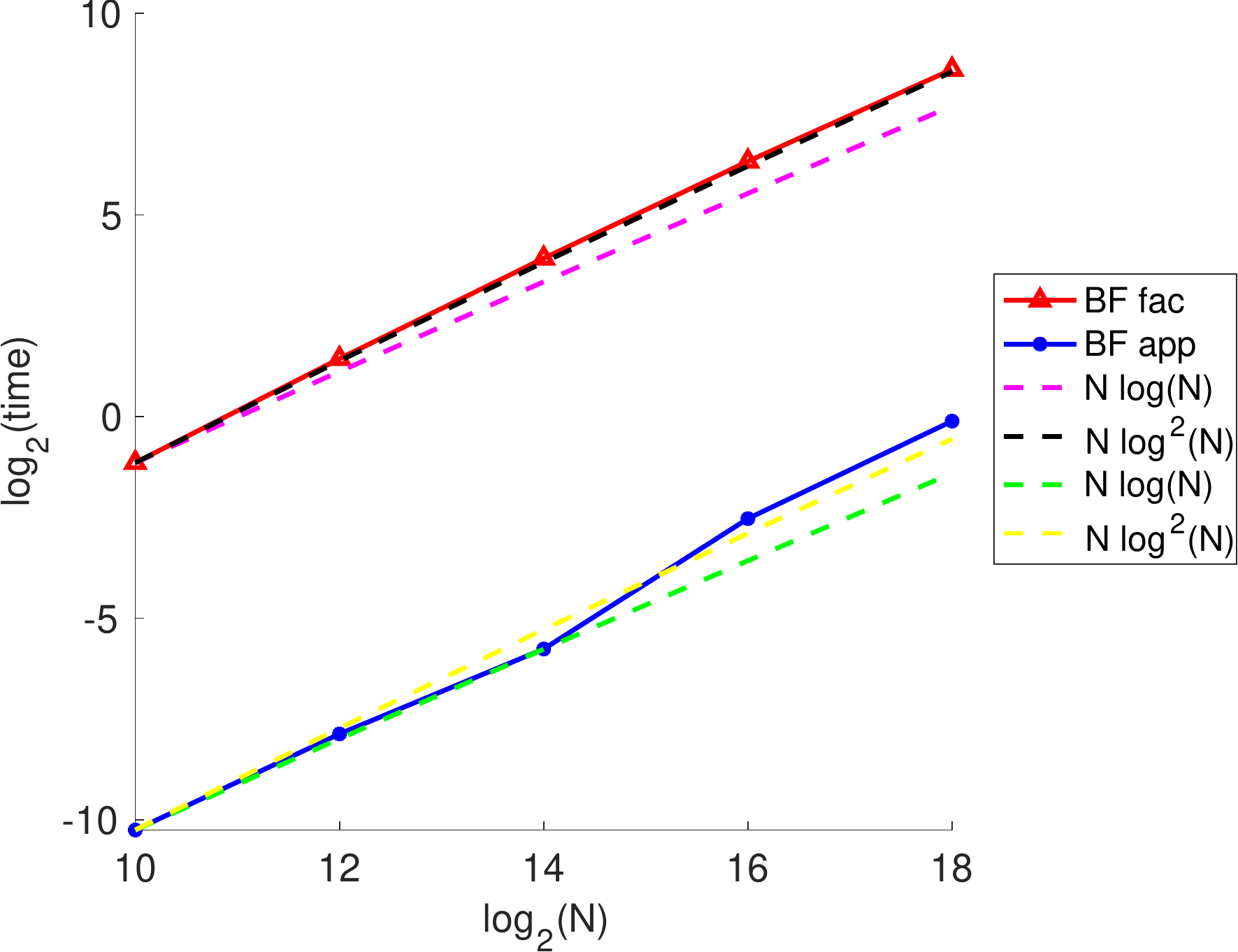}&
      \includegraphics[height=1.7in]{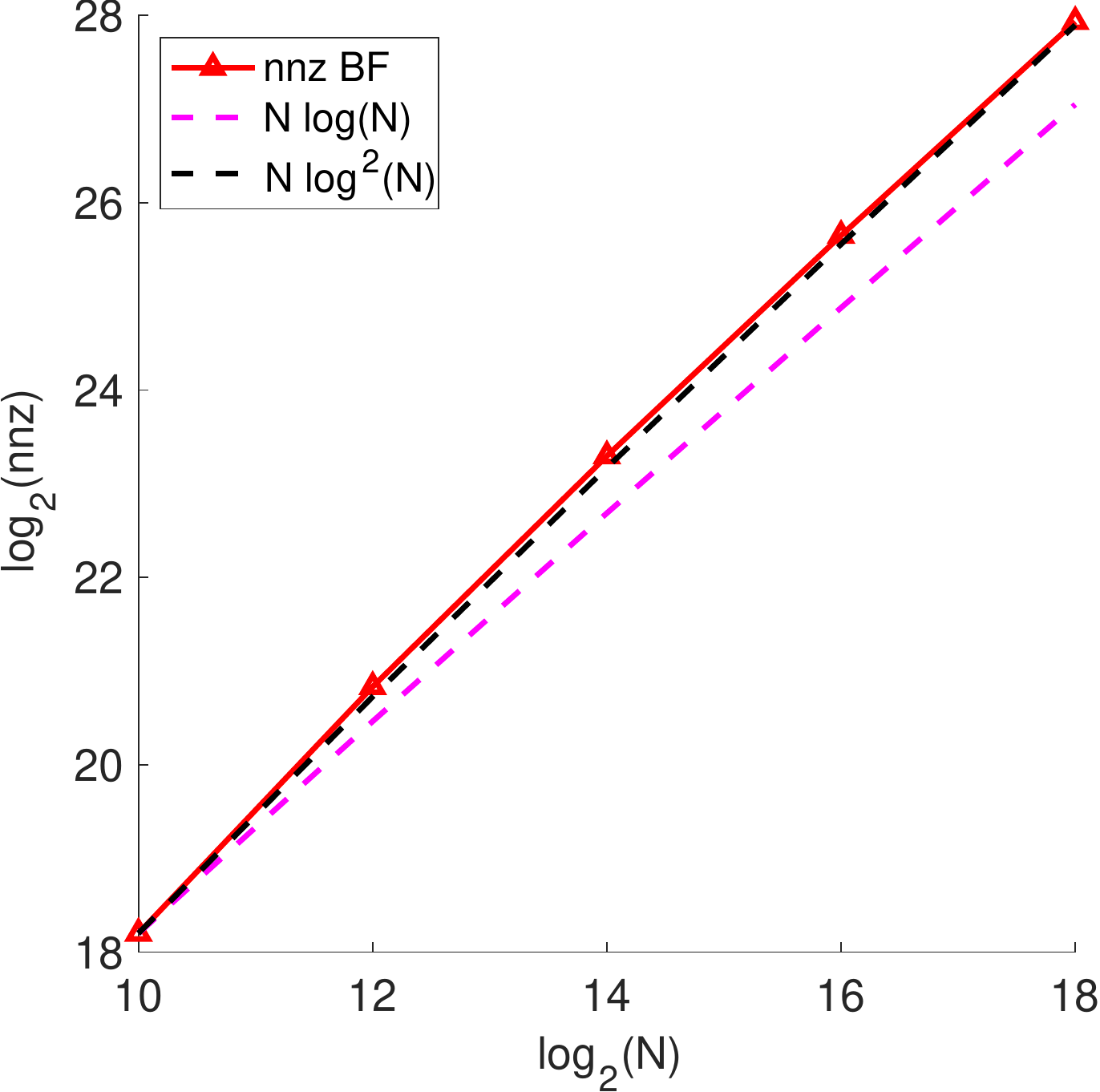}&
      \includegraphics[height=1.7in]{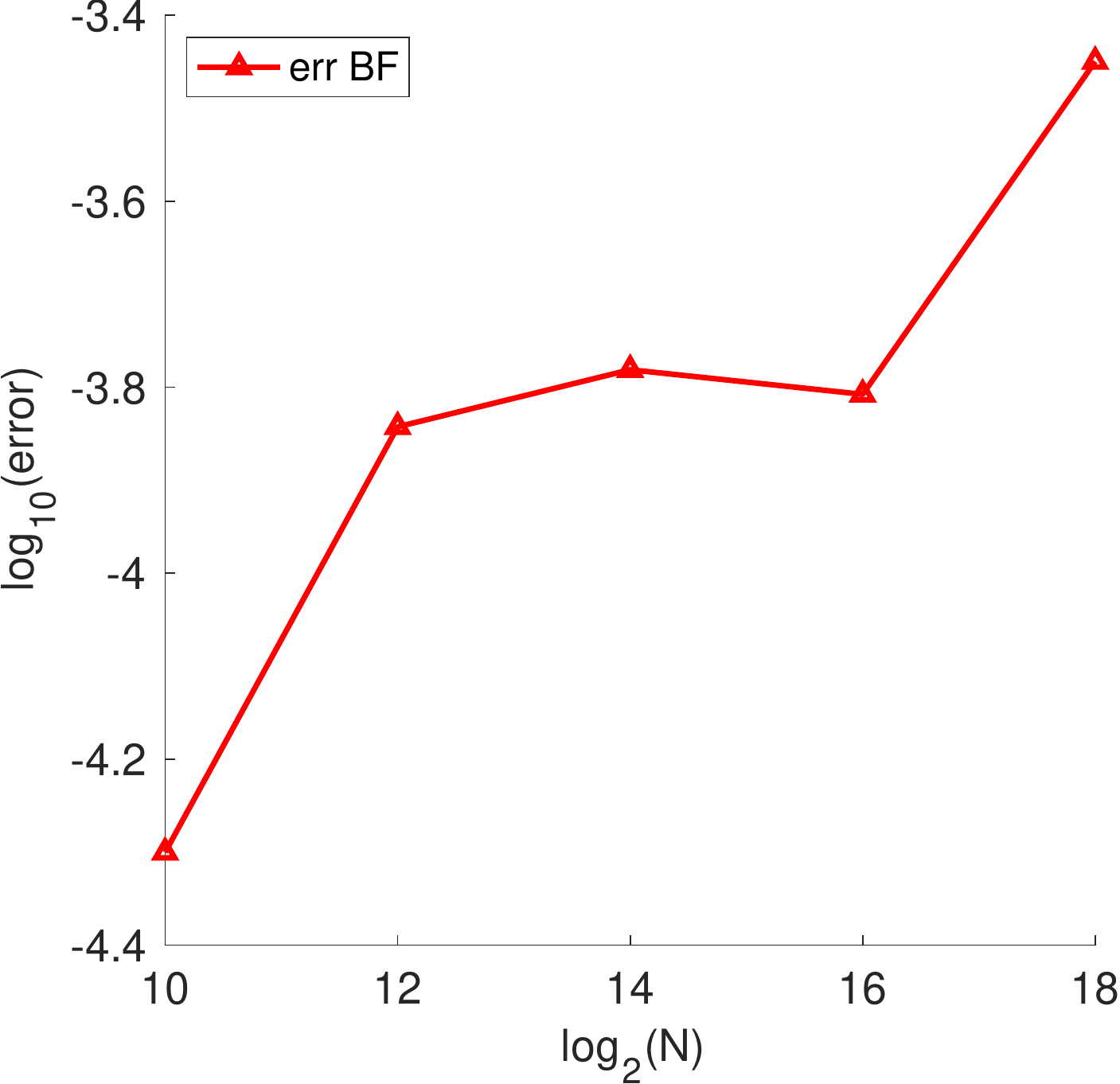}\\   
      \includegraphics[height=1.7in]{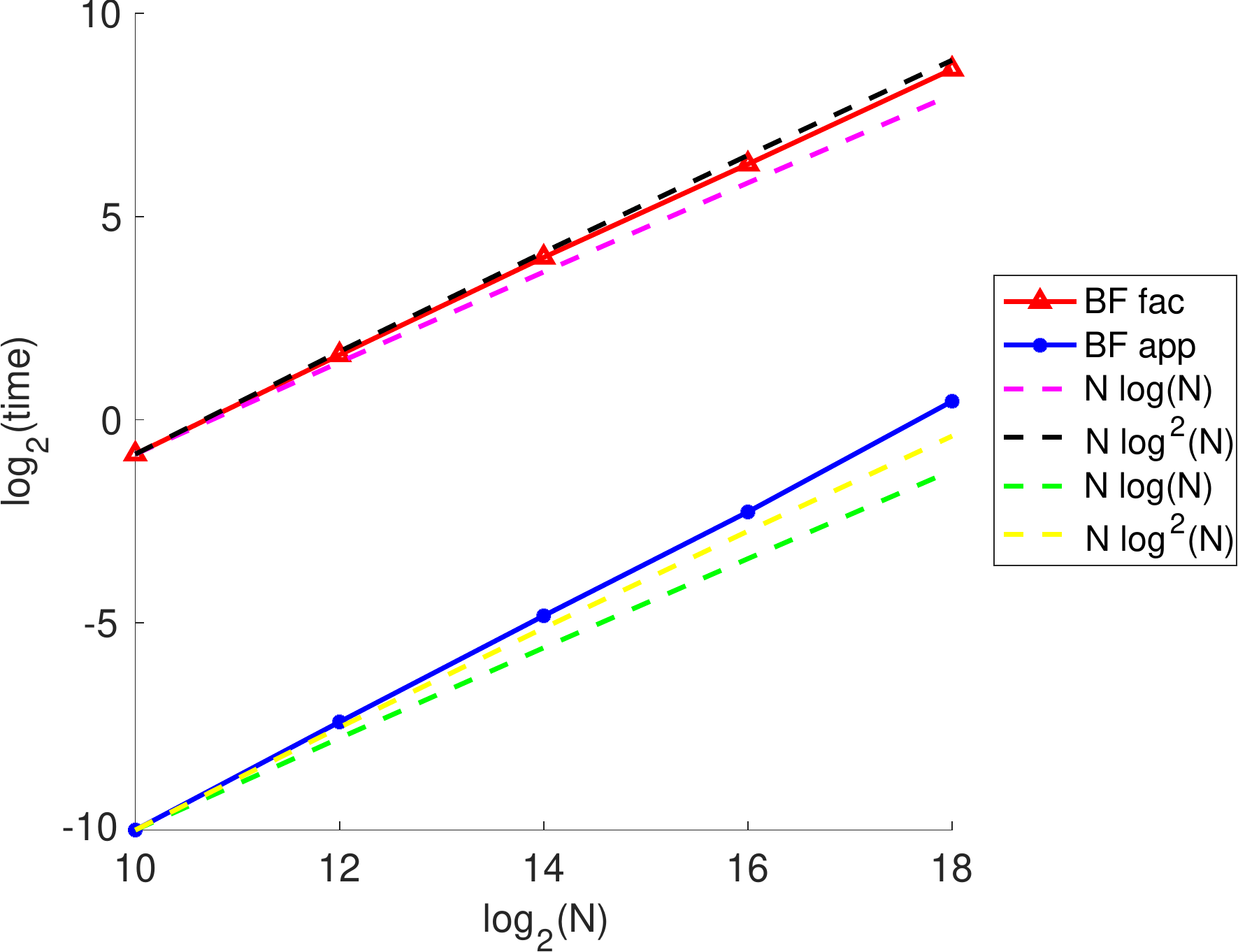}&
      \includegraphics[height=1.7in]{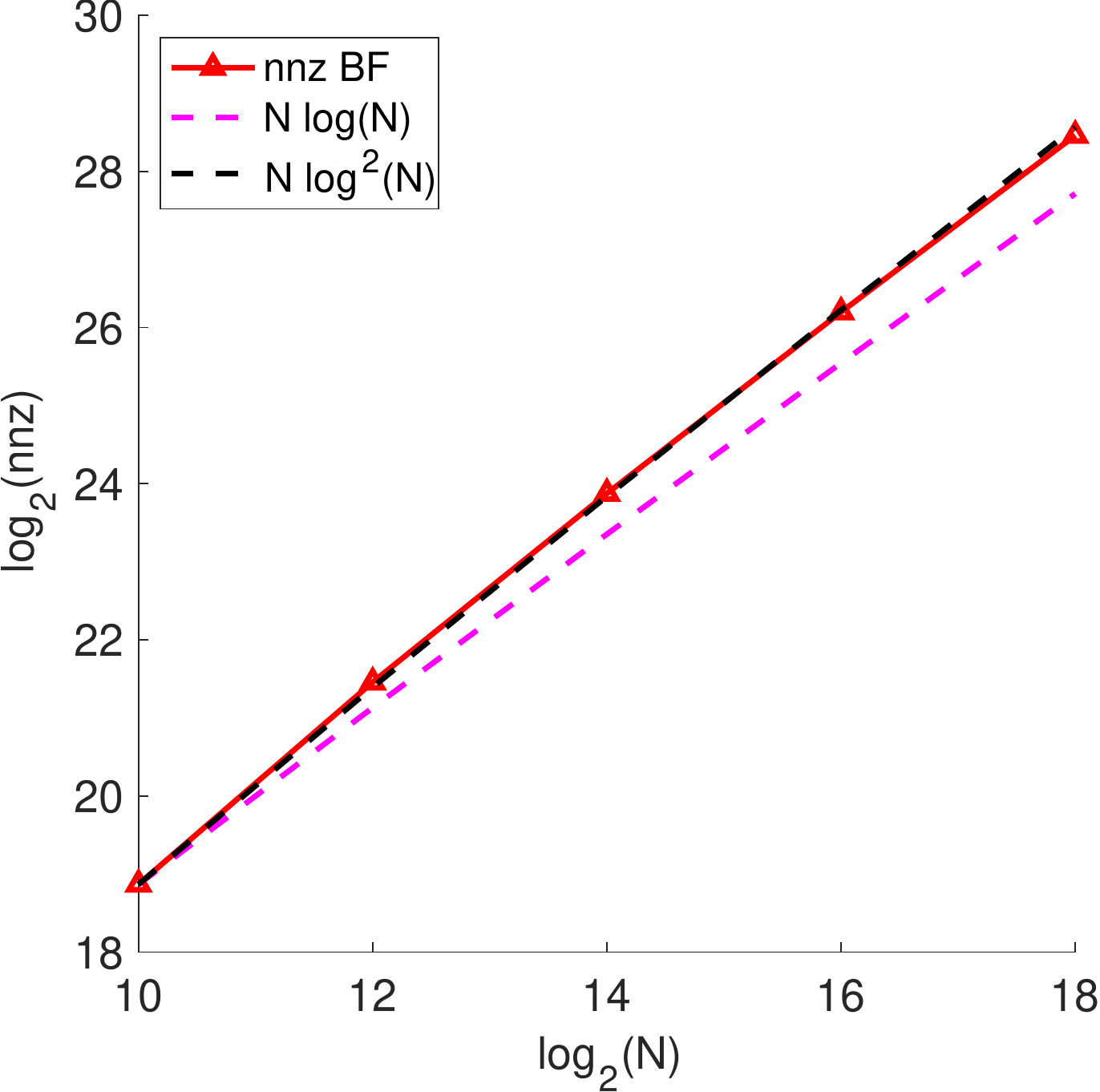}&
      \includegraphics[height=1.7in]{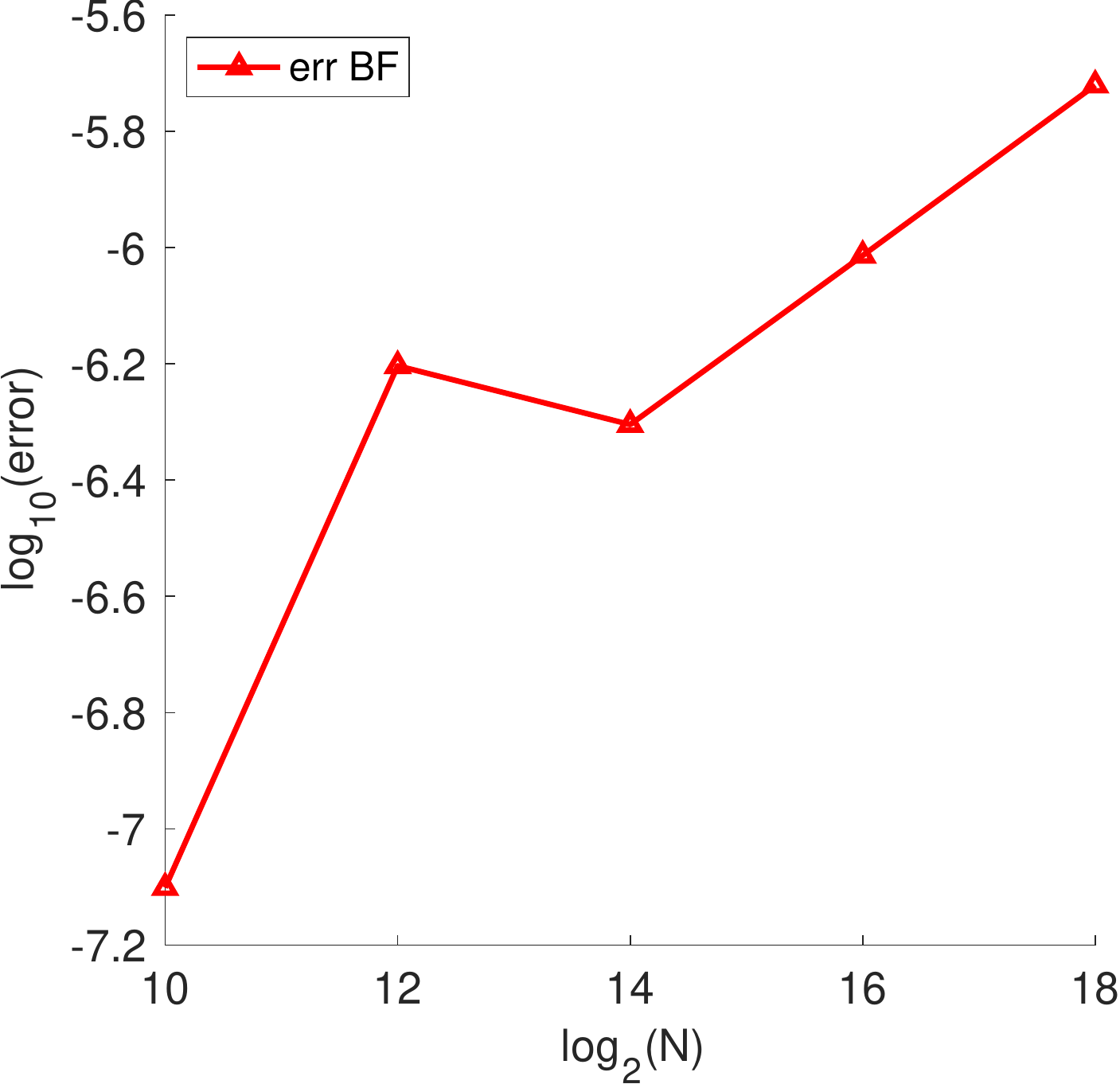}\\   
      \includegraphics[height=1.7in]{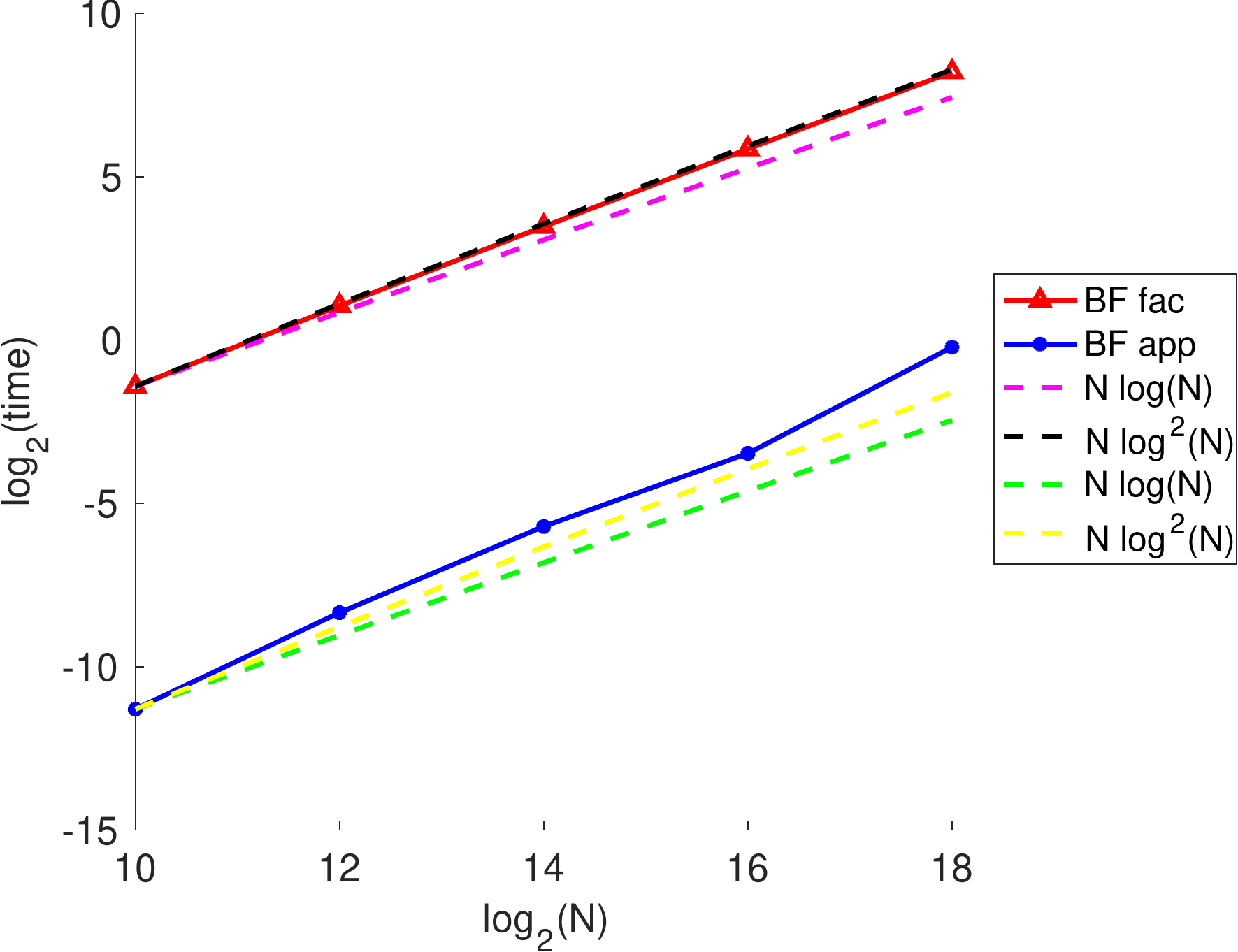}&
      \includegraphics[height=1.7in]{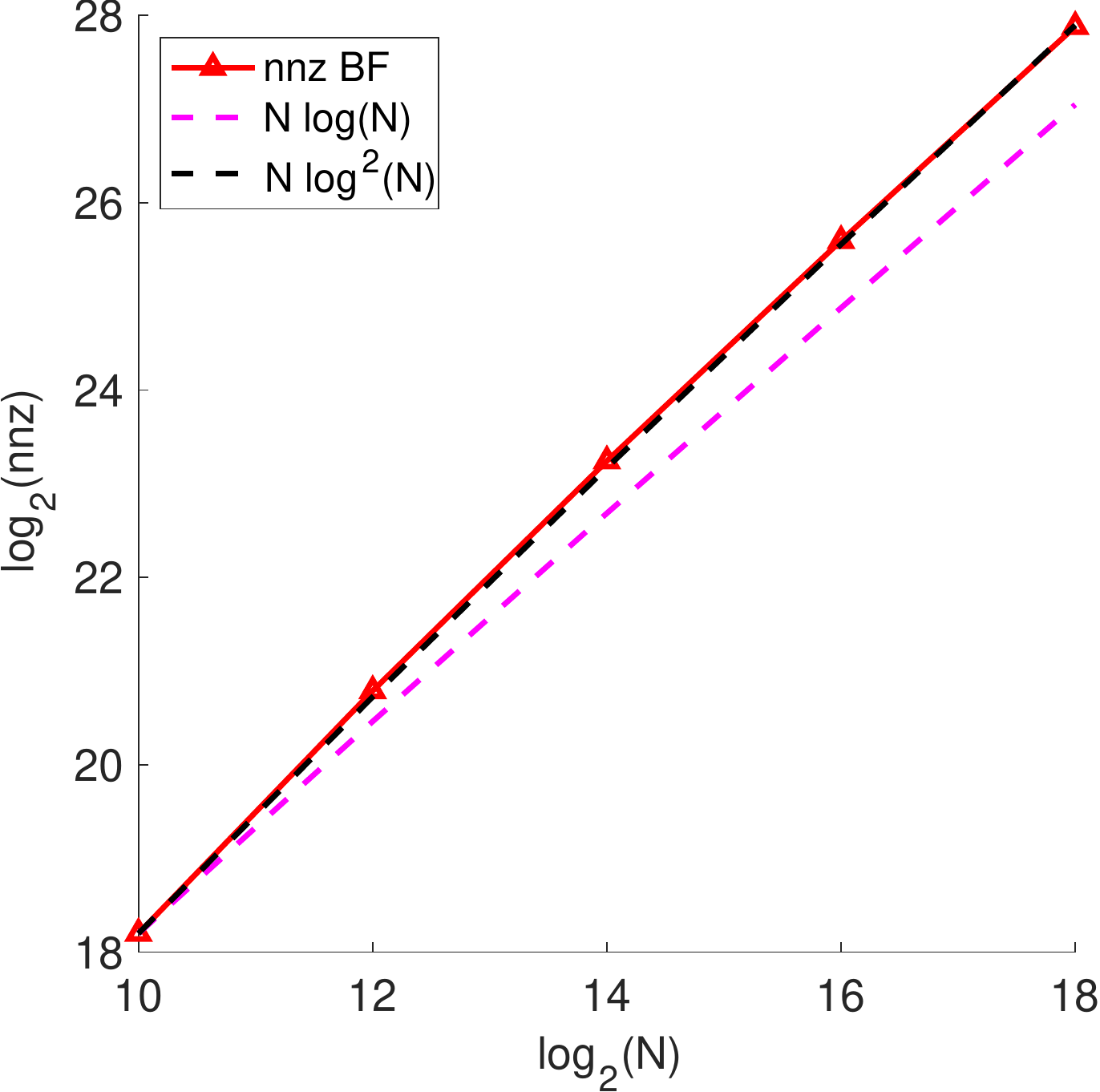}&
      \includegraphics[height=1.7in]{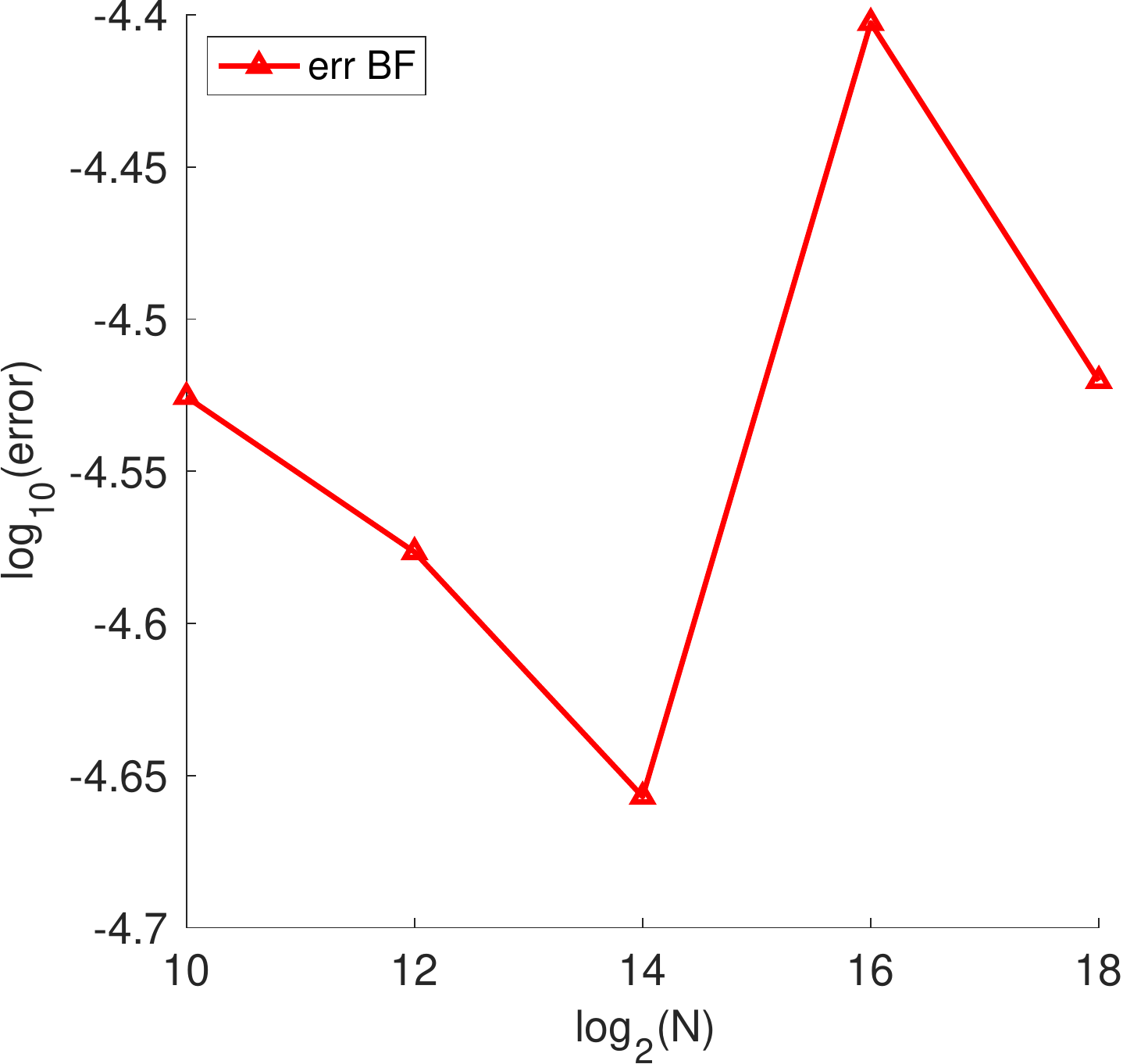}\\   
    \end{tabular}
  \end{center}
\caption{Numerical results for the  Schl{\"o}milch expansions given in \eqref{eqn:S}.
  $N$ is the size of the matrix; $nnz$ is the number of non-zero entries in the butterfly factorization, $err$ is the approximation error of the IDBF matvec. Top row: $(\epsilon,k)=(10^{-6},\min\{m,n\})$. Middle row: $(\epsilon,k)=(10^{-15},30)$. Bottom row: $(\epsilon,k)=(10^{-6},30)$. }
\label{fig:S}
\end{figure}

\paragraph{Example 3.}
In this example, we consider the {1D} non-uniform Fourier transform as follows:
\begin{equation}\label{eqn:N}
u_k = \sum_{n=1}^N e^{-2\pi \imath x_n \omega_k }g_n,
\end{equation}
for $1\leq k\leq N$, where $x_n$ is randomly selected in $[0,1)$, and $\omega_k$ is randomly selected in $[-\frac{N}{2},\frac{N}{2})$ according to uniform distributions in these intervals. 

Figure~\ref{fig:N} summarizes the results of
this example for different grid sizes $N$ with different parameter pairs $(\epsilon,k)$. Numerical results show that IDBF admits at most $O(N\log^2 (N))$ factorization and application time for the non-uniform Fourier transform. The running time agrees with the scaling of the number of non-zero entries required in the data-sparse representation. In fact, when $N$ is large enough, the number of non-zero entries in the IDBF tends to scale as $O(N\log N)$, which means that the numerical scaling can approach to $O(N \log N)$ in both factorization and application when $N$ is large enough.

\begin{figure}[ht!]
  \begin{center}
    \begin{tabular}{ccc}
      \includegraphics[height=1.7in]{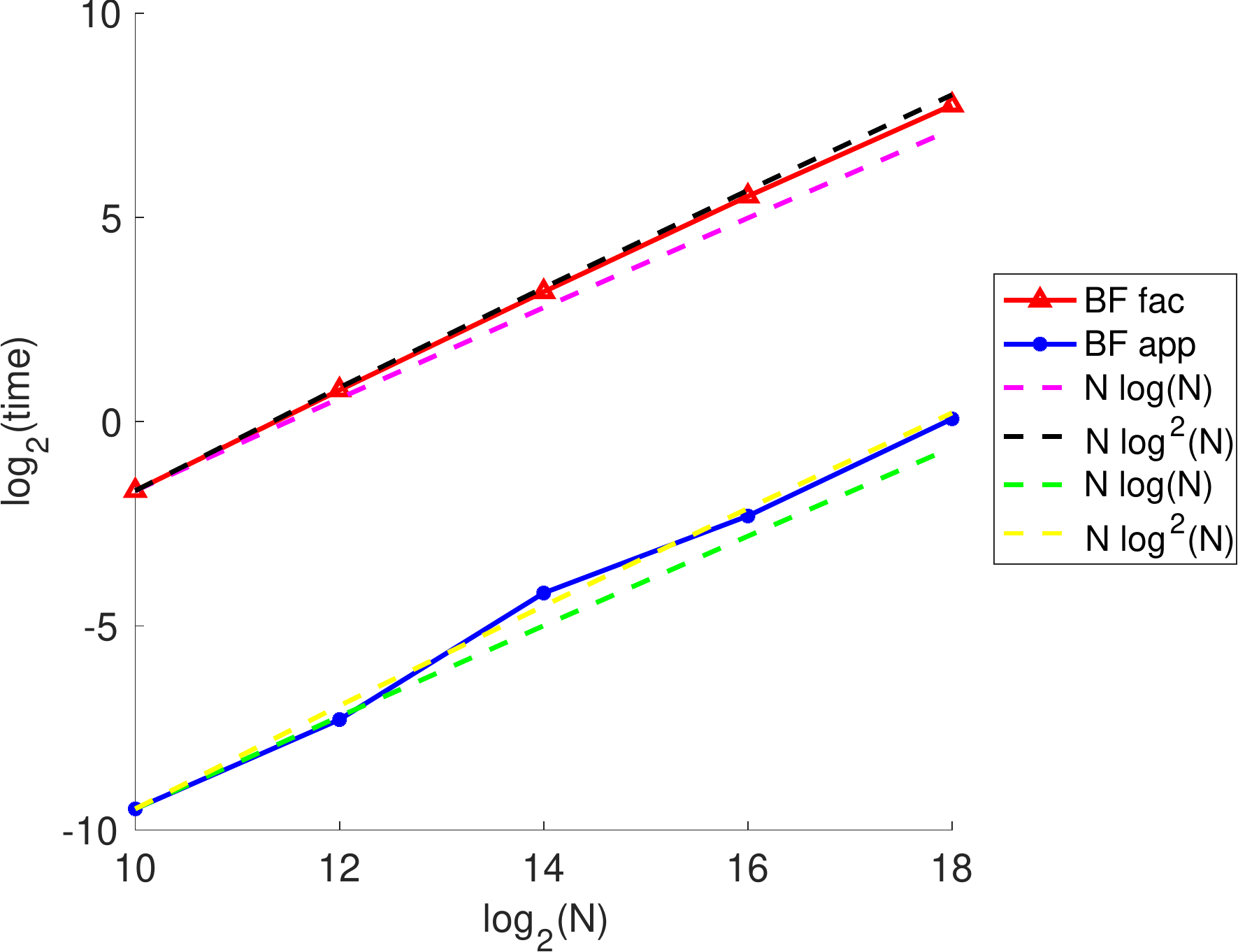}&
      \includegraphics[height=1.7in]{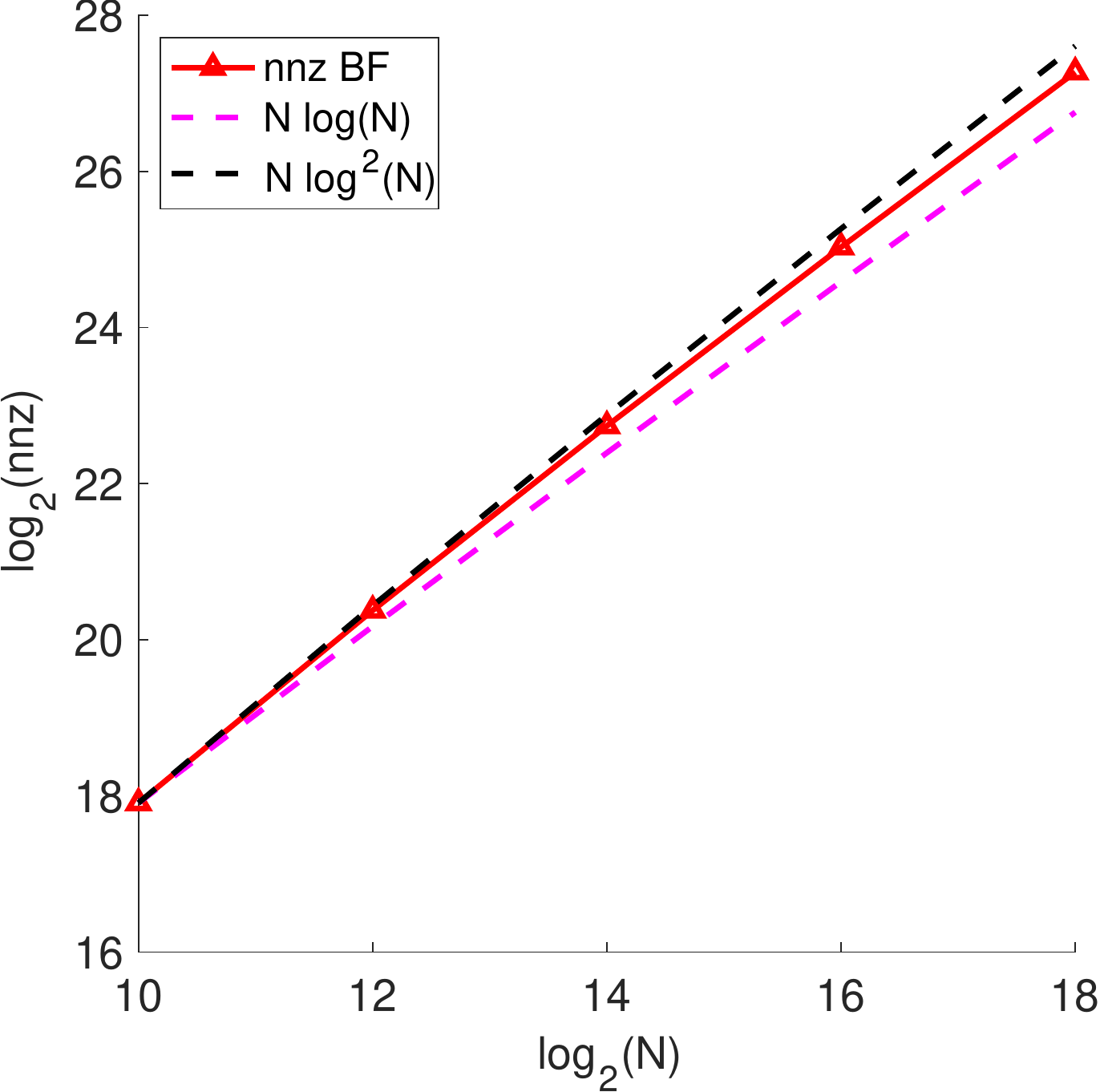}&
      \includegraphics[height=1.7in]{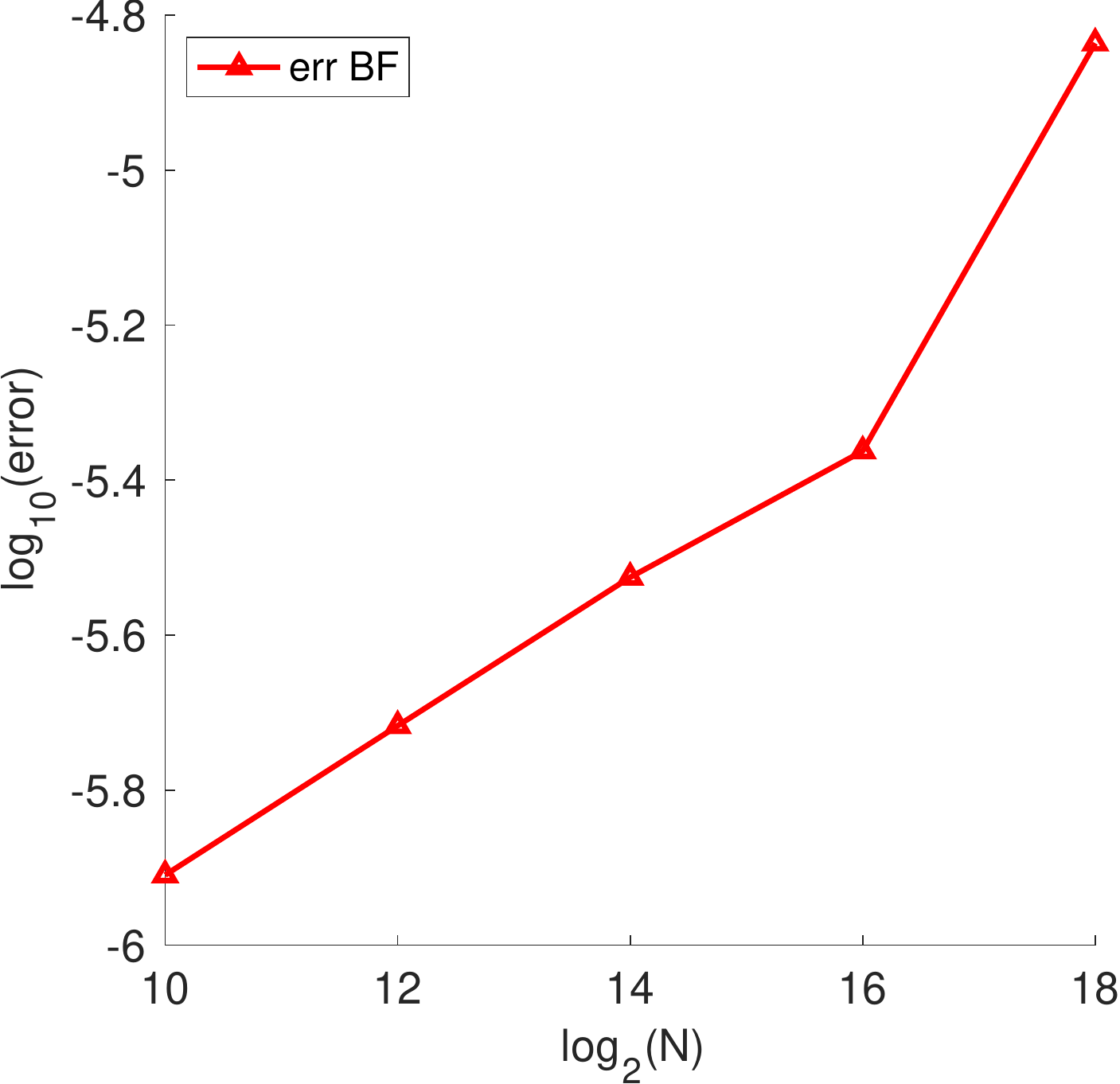}\\   
      \includegraphics[height=1.7in]{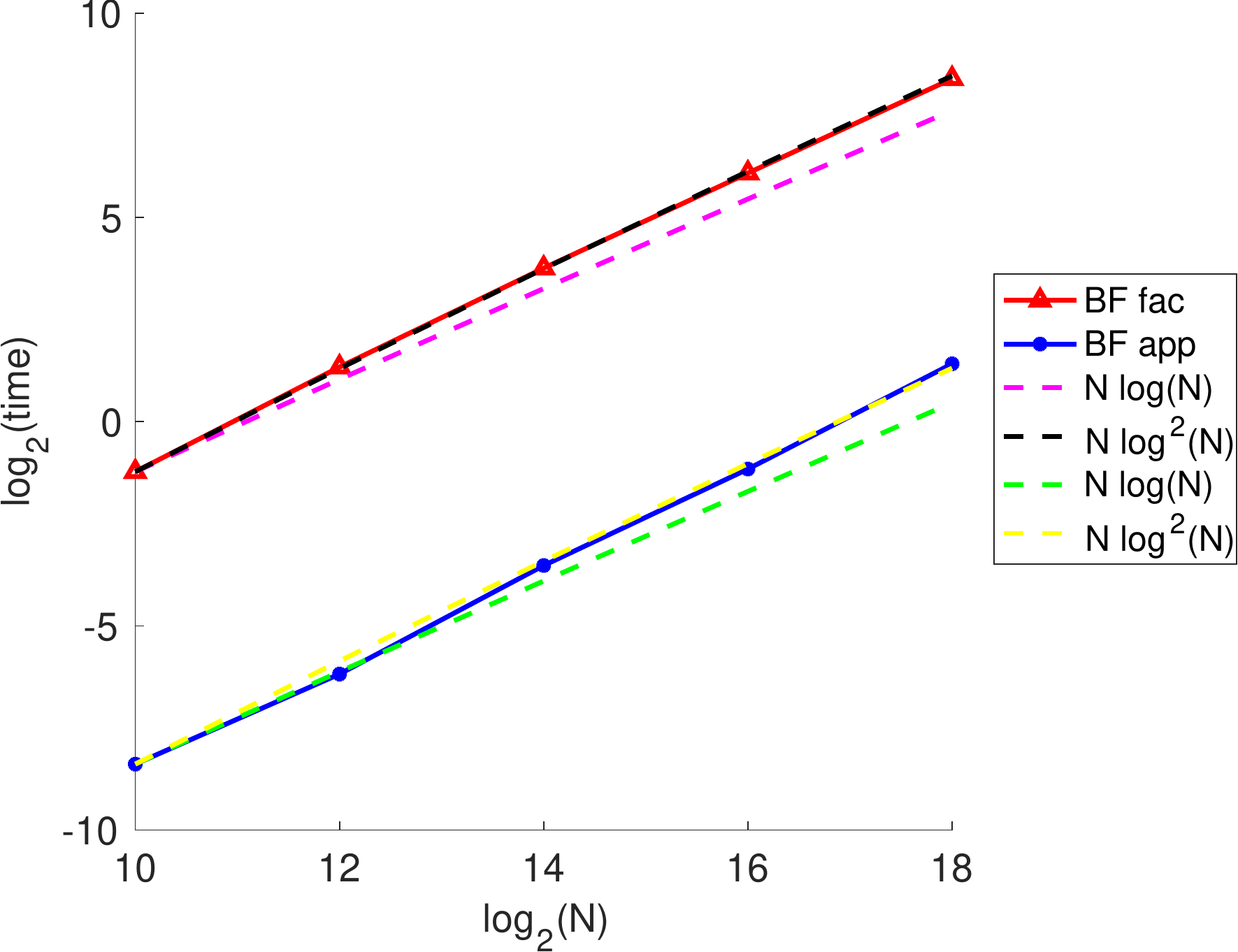}&
      \includegraphics[height=1.7in]{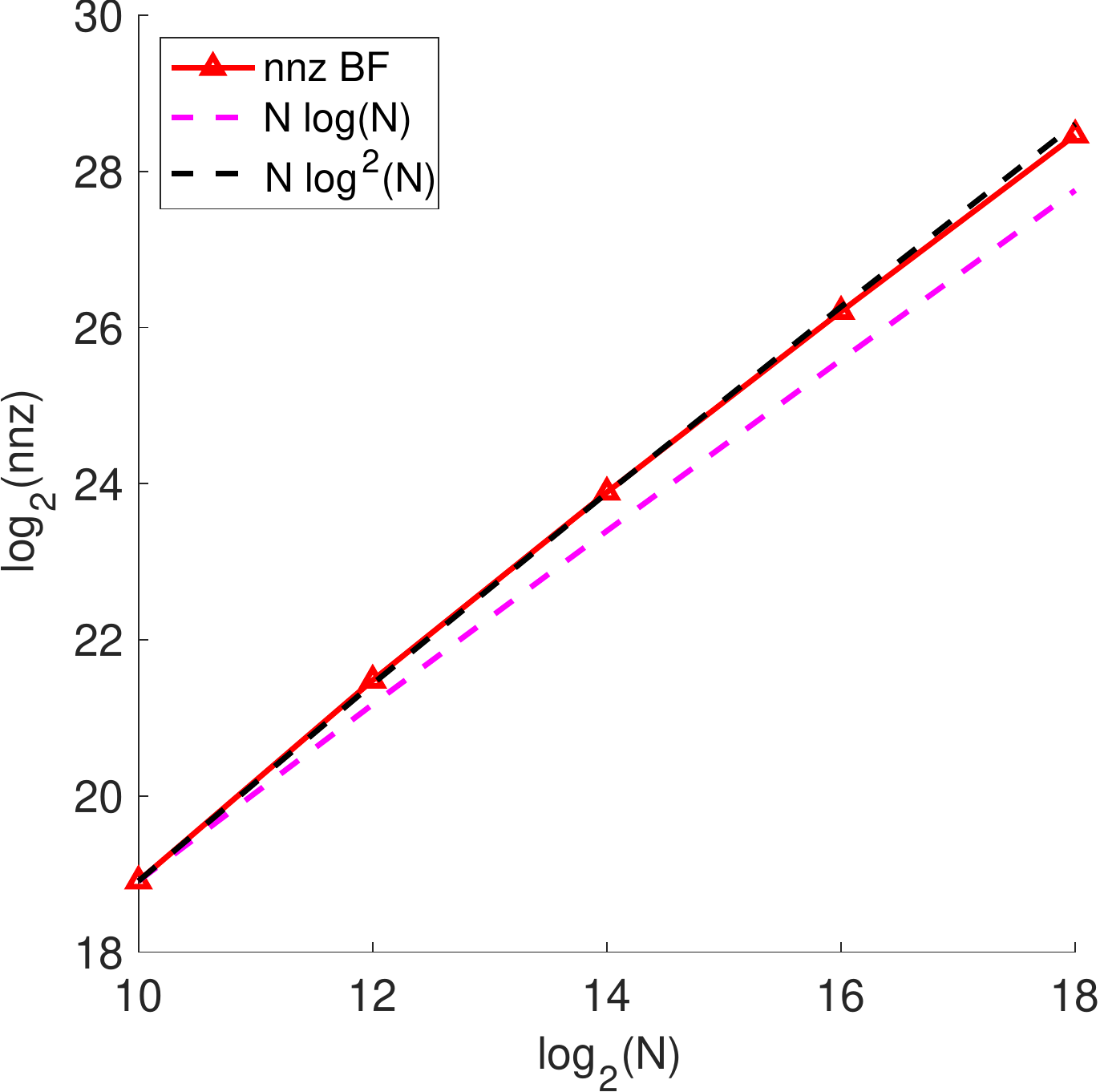}&
      \includegraphics[height=1.7in]{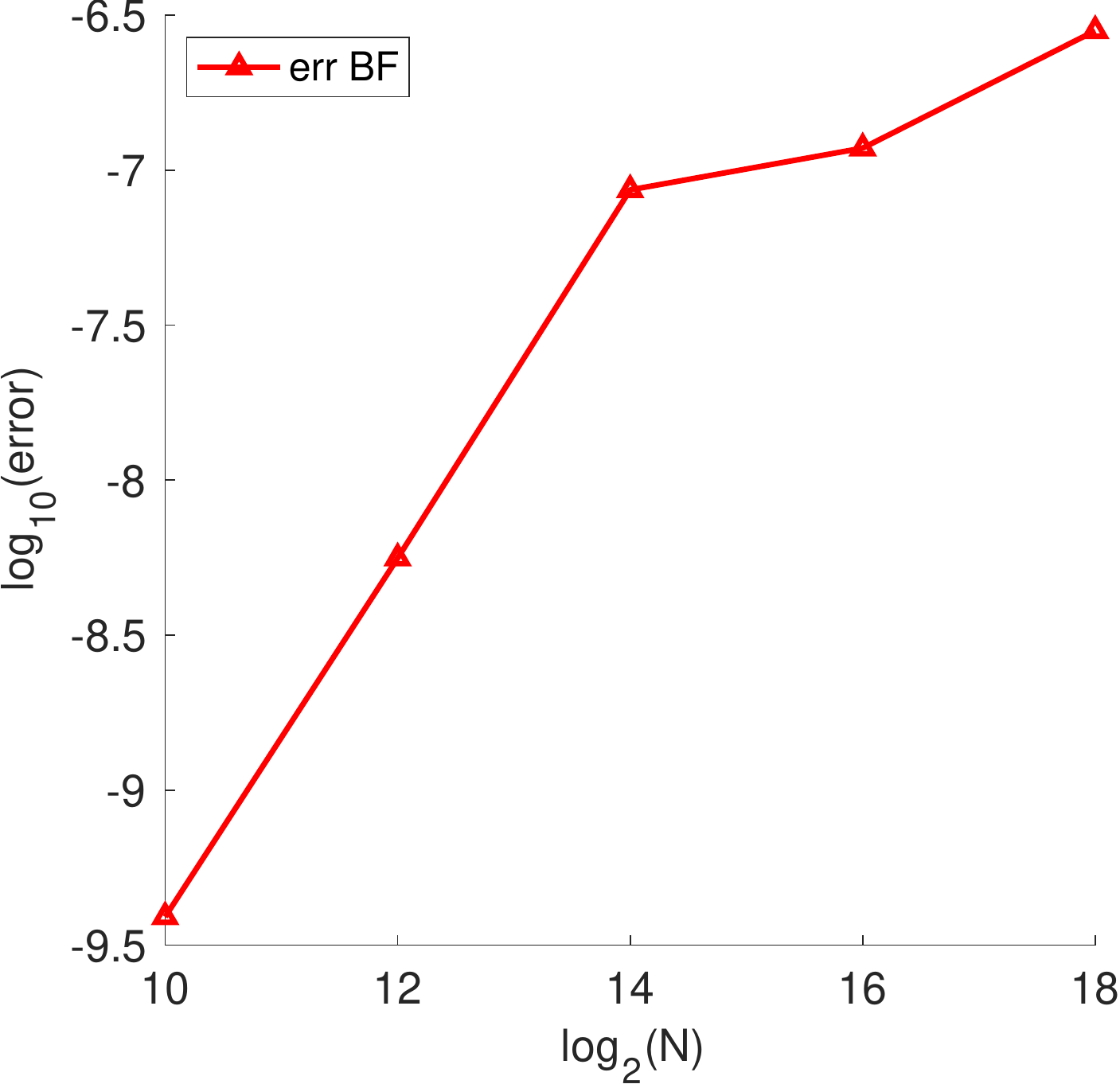}\\   
      \includegraphics[height=1.7in]{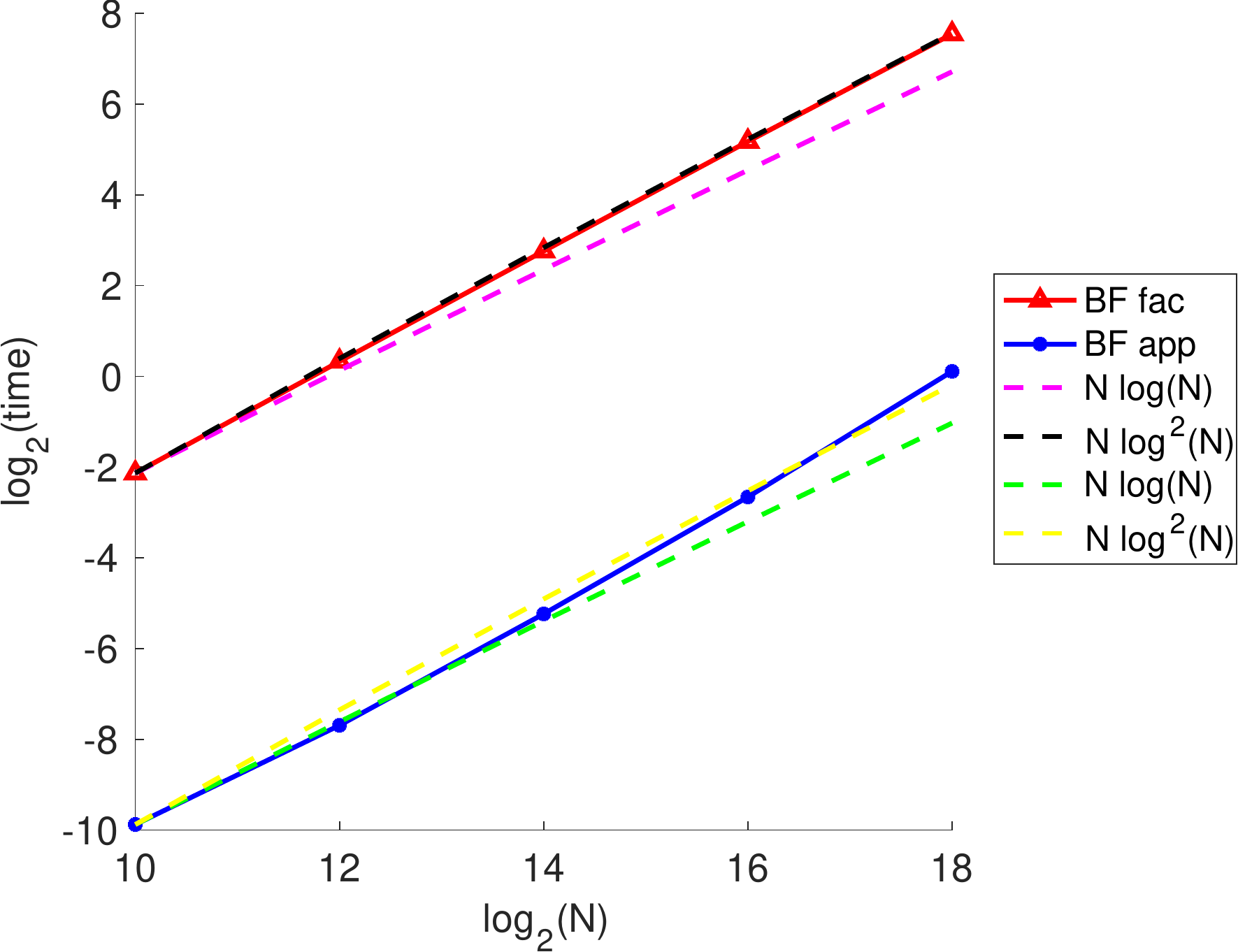}&
      \includegraphics[height=1.7in]{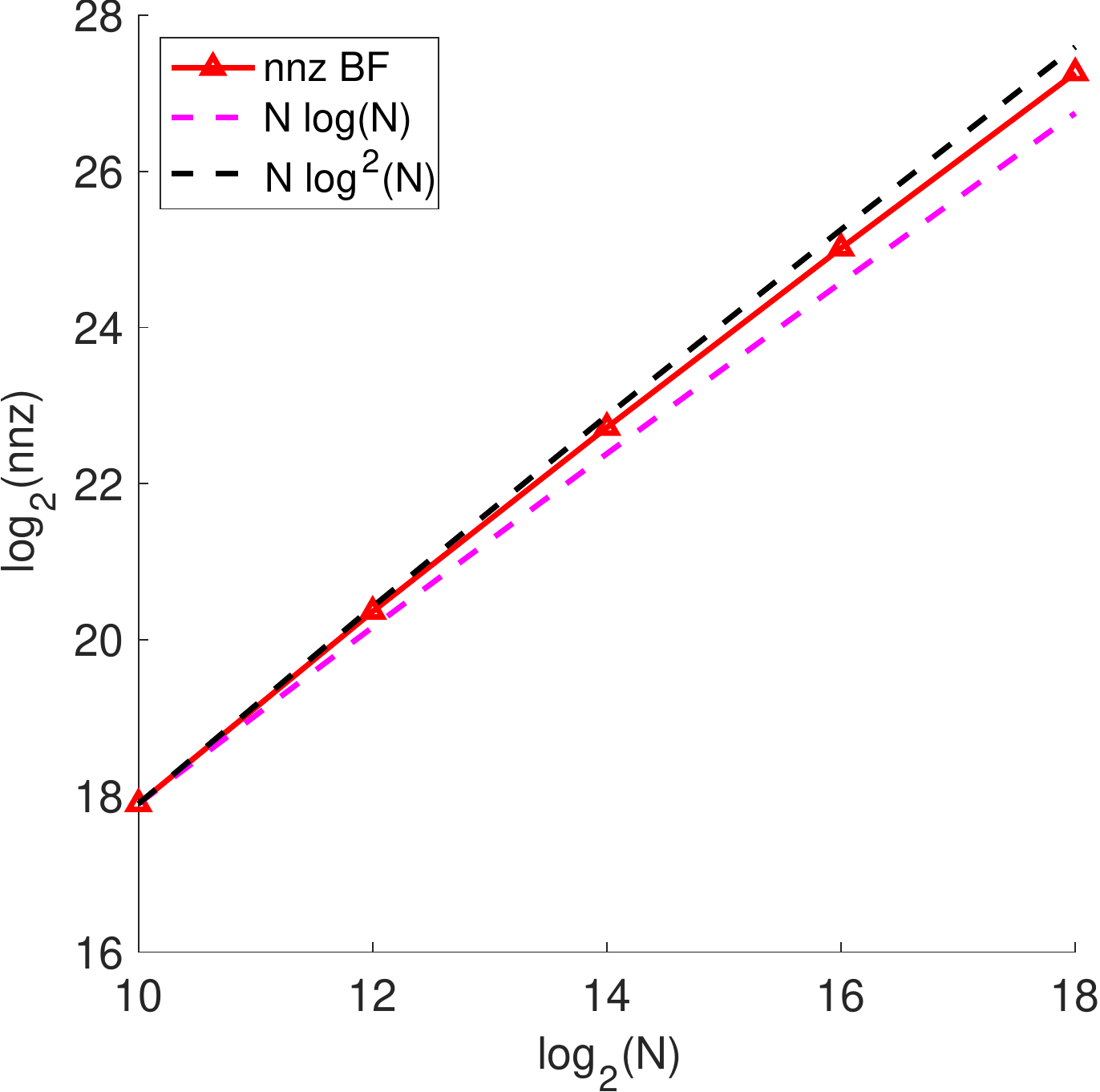}&
      \includegraphics[height=1.7in]{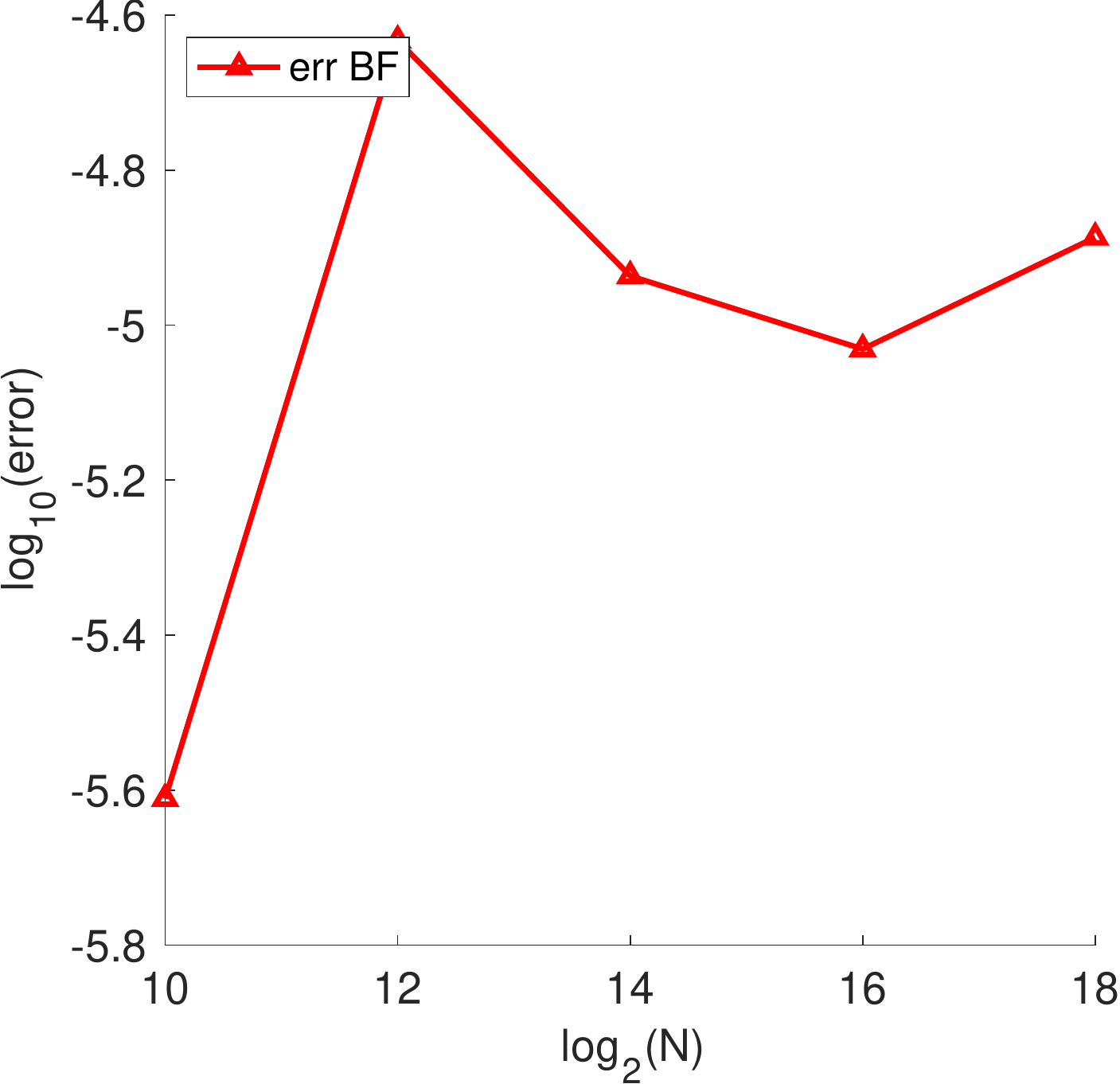}\\    
    \end{tabular}
  \end{center}
\caption{Numerical results for the NUFFT given in \eqref{eqn:N}.
  $N$ is the size of the matrix; $nnz$ is the number of non-zero entries in the butterfly factorization, $err$ is the approximation error of the IDBF matvec. Top row: $(\epsilon,k)=(10^{-6},\min\{m,n\})$. Middle row: $(\epsilon,k)=(10^{-15},30)$. Bottom row: $(\epsilon,k)=(10^{-6},30)$. }
\label{fig:N}
\end{figure}

\begin{figure}[ht!]
  \begin{center}
    \begin{tabular}{cc}
      \includegraphics[height=1.7in]{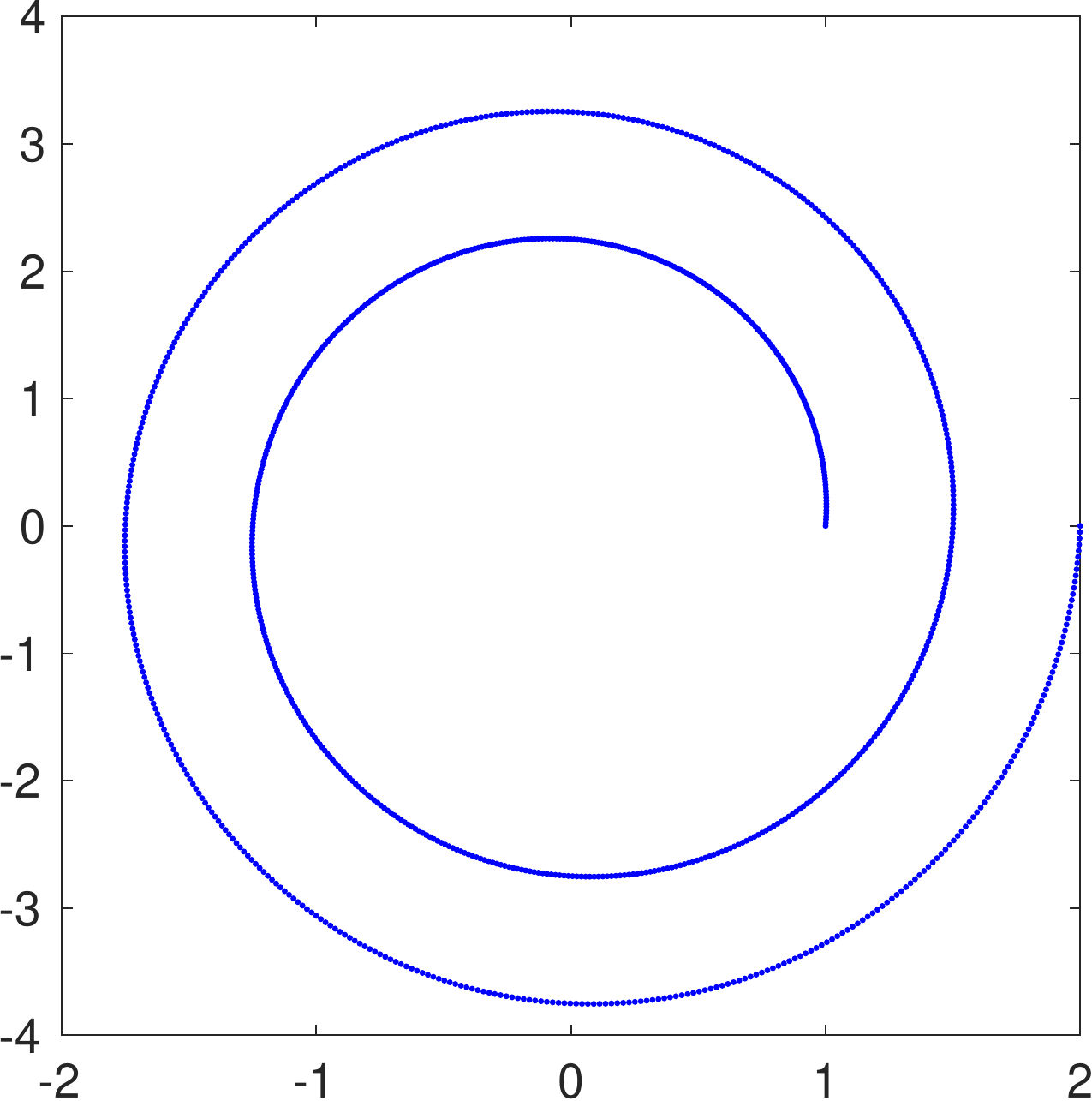}&
      \includegraphics[height=1.7in]{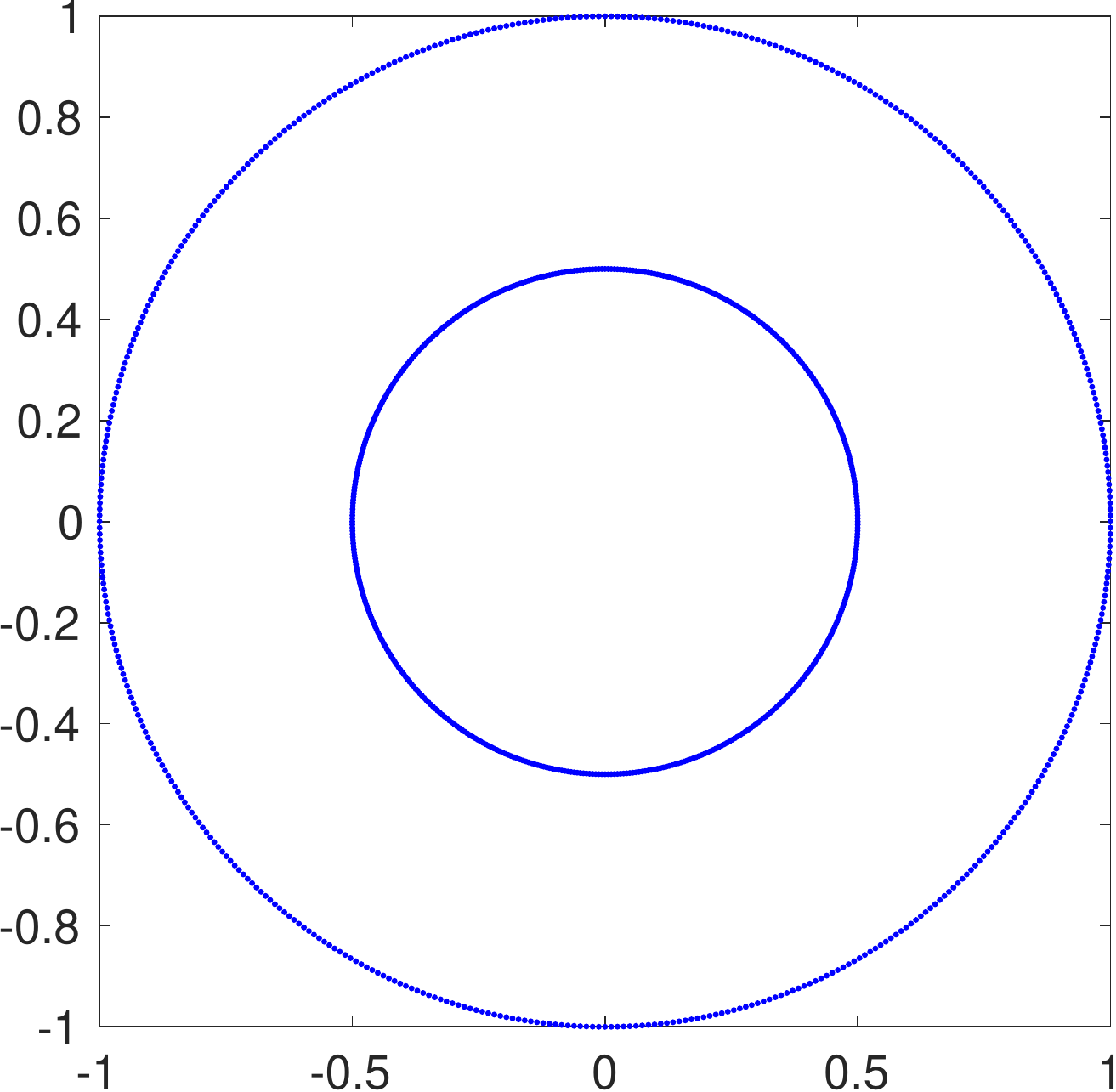}\\
      (a) & (b)
    \end{tabular}
  \end{center}
\caption{The two scatterers used in Example $4$ and $5$: (a) a spiral object; (b) a round object with a hole in center which is the port.}
\label{fig:p}
\end{figure}

\paragraph{Example 4.}
The fourth example is from {the electric field integral equation (EFIE) for analyzing scattering from a two-dimensional curve. Using the method of moments on a linear segmentation of the curve, the EFIE takes the} form \cite{Butterfly1}
\[
Zx=b,
\]
where $Z$ is an impedance matrix with (up to scaling)
\begin{align*}
 Z_{ij} &=
  \begin{cases}
   {w_i w_j} H_0^{(2)}({\kappa}|\rho_i-\rho_j|),        & \text{if } i\neq j, \\
   {w_i^2}\left[ 1-\mathrm{i} \frac{2}{\pi}\ln \left( \frac{\gamma k{w_i}}{4e} \right)  \right],       & \text{otherwise},
  \end{cases}
\end{align*}
where $e\approx 2.718$, $\gamma\approx 1.781$, ${\kappa}=2\pi/\lambda_0$ is the wavenumber, $\lambda_0$ represents the free-space wavelength, $H^{(2)}_0$ denotes the zeroth-order Hankel function of the second kind, {$w_i$} is the length of the $i$-th linear segment of the {scatterer object}, $\rho_i$ is the center of the $i$-th segment.

It was shown in \cite{Butterfly1,HSSBF} that {$Z$} admits a HSS-type complementary low-rank property, i.e., off-diagonal blocks are complementary low-rank matrices. The method in \cite{HSSBF} requires $O(N^{1.5}\log N)$ operations to compress the impedance matrix via a slower version of butterfly factorization. After compression, it requires $O(N\log^2(N))$ operations to apply the impedance matrix and makes it possible to design efficient iterative solvers to solve the linear system for the impedance matrix. Replacing the butterfly factorization in \cite{HSSBF} with IDBF, we reduce the factorization time to $O(N\log^2(N))$ as well.

Figure \ref{fig:Z} shows the results of the fast matvec of the impedance matrix from {a} 2D EFIE generated with a spiral object as shown in Figure \ref{fig:p} (a). We vary the number {of segments} $N$ and let ${\kappa}=O(N)$ in the construction of $Z$. In the IDBF, we use the same truncation rank ${k}=40$ and tolerance $\epsilon=10^{-4}$ in IDs with Mock-Chebyshev points. Numerical results verifies the $O(N\log^2(N))$ scaling for both the factorization and application of the new HSS-type butterfly factorization by IDBF.

\begin{figure}[ht!]
  \begin{center}
    \begin{tabular}{ccc}
      \includegraphics[height=1.7in]{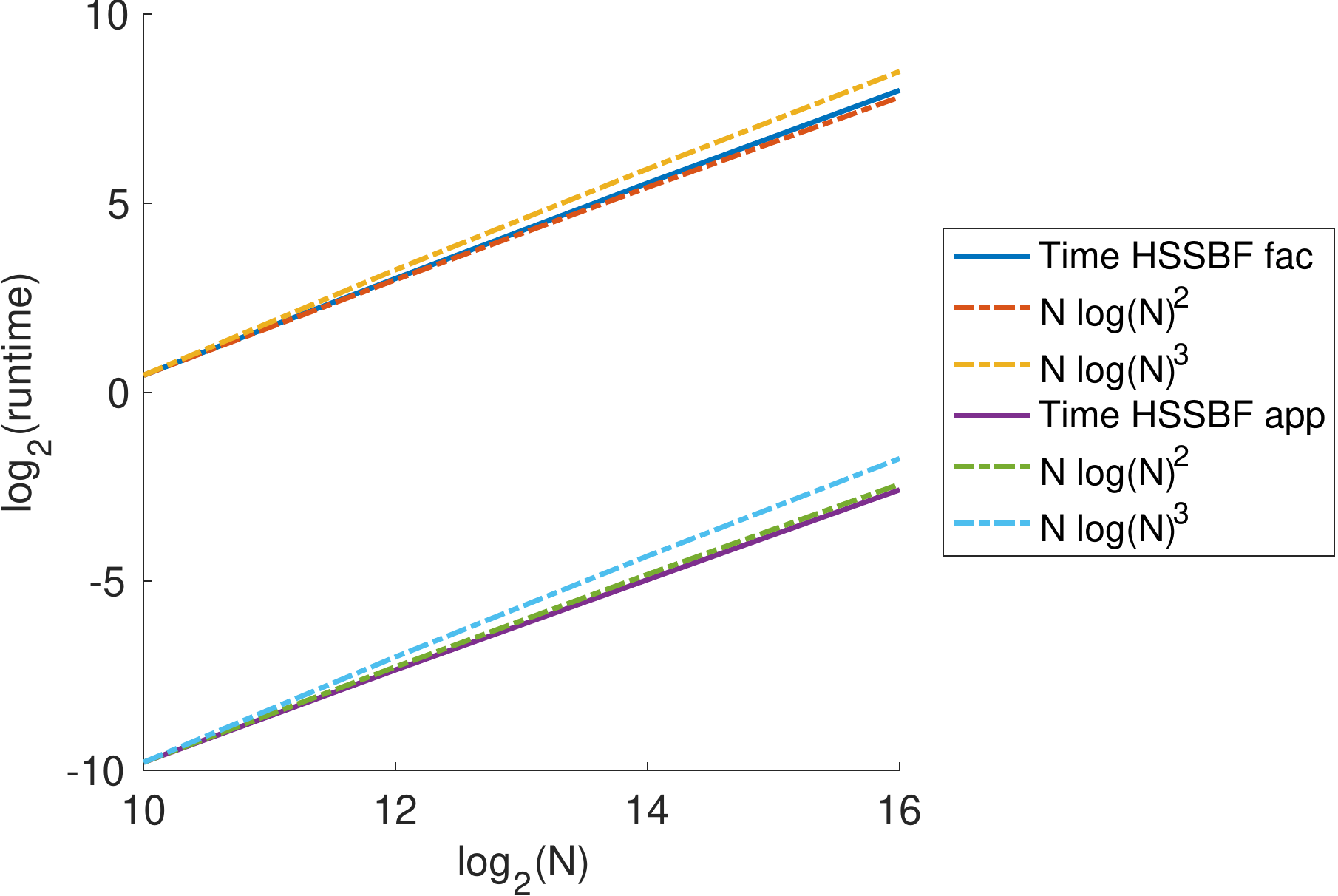}&
      \includegraphics[height=1.7in]{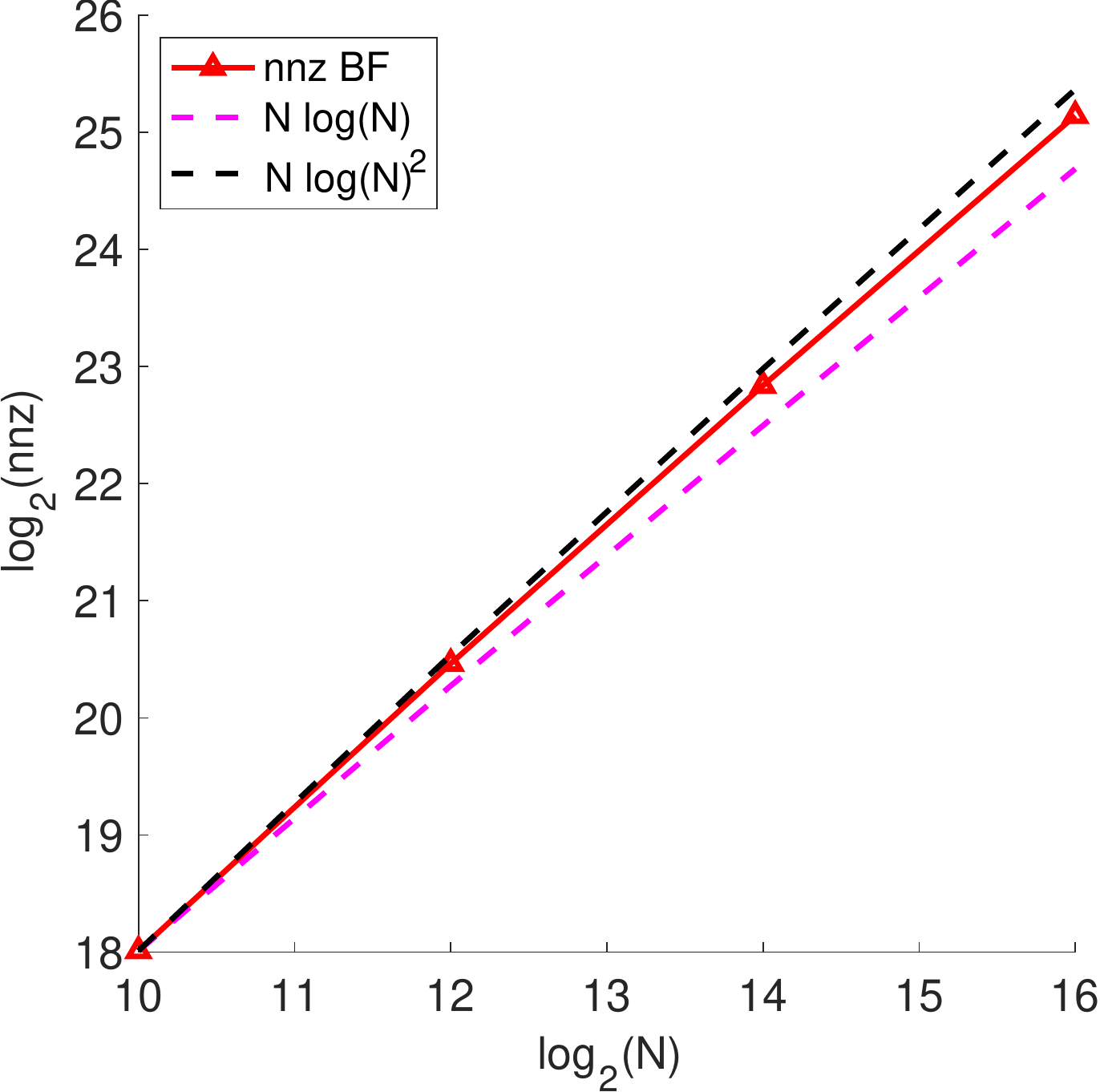}&
      \includegraphics[height=1.7in]{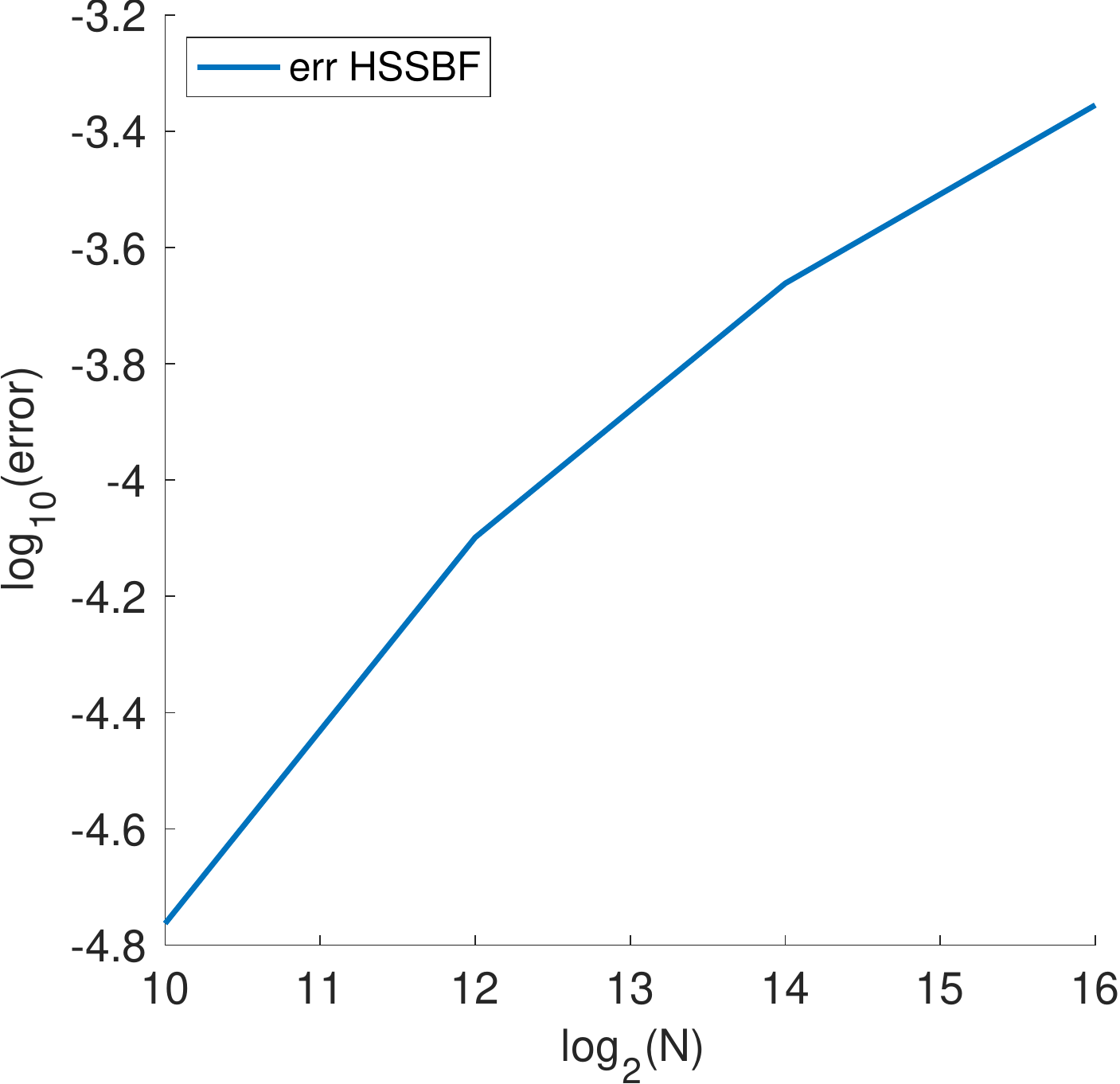}
    \end{tabular}
  \end{center}
\caption{Numerical results for the 2D electric field integral equation.
  $N$ is the size of the matrix; $nnz$ is the number of non-zero entries in the butterfly factorization, $err$ is the approximation error of the matvec by hierarchically applying IDBF.}
\label{fig:Z}
\end{figure}

\paragraph{Example 5.}
The fifth example is from {the} combined field integral {equation (CFIE)}. Similar to the ideas in \cite{Butterfly1,HSSBF} for EFIE, we verify that the impedance matrix of the CFIE\footnote{Codes for generating the impedance matrix are from a MATLAB package ``emsolver" available at \url{https://github.com/dsmi/emsolver}.} by the method of moments for analyzing scattering from {2D} objects also admits a HSS-type complementary low-rank property. Applying the same HSS-type butterfly factorization by IDBF, we obtain $O(N\log^2(N))$ scaling for both the factorization and application time for impedance matrices of CFIEs.  This makes it possible to design efficient iterative solvers to solve the linear system for the impedance matrix. Figure \ref{fig:C} shows the results of the fast matvec of the impedance matrix from {a} 2D CFIE generated with a round object as shown in Figure \ref{fig:p} (b). We vary grid sizes $N$ with the same truncation rank ${k}=40$ and tolerance $\epsilon=10^{-4}$ in IDs with Mock-Chebyshev points. Numerical results {verify} the $O(N\log^2(N))$ scaling for both the factorization and application of the new HSS-type butterfly factorization by IDBF.

\begin{figure}[ht!]
  \begin{center}
    \begin{tabular}{ccc}
      \includegraphics[height=1.7in]{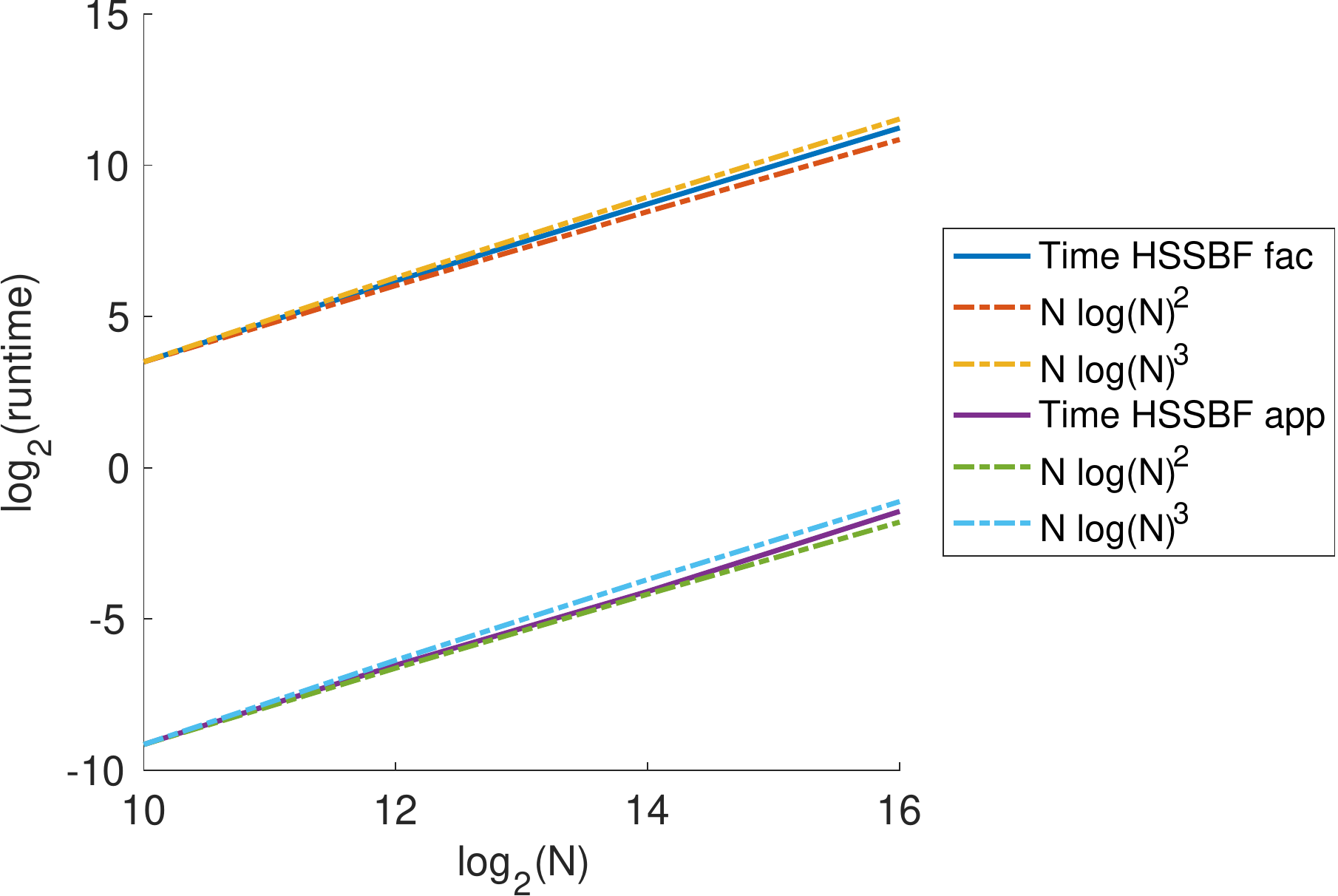}&
      \includegraphics[height=1.7in]{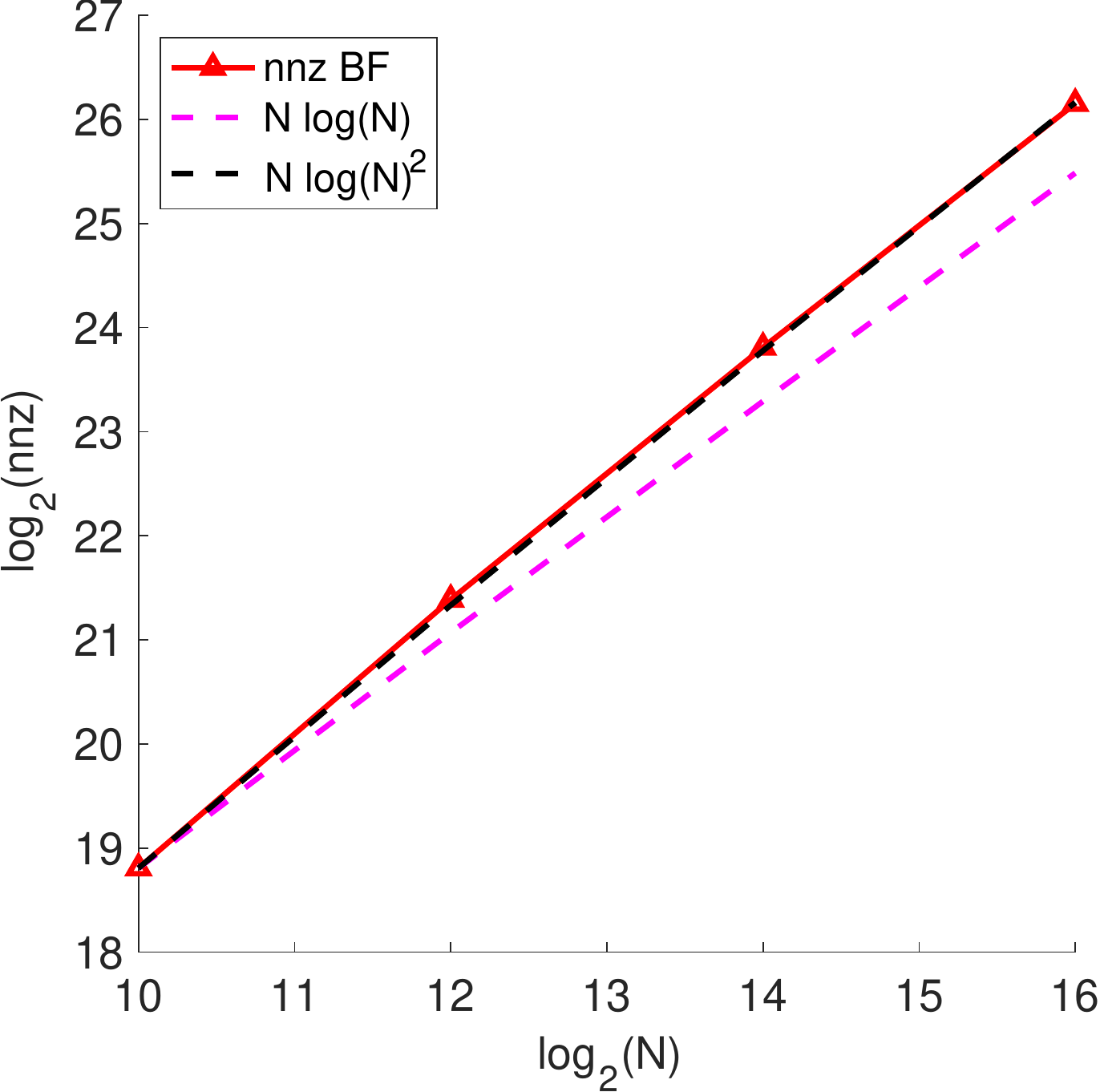}&
      \includegraphics[height=1.7in]{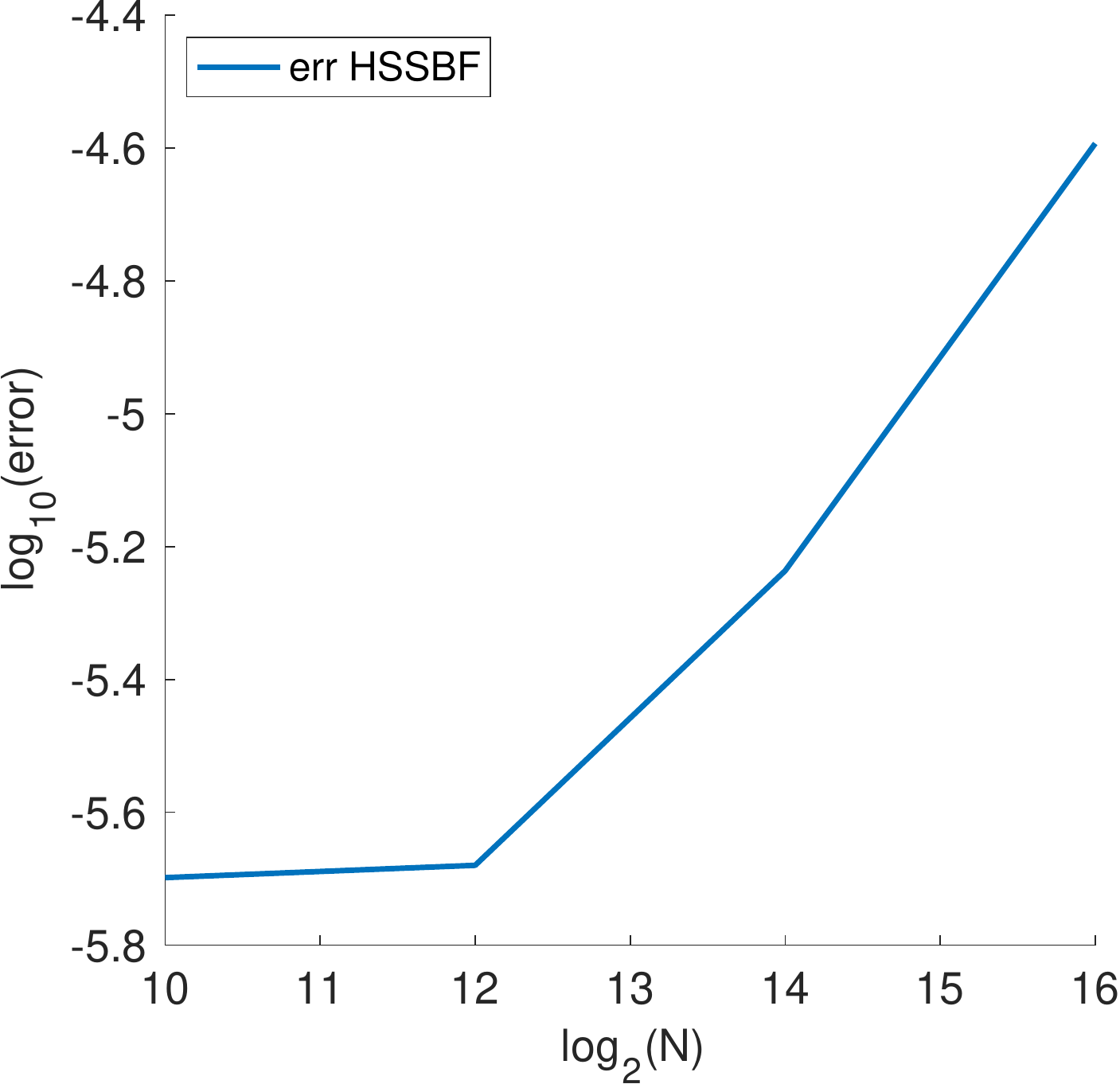}
    \end{tabular}
  \end{center}
\caption{Numerical results for the 2D combined field integral equation.
  $N$ is the size of the matrix; $nnz$ is the number of non-zero entries in the butterfly factorization, $err$ is the approximation error of the matvec by hierarchically applying IDBF.}
\label{fig:C}
\end{figure}

\section{Conclusion and discussion}
\label{sec:conclusion}
This paper introduces an interpolative decomposition butterfly factorization as a data-sparse
approximation of complementary low-rank matrices. It
represents such an $N\times N$ dense matrix as a product of $\O(\log
N)$ sparse matrices. The factorization and application time, and the memory of IDBF all scale as $\O(N\log N)$. The order of {factorization} is from the leaf-root and root-leaf levels of matrix partitioning (e.g., the left and right panels in Figure \ref{fig:complement}) and moves towards the middle level of matrix partitioning (e.g., the middle panel of Figure \ref{fig:complement}). Other orders of factorization are also possible, e.g., an order from the root of the column space to its leaves, an order from the root of the row space to its leaves, or an order from the middle level towards two sides. We leave the extensions of these $O(N\log N)$ IDBFs to the reader.

As shown by numerical examples, IDBF is able to reduce the construction time of the data-sparse representation of the HSS-type complementary matrix in \cite{HSSBF} from $N^{1.5}$ to nearly linear scaling. These matrices arise widely in 2D high-frequency integral equation methods. 

IDBF can also accelerate the factorization time of the hierarchical 
complementary matrix in \cite{LUBF} to nearly linear scaling for 3D 
high-frequency boundary integral methods. After factorization, the application time of matrices in these two integral methods is nearly linear scaling. 
We leave the trivial extension to 3D high-frequency integral methods 
to the reader. 

{\bf Acknowledgments.} H. Yang thanks the support of the start-up grant by the Department of Mathematics at the National University of Singapore.

\bibliographystyle{unsrt} 
\bibliography{ref}

\end{document}